\documentclass{amsart}
\usepackage{amsfonts}
\usepackage{amsmath}
\usepackage{amssymb}
\usepackage{mathrsfs}
\usepackage{url}
\usepackage{caption}
\usepackage{graphicx}
\usepackage{subfigure} 
\usepackage[breaklinks]{hyperref}
\usepackage{enumitem}
\usepackage{tikz-cd}
\usepackage{tikz}
\usepackage{dsfont}
\usepackage{geometry}
\usetikzlibrary{bending,decorations.markings,arrows.meta,calc}

\newtheorem{theorem}{Theorem}[section]
\newtheorem{lemma}[theorem]{Lemma}
\newtheorem{proposition}[theorem]{Proposition}

\theoremstyle{definition}

\newtheorem{corollary}[theorem]{Corollary}

\theoremstyle{remark}
\newtheorem{remark}[theorem]{Remark}

\numberwithin{equation}{section}

\begin{document}

\title{the Birman Krein Formula and scattering phase on product space}

%    Remove any unused author tags.

%    author one information
\author{Hong Zhang}
\address{}
\curraddr{}
\email{}
\thanks{}

%    author two information
%\author{}
\address{}
\curraddr{}
\email{}
\thanks{}


\date{\today}

\dedicatory{}

\begin{abstract}
    In this paper, we study the Birman-Krein formula 
    for the potential scattering on the product space $\mathbb{R}^n\times M$, where $M$ is a compact Riemannian manifold possibly 
    with boundary, and $\mathbb{R}^n$ is the Euclidean space with $n\geq 3$ being an odd number. We also derive 
    an upper bound for the scattering trace when $M$ is a bounded Euclidean domain.
\end{abstract}

\maketitle

\tableofcontents

\setcounter{section}{-1}
\section{Introduction}
For a Schr{\"o}dinger operator $P_{V}^{\mathbb{R}^n}:=-\Delta+V(x)$ with
%  real-valued, bounded, compactly supported potential 
 $V\in L^\infty_{\text{comp}}(\mathbb{R}^n,\mathbb{R})$
on $\mathbb{R}^n$ where $n$ is an odd number, 
the Birman-Krein trace formula (see for example \cite[Theorem 3.51]{MathematicalTheoryofScatteringResonances}) below describes 
the difference of spectrl measures between $P_{V}^{\mathbb{R}^n}$ and the free operator $P_{0}^{\mathbb{R}^n}$
\begin{equation}\label{eq:trace formula on Euclidean space}
        \begin{aligned}
            \operatorname{tr}(f(P_{V}^{\mathbb{R}^n})-f(P_{0}^{\mathbb{R}^n}))=&\frac{1}{2\pi i} 
            \int_0^\infty f(\lambda^2)
            \operatorname{tr}(S(\lambda)^{-1}\partial_{\lambda}S(\lambda))
             d\lambda \\
            &+ \sum_{E_k\in \operatorname{Spec_{pp}}(P_{V}^{\mathbb{R}^n})} f(E_k)+ 
            \frac{1}{2}f(0)c_{n,V}
    \end{aligned}
\end{equation}
Here $f\in \mathscr{S}(\mathbb{R})$ is any Schwartz function, $S(\lambda)$ is a unitary operator on $L^2(\mathbb{S}^{n-1})$ called \textit{scattering matrix}, 
$\operatorname{Spec_{pp}}(P_{V}^{\mathbb{R}^n})$ is the set of eigenvalues of $P_V$ in $L^2$ spaces counted with multiplicity, and $c_{n,V}$ is a constant determined by 
\[
    c_{n,V}:=\left\{ 
        \begin{aligned}
            &m_V(0)-1 \qquad n=1 \\
            &m_V(0)-\dim\left(\ker(P_{V}^{\mathbb{R}^n})\cap L^2\right) \qquad n\geq 3\\
            % &m_R(0)-2\dim\left(\ker(P_{V}^{\mathbb{R}^n})\cap L^2\right)=0 \qquad n\geq 5
        \end{aligned}
    \right.
\]
where $m_V(0)$ is the multiplicity of poles of the (analytically continued) resolvent  
$R_{V}^{\mathbb{R}^n}(\lambda):=(P_{V}^{\mathbb{R}^n}-\lambda^2)^{-1}$ at zero. When $n\geq 5$, the constant $c_{n,V}$ 
is in fact zero. For more detailed discussion about the operator 
$P_{V}^{\mathbb{R}^n}$, 
see \cite[Chapter 3]{MathematicalTheoryofScatteringResonances}.

In this paper, we generalize the Birman-Krein trace formula to the space $X=\mathbb{R}^n\times M$ with product metric, where $(M,g)$ is a compact 
Riemannian manifold without boundary, or a compact 
Riemannian manifold with boundary, imposed with 
Dirichlet or Neumann boundary value, and $n\geq 3$ is an odd number. For a real-valued, bounded, compactly supported potential $V\in L^\infty_{\text{comp}}(X,\mathbb{R})$, 
we consider the corresponding Schr{\"o}dinger operator $P_{V}$ on $X$
\[
    P_V:=-\Delta_X+V
\]
The main result of this paper is the following version of Birman-Krein trace formula
\begin{theorem}\label{thm:Birman-Krein formula on product space}
    Let $f\in \mathscr{S}(\mathbb{R})$, then the operator $f(P_V)-f(P_0)$ is 
    of trace class, and the following trace formula holds
    \begin{equation}\label{eq:trace formula on product space}
        \begin{aligned}
            \operatorname{tr}(f(P_V)-f(P_0))=&\frac{1}{2\pi i} 
            \int_0^\infty f(\lambda^2)
            \operatorname{tr}(S_{\operatorname{nor}}(\lambda)^{-1}\partial_{\lambda}S_{\operatorname{nor}}(\lambda))
             d\lambda \\
            &+ \sum_{E_k\in \operatorname{Spec_{pp}}(P_V)} f(E_k)+ 
            \sum_{ \lambda\in \{\sigma_k\}_{k\geq 0}} \frac{1}{2}f(
                \lambda^2
            )\tilde{m}_V(\lambda)
    \end{aligned}
    \end{equation}
\end{theorem}
Here $0\leq \sigma_0^2\leq \sigma_1^2\leq \sigma_2^2\cdots$ are all eigenvalues of the Laplace-Beltrami operator $-\Delta_M$ on $(M,g)$ counted 
with multiplicity, $S_{\operatorname{nor}}$ is a unitary operator on the space 
\[
    L^2(\mathbb{S}^{n-1},\mathbb{C}^{\sharp \{k:\sigma_k\leq \lambda\}})
\]
called \textit{normalized scattering matrix} which will be defined in Section \ref{sec:Scattering Matrix}, $\tilde{m}_V(\lambda)$ is a 
real number which will be defined in \eqref{eq:definition of tildemV}, and we will show $\tilde{m}_V(\lambda)$ 
is actually zero when $n\geq 5$. The Birman-Krein formula \eqref{eq:trace formula on product space} 
in the product setting should be regarded as the same as the one in the Euclidean setting \eqref{eq:trace formula on Euclidean space}, 
except for that the zero term in \eqref{eq:trace formula on Euclidean space} is replaced by those terms given by eigenvalues of $-\Delta_M$, 
which are referred as \textit{thresholds}. The reason for this replacement will be clear in our paper.

We essentially follow \cite[Chapter 3]{MathematicalTheoryofScatteringResonances} to prove Theorem \ref{thm:Birman-Krein formula on product space}, 
with only the slightest modification to adpat to our setting. The structure of the paper is as following: 

\begin{itemize}
    \item In chapter 1, we briefly review some results about the resolvents in Euclidean space may be used later. The analoguous 
    result in the product setting will be discussed.
    \item In chapter 2, we will first establish the analytical continuation of the resolvent $R_V(z):=(P_V-z)^{-1}$, starting from $z\in \mathbb{C}-\mathbb{R}_{\geq 0}$, 
    and then for $z$ lying in a Riemann surface $\hat{\mathcal{Z}}$ defined in Section \ref{subsection:The construction of Riemann surface}, 
    in which the square roots $\sqrt{z-\sigma_k^2}$ are well-defined for all $k\in \mathbb{N}_0$. Next we will examine the behaviour of $R_V(z)$ 
    for $z$ near the real line carefully, with the help of Rellich's uniqueness theorem in our setting. 
    \item In Chapter 3, the scattering matrix will be defined, where its regularity will be analyzed. 
    Then it is clear that the relation between the spectral measure of $P_V$ and the scattering matrix is as that in Euclidean space. 
    \item In Chapter 4, we devote the whole chapter to the proof of the main Theorem \ref{thm:Birman-Krein formula on product space}. We will first 
    show that the formula holds for $f\in C_c^\infty(\mathbb{R})$ with support away from $\{\sigma_k\}_{k\geq 0}$, and then tackle with the contribution 
    for $\lambda$ near the thresholds. The method we use is essentially the same as that in 
    \cite[Chapter 3]{MathematicalTheoryofScatteringResonances}.
    \item In Chapter 5, we will establish an upper bound for the scattering phase when $M$ is a bounded Euclidean domain, exploiting  
    Robert's commutator argument(See \cite[Chapter 3]{10.1007/978-1-4612-0775-7_18}). Then we will use 
    the usual heat kernel argument to obtain a lower bound for the total variation of the 
    scattering phase. 
\end{itemize}

\noindent \textbf{Related work.} The Birman Krein formula goes back to the classical paper \cite{birman1962}, and is related to the
    more general study of spectral shift functions in an abstract setting, see 
    \cite[Chapter 8]{yafaev1998mathematical}
    for a detailed exposition. For more recent advances on trace formula in Euclidean scattering theory, see \cite{Bruneau2020Threshold} 
    and \cite{Hanisch2022A}. 
    
    The trace formula in product setting has been proved by T. Christiansen for $n=1$ in 
    \cite{Christiansen1995Scattering}, who used Melrose's b-calculus as tools to establish trace type formula on manifolds with asymptotically cylindrical ends, which is 
    much more general than the case $\mathbb{R}\times M$. Furthermore, 
    T. Christiansen and Zworski \cite{Christiansen1995Spectral} proved that the spectral asymptotics of the embedded eigenvalues 
    and 
    the scattering phase on manifolds with cylindrical ends, 
    exploiting the trace formula established in \cite{Christiansen1995Scattering}. 
    In our setting where $n\geq 3$, results like spectral asymptotics in general cases seem impossible, although any negative example is unknown. 

    Moreover, when $M$ has no boundaries, our setting $\mathbb{R}^n\times M$ should be viewed as the model case of 
    compact manifolds with a \textit{fibred boundary metrics}, also called $\varphi$-metrics, if we take a fibered compactification 
    over $\mathbb{R}^n$. Mazzeo 
    and Melrose \cite{Mazzeo1998Pseudodifferential} studied the pseudo-differential operator calculus
    adpatted to this fibred boundary setting, 
    in this setting the scattering matrix $S(\lambda)$ generally
    can only be defined for those $z\in \mathbb{R}$ smaller than 
    the first eigenvalue $\sigma_1^2$ of $\Delta_M$ \cite{Melrose1998Fibrations}.

    For more study on the scattering or spectral theory on product-type or boundary-fibered space, see 
    for example \cite{Christiansen2017ResolventEO}, \cite{10.1093/imrn/rnab254}, 
    and also \cite{Grieser2020SpectralGO} 
    and \cite{Talebi2021SpectralGO}. The research closest to our setting is the work 
    \cite{Christiansen2020ResonancesFS} by T. Christiansen, who systematically investigated
    the potential scattering on $\mathbb{R}^n\times \mathbb{S}^1$. But her work relied heavily on 
    the properties of the eigenfunctions of $-\Delta_{\mathbb{S}^1}$.

\noindent \textbf{Further possible result}. 
One may naturally ask whether all known results for potential scattering $P_V^{\mathbb{R}^n}$ on $\mathbb{R}^n$
also hold in the product setting:

\begin{itemize}
    \item Upper or lower bounds on numbers of poles (or resonances) of $R_V(z)$ near the real line or in some sheets $\mathbb{C}$ as a 
    subset of $\hat{\mathcal{Z}}$. This kind of result and actually even a stronger asymptotic result 
    has been obtained when $n=1$ by T. Christiansen\cite{Christiansen2003AsymptoticsFA}. 
    The upper bound result is unknown because the usual zero-counting for holomorphic functions on $\mathbb{C}$ does not hold in the complicated Riemann surface $\hat{\mathcal{Z}}$.
    The author believes that the usual zero-counting method can obtain the upper bound for those resonances in a single sheet, far away from the real line.
    For the lower bound, the author believes that for non-zero potential $V \in C_c^\infty(\mathbb{R}^n \times M)$, there are infinitely many 
    poles in $\hat{\mathcal{Z}}$. However, the author does not even know the existence of any poles of $R_V(z)$ for such $V$ 
    except in the case that $M=\mathbb{S}^1$(\cite{Christiansen2020ResonancesFS}).
    \item Spectral asymptotics of eigenvalues and the scattering matrix. 
    The derivation of the asymptotic behavior of the scattering matrix on $\mathbb{R}^n$ uses the Schrödinger propagator to 
    approximate the resolvent, but this method seems no longer effective since there may 
    be poles of the resolvent on the real line in our setting. In view of classical quantum correspondence, 
    the presence of the manifold $M$ causes a trap of the Schrödinger propagator, namely, the existence of
    geodesics tangent to $M$.
    For the spectral asymptotics on $\mathbb{R} \times M$ obtained by T. Christiansen and Zworski
    \cite{Christiansen1995Spectral}, 
    their work relies on the fact that the scattering matrix is really a finite-dimensional matrix when $n=1$, 
    instead of being an operator, \textit{i.e.}, an infinite-dimensional matrix. 
    Therefore, the phase of the scattering matrix can be controlled when $n=1$. 
    In fact, except the case that $M$ is a bounded Euclidean domain which is presented in 
    this paper, the author does not know any upper bound or lower bound results for eigenvalues counting or 
    the scattering matrix in the setting $\mathbb{R}^n \times M$ for generic manifold $M$.
    \item Some special cases. For example, we can take $M = \mathbb{T}^m$ or $M = \mathbb{S}^m$,
    where the eigenfunctions and eigenvalues of $\Delta_M$ can be expressed explicitly, and we take a special potential $V$. 
    In these cases some partial results may be obtained.
\end{itemize}
It is also natural to generalize the potential scattering to the \textit{black-box} scattering 
setting(see, for example \cite[Chapter 4]{MathematicalTheoryofScatteringResonances}), in this setting the behaviour near thresholds will be more complicated.
Once the scattering trace formula is established for the black-box scattering, 
the commutator argument in Chapter 5 of this paper can lead to an asymptotic of 
the scattering phase when $M$ is a bounded Euclidean domain, stronger than the upper bound result of 
the scattering phase in potential scattering, if the black-box is a \textit{second-order} perturbation in some sense, 
for example the metric is perturbed or we consider the obstacle scattering. This kind of result is well-known in Euclidean 
scattering theory, see, for example \cite{1998Spectral}.

\section{Results in Euclidean space}

In this chapter, we list some of the results concerning the free resolvent $R_0^{\mathbb{R}^n}(\lambda)$ in $\mathbb{R}^n$ with odd number $n\geq 3$ 
which will be used later. The following proposition is \cite[Theorem 3.1]{MathematicalTheoryofScatteringResonances}.
\begin{proposition}
    Let $n\geq 3$ be odd. Then the resolvent defined by 
    \[
        R_0^{\mathbb{R}^n}(\lambda)=(-\Delta_{\mathbb{R}^n}-\lambda^2)^{-1}:L^2(\mathbb{R}^n)\to L^2(\mathbb{R}^n)
    \]
    for $\operatorname{Im} \lambda>0$, continuuous analytically to an entire family of operators 
    \[
        R_0^{\mathbb{R}^n}(\lambda):L^2_{\operatorname{comp}}(\mathbb{R}^n)\to H^2_{\operatorname{loc}}(\mathbb{R}^n)
    \]
    For any $\rho\in C_c^\infty(\mathbb{R}^n)$ and any $L>\sup \{|x-y|:x,y\in \operatorname{supp} \rho\}$ we have 
    \[
        \rho R_0^{\mathbb{R}^n}(\lambda)\rho=\mathcal{O}((1+|\lambda|^{j-1})e^{L\max(-\operatorname{Im}(\lambda),0)})
    \]
\end{proposition}

The free resolvent has an explicit expression, see \cite[Theorem 3.3]{MathematicalTheoryofScatteringResonances}.
\begin{proposition}\label{thm:3.3}
    Suppose $n\geq 3$ is odd. Then the Schwartz kernel of the free resolvent $R_0^{\mathbb{R}^n}(\lambda)$ is given by 
    \[
        R_0^{\mathbb{R}^n}(\lambda,x,y)=\frac{e^{i\lambda|x-y|}}{|x-y|^{n-2}}P_n(\lambda|x-y|)
    \]
    where $P_n$ is a polynomial of degree $(n-3)/2$. When $n=3$ we have 
    \[
        R_0^{\mathbb{R}^3}(\lambda,x,y)=\frac{e^{i\lambda|x-y|}}{4\pi|x-y|}
    \]
    and when $n=5$ we have 
    \[
        R_0^{\mathbb{R}^5}(\lambda,x,y)=\frac{e^{i\lambda|x-y|}}{8\pi^2|x-y|^{3}}\left(\frac{\lambda|x-y|}{i}+1\right)
    \]
\end{proposition}

The next proposition describes the asymptotic of $R_0^{\mathbb{R}^n}(\lambda)(f)$ at infinity, see \cite[Theorem 3.5]{MathematicalTheoryofScatteringResonances}.
\begin{proposition}\label{thm:3.5}
    Suppose that $n\geq 3$ is odd, and that $f\in \mathcal{E}'(\mathbb{R}^n)$ is a compactly supported distribution.
    Then for $\lambda\in \mathbb{R}-\{0\}$ we have for some smooth function $h(r,\theta)$ 
    defined for sufficiently large $r$ and $\theta\in \mathbb{S}^{n-1}$
    \[
        R_0(\lambda)f(x)=e^{i\lambda|x|}|x|^{-\frac{n-1}{2}}h(|x|,\frac{x}{|x|}),\quad x\neq 0
    \]
    where $h$ has radial asymptotic expansion as $|x|\to \infty$
    \[
        h(r,\theta)\sim \sum_{j=0}^\infty r^{-j} h_j(\theta), \quad h_0(\theta)=\frac{1}{4\pi}\left(\frac{\lambda}{2\pi i}\right)^{(n-3)/2}\hat{f}(\lambda\theta)
    \]
    More precisely, the asymptotic expansion is intepreted in the following way: there exists some $\rho>0$ 
    depending on the support of $f$ so that the remainder term $R_J$ defined by 
    \[
        R_J(r\theta):=h(r,\theta)-\sum_{j=0}^{J-1} r^{-j} h_j(\theta),\quad (r,\theta)\in (\rho,+\infty)\times \mathbb{S}^{n-1}
    \]
    satisfes $R_J\in C^\infty(\mathbb{R}^n-B_{\mathbb{R}^n}(0,\rho))$
    and 
    \[
        |\partial_x^\alpha R_J(x)|\leq C_{\alpha,J}|x|^{-J},\quad |x|>\rho
    \]
    where the constant $C_{\alpha,J}$ only depends on the semi-norms of $f$ as an element in the dual space of $C^\infty(\mathbb{R}^n)$.
\end{proposition}

We have the following decomposition of the plane wave $e^{-i\lambda\langle x,\omega\rangle}$ 
as $|x|\to +\infty$, 
see \cite[Theorem 3.38]{MathematicalTheoryofScatteringResonances} and the remark 
after that.
\begin{proposition}\label{prop:3.38}
    For $\lambda\in \mathbb{R}-\{0\}$, we have, in 
    the sense of distribution in $\theta\in \mathbb{S}^{n-1}$
    \[
        e^{-i\lambda r\langle\theta,\omega \rangle}
        \sim \frac{1}{(\lambda r)^{\frac{n-1}{2}}}
        \left(c_n^+ e^{-i\lambda r} \delta_\omega(\theta)+c_n^- 
        e^{+i\lambda r} \delta_{-\omega}(\theta)\right)
    \]
    as $r\to +\infty$, where 
    \[
        c_n^\pm=(2\pi)^{\frac{n-1}{2}}e^{\pm \frac{\pi}{4}(n-1)i}
    \]
    Moreover, we know as $r\to +\infty$
    \[
        e^{-i\lambda r\langle\theta,\omega \rangle}=e^{-i\lambda r}a^+(\lambda r,\omega,\theta)
        +e^{i\lambda r}a^-(\lambda r,\omega,\theta)
    \]
    where $a^{\pm}(r,\omega,\theta)$ has an full expansion as $r\to +\infty$, taking values in 
    $C^\infty(\mathbb{S}^{n-1}_{\omega},\mathscr{D}'(\mathbb{S}^{n-1}_{\theta}))$.
\end{proposition}

\section{Basic facts of the resolvent}

We briefly recall our setting. Let $(M,g)$ be a compact smooth manifold equipped with a 
Riemannian metric $g$, and $X:=(\mathbb{R}^n\times M,\delta_{ij}\oplus g)$ be 
the product manifold with the product metric. Suppose $0=\sigma_0^2<\sigma_1^2\leq \sigma_2^2\cdots$ are all eigenvalues of 
the Laplace-Beltrami operator $-\Delta_g$ on $(M,g)$ counted with multiplicity, subject to certain boundary conditions if $M$ has non-empty boundaries. 
Let $\{\varphi_k\}_{k\geq 0}\subset C^\infty(M,\mathbb{R})$ forms a complete orthonormal basis of $L^2(M,d\operatorname{vol}_g)$, 
and $\varphi_k$ corresponds to eigenvalue $\sigma_k^2$. 
We refer to the numbers in the set $\{\pm \sigma_k\}_{k\geq 0}\subset \mathbb{R}$ as thresholds.
We consider a bounded, compactly supported, real-valued potential $V\in L^\infty_{\text{comp}}(X;\mathbb{R})$, and define 
\[
    P_V:=-\Delta_X+V
\]

The \textit{free resolvent} $R_0(z)$ is first defined for $z\in \mathbb{C}-\mathbb{R}_{\geq 0}$. For $u\in L^2(X)$, 
it is given by
\begin{equation}\label{eq:Expression of R0 lambda}
    R_0(z)(u):=(P_0-z)^{-1}(u)=\sum_{k\geq 0} R_0^{\mathbb{R}^n}(\sqrt{z-\sigma_k^2})\left( 
        \langle u,\varphi_k\rangle_{L^2(M)}
    \right)\otimes \varphi_k,\quad z\in \mathbb{C}-\mathbb{R}_{\geq 0}
\end{equation}
where we choose the branch of $\sqrt{z-\sigma_k^2}$ with argument $(0,\pi)$. Note that $\langle u,\varphi_k\rangle_{L^2(M)}$ is an $L^2$ function on $\mathbb{R}^n$. 
Thus 
\[
    R_0(z):L^2(X)\to H^2(X)
\]
is a family of operators depending holomorphically for $z\in \mathbb{C}-\mathbb{R}_{\geq 0}$. We will next construct a Riemann surface $\hat{\mathcal{Z}}$, 
with a natural projection $\hat{\mathcal{Z}}\to \mathbb{C}$, and a sequence of analytic function $\tau_k:\hat{\mathcal{Z}}\to \mathbb{C}$ with 
\[\begin{tikzcd}
	\hat{\mathcal{Z}} & {\mathbb{C}} \\
	{\mathbb{C}}
	\arrow[from=1-1, to=1-2]
	\arrow["{\tau_k}"', from=1-1, to=2-1]
	\arrow["{z\mapsto z^2+\sigma_k}"', from=2-1, to=1-2]
\end{tikzcd}\]
where the horizontal map $\hat{\mathcal{Z}} \to {\mathbb{C}}$ is the natural projection, and $\tau_k$ can be viewed as the analytic continuation of $z\mapsto \sqrt{z-\sigma_k^2}$.
Then using the holomorphy of the free resolvent $R_0^{\mathbb{R}^n}$ in $\mathbb{R}^n$, we obtain 
\[
    R_0(z):L^2_{\operatorname{comp}}(X)\to H^2_{\operatorname{loc}}(X)
\]
is a family of operators depending holomorphically for $z\in \mathcal{Z}$.

\subsection{The construction of $\hat{\mathcal{Z}}$}\label{subsection:The construction of Riemann surface}
The idea of the construction of the Riemann surface $\hat{\mathcal{Z}}$ analytically continuing 
$\sqrt{z-\sigma_k^2}$ for all $\sigma_k$, comes from \cite[Section 6.7]{melrose1993atiyah}. 
For the reader’s convenience, we provide a detailed—albeit somewhat tedious—description of this construction. The structure of 
this surface will only be used in the proof of the symmetry of the scattering matrix
in Section \ref{subsection:Regulairity and symmetry of scattering matrix}. 

Without loss of generality, we assume that $\sigma_k<\sigma_{k+1}$ for each $k\in \mathbb{N}_0$.
We will construct a sequence $\{\mathcal{Z}_k\}_{k\geq {0}}$ of Riemann surfaces inductively, 
such that on each $\mathcal{Z}_k$, the square roots $\sqrt{z-\sigma_j^2}$ are well-defined and analytic
for $j=0,1\cdots k$. 

To construct $\mathcal{Z}_0$, we begin by cutting $\mathbb{C}$ along 
the non-negative real axis $\mathbb{R}_{\geq 0}\subset \mathbb{C}$.
This creates two copies of the cut half-line, one is adjacent to the first quadrant and is labeld by $(0,0)$, 
while the other is adjacent to the fourth quadrant and is labeld by $(0,1)$. Then we glue together two such cut copies of $\mathbb{C}$ 
, via identifying $(0,0)$-line in the first copy with $(0,1)$-line 
in the second copy, and the $(0,1)$-line in the first with the $(0,0)$-line in the second.

Inductively, $\mathcal{Z}_k$ consists of $2^{k+1}$ copies of cut $\mathbb{C}$, each assigned a ranking $r=1,2\cdots 2^{k+1}$.
In each copy, the non-negative real axis is divided into $k$ intervals (referred to as “parts”):
for $j = 1, \dots, k - 1$, the $j$-th part corresponds to $[\sigma_{j-1}^2, \sigma_j^2]$,
and the $k$-th part corresponds to $[\sigma_k^2, +\infty)$. We will label every part a 
key, which includes a binary code of length $k+1$ and a number $j$, if this part is the $j$-th part in the corresponding half real axis. 
As a topological space, two parts are identified if and only if they share the same key, that is, they have the same binary code 
and they are both the $j$-th parts of their respective half-lines. Therefore, for each $j=0,1\cdots k$, each binary string of length $k+1$ corrsponds 
to exactly two $j$-th parts, which are glued together. There is a natural projection $\mathcal{Z}_k\to \mathbb{C}$ 
restricted to the on each copy of cut $\mathbb{C}$, 
and the square roots functions $\sqrt{\bullet-\sigma_j^2}$ for $j=0,1\cdots k$, are all well defined continuous functions so that the following diagram commutes
\[\begin{tikzcd}
	{\mathcal{Z}_k} & {\mathbb{C}} \\
	{\mathbb{C}}
	\arrow[from=1-1, to=1-2]
	\arrow["{\sqrt{\bullet-\sigma_k^2}}"', from=1-1, to=2-1]
	\arrow["{z\mapsto z^2+\sigma_k}"', from=2-1, to=1-2]
\end{tikzcd}\]

To construct $\mathcal{Z}_{k+1}$ from $\mathcal{Z}_k$, we begin by making 
two copies of $\mathcal{Z}_k$, denoted by $\mathcal{Z}_{k,0}$ 
and $\mathcal{Z}_{k,1}$. Then we divide the $k$-th part of each real half line into two parts, 
the new $k$-th part (corresponding to $[\sigma_k^2, \sigma_{k+1}^2]$),
and the $(k+1)$-th part (corresponding to $[\sigma_{k+1}^2, +\infty)$). 
For each part  
which is not the $k+1$-th part of a half-line, 
its new binary code is obtained by appending a bit $s \in \{0,1\}$ to the end of the existing code,
depending on whether the part comes from $\mathcal{Z}_{k,s}$. For the $k+1$-th part:
\begin{itemize}
    \item If it lies on the half-real line adjacent to the first quadrant in a cut $\mathbb{C}$ of rank $r$ in 
$\mathcal{Z}_{k,s}$ where $s\in \{0,1\}$, the binary code should be 
\[
(\text{binary representation of } r-1 \quad s)
\]  
    \item If it lies on the half-line adjacent to the fourth quadrant n a cut $\mathbb{C}$ of rank $r$ in 
$\mathcal{Z}_{k,s}$ where $s\in \{0,1\}$, the binary code should be 
\[(\text{binary representation of } r-1\quad (1+s)\operatorname{mod} 2)\]
\end{itemize}
The new ranking of a cut copy $\mathbb{C}$ in $\mathcal{Z}_{k,0}$ remains the same as before, while in 
$\mathcal{Z}_{k,1}$ it is increased by $2^{k+1}$. The following figures illustrate how $\mathcal{Z}_0,\mathcal{Z}_1$ and $\mathcal{Z}_2$ 
are constructed.

% 插入图片，指定宽度为文本宽度
\begin{figure}[h] % [h] 参数表示“here”，即在当前位置插入图片
    \centering % 居中图片
    \subfigure[picture of $\mathcal{Z}_0$, the ranking of the upper cut $\mathbb{C}$ is one, while the ranking of the lower cut $\mathbb{C}$ is two]
    {
        \begin{tikzpicture}
            \tikzset{->-/.style=
            {decoration={markings,mark=at position #1 with 
            {\arrow{latex}}},postaction={decorate}}}
            \tikzset{-<-/.style=
            {decoration={markings,mark=at position #1 with 
            {\arrow{latex reversed}}},postaction={decorate}}}
            
            \draw (-1,0) -- (0,0);
            \node[left] at (1,0.25) {\tiny{[(0),0]}};
            \node[left] at (1,-0.25) {\tiny{[(1),0]}};
            \draw[-Latex] (0,-0.75) -- (0,0.75);
            \draw[-Latex][red] (0,0) -- (1,0.25) ;
            \draw[-Latex][orange] (0,0) -- (1,-0.25);
            
            \draw[shift={(0,-2)}] (-1,0) -- (0,0);
            \node[shift={(0,-2)}][left] at (1,0.25) {\tiny{[(1),0]}};
            \node[shift={(0,-2)}][left] at (1,-0.25) {\tiny{[(0),0]}};
            \draw[shift={(0,-2)}][-Latex] (0,-0.75) -- (0,0.75);
            \draw[shift={(0,-2)}][-Latex][orange] (0,0) -- (1,0.25) ;
            \draw[shift={(0,-2)}][-Latex][red] (0,0) -- (1,-0.25);
        \end{tikzpicture}
    }
    \quad
    \subfigure[picture of $\mathcal{Z}_1$, the ranking of the upper left cut $\mathbb{C}$ is one, the ranking of the lower left cut $\mathbb{C}$ is two, 
    the ranking of the upper right cut $\mathbb{C}$ is three, the ranking of the lower right cut $\mathbb{C}$ is four.]
    {
        \begin{tikzpicture}
            \tikzset{->-/.style=
            {decoration={markings,mark=at position #1 with 
            {\arrow{latex}}},postaction={decorate}}}
            \tikzset{-<-/.style=
            {decoration={markings,mark=at position #1 with 
            {\arrow{latex reversed}}},postaction={decorate}}}
            
            \draw (-1,0) -- (0,0);
            \node[left] at (1,0.25) {\tiny{[(00),0]}};
            \node[left] at (1,-0.25) {\tiny{[(10),0]}};
            \node[left] at (2,0.5) {\tiny{[(00),1]}};
            \node[left] at (2,-0.5) {\tiny{[(01),1]}};
            \filldraw[red] (1,0.25) circle (1pt);
            \filldraw[orange] (1,-0.25) circle (1pt);
            \draw[-Latex] (0,-0.75) -- (0,0.75);
            \draw[red] (0,0) -- (1,0.25) ;
            \draw[-Latex][green] (1,0.25) -- (2,0.5) ;
            \draw[orange] (0,0) -- (1,-0.25);
            \draw[-Latex][blue] (1,-0.25) -- (2,-0.5);
        
            \draw[shift={(0,-2)}] (-1,0) -- (0,0);
            \node[shift={(0,-2)}][left] at (1,0.25) {\tiny{[(10),0]}};
            \node[shift={(0,-2)}][left] at (1,-0.25) {\tiny{[(00),0]}};
            \node[shift={(0,-2)}][left] at (2,0.5) {\tiny{[(10),1]}};
            \node[shift={(0,-2)}][left] at (2,-0.5) {\tiny{[(11),1]}};
            \filldraw[shift={(0,-2)}][orange] (1,0.25) circle (1pt);
            \filldraw[shift={(0,-2)}][red] (1,-0.25) circle (1pt);
            \draw[shift={(0,-2)}][-Latex] (0,-0.75) -- (0,0.75);
            \draw[shift={(0,-2)}][orange] (0,0) -- (1,0.25) ;
            \draw[shift={(0,-2)}][-Latex][yellow] (1,0.25) -- (2,0.5) ;
            \draw[shift={(0,-2)}][red] (0,0) -- (1,-0.25);
            \draw[shift={(0,-2)}][-Latex][brown] (1,-0.25) -- (2,-0.5);
        
            \draw[shift={(3,0)}] (-1,0) -- (0,0);
            \node[shift={(3,0)}][left] at (1,0.25) {\tiny{[(01),0]}};
            \node[shift={(3,0)}][left] at (1,-0.25) {\tiny{[(11),0]}};
            \node[shift={(3,0)}][left] at (2,0.5) {\tiny{[(01),1]}};
            \node[shift={(3,0)}][left] at (2,-0.5) {\tiny{[(00),1]}};
            \filldraw[shift={(3,0)}][violet] (1,0.25) circle (1pt);
            \filldraw[shift={(3,0)}][cyan] (1,-0.25) circle (1pt);
            \draw[shift={(3,0)}][-Latex] (0,-0.75) -- (0,0.75);
            \draw[shift={(3,0)}][violet] (0,0) -- (1,0.25) ;
            \draw[shift={(3,0)}][-Latex][blue] (1,0.25) -- (2,0.5) ;
            \draw[shift={(3,0)}][cyan] (0,0) -- (1,-0.25);
            \draw[shift={(3,0)}][-Latex][green] (1,-0.25) -- (2,-0.5);
        
            \draw[shift={(3,-2)}] (-1,0) -- (0,0);
            \node[shift={(3,-2)}][left] at (1,0.25) {\tiny{[(11),0]}};
            \node[shift={(3,-2)}][left] at (1,-0.25) {\tiny{[(01),0]}};
            \node[shift={(3,-2)}][left] at (2,0.5) {\tiny{[(11),1]}};
            \node[shift={(3,-2)}][left] at (2,-0.5) {\tiny{[(10),1]}};
            \filldraw[shift={(3,-2)}][cyan] (1,0.25) circle (1pt);
            \filldraw[shift={(3,-2)}][violet] (1,-0.25) circle (1pt);
            \draw[shift={(3,-2)}][-Latex] (0,-0.75) -- (0,0.75);
            \draw[shift={(3,-2)}][cyan] (0,0) -- (1,0.25) ;
            \draw[shift={(3,-2)}][-Latex][brown] (1,0.25) -- (2,0.5) ;
            \draw[shift={(3,-2)}][violet] (0,0) -- (1,-0.25);
            \draw[shift={(3,-2)}][-Latex][yellow] (1,-0.25) -- (2,-0.5);
            % \draw (-1,0) -- (0,0);
            % \node[left] at (1,0.25) {\tiny{[(00),0]}};
            % \node[left] at (1,-0.25) {\tiny{[(10),0]}};
            % \node[left] at (2,0.5) {\tiny{[(00),1]}};
            % \node[left] at (2,-0.5) {\tiny{[(01),1]}};
            % \filldraw[red] (1,0.25) circle (1pt);
            % \filldraw[orange] (1,-0.25) circle (1pt);
            % \draw[-Latex] (0,-0.75) -- (0,0.75);
            % \draw[red] (0,0) -- (1,0.25) ;
            % \draw[-Latex][green] (1,0.25) -- (2,0.5) ;
            % \draw[orange] (0,0) -- (1,-0.25);
            % \draw[-Latex][blue] (1,-0.25) -- (2,-0.5);
        
            % \draw (-1,0) -- (0,0);
            % \node[left] at (1,0.25) {\tiny{[(00),0]}};
            % \node[left] at (1,-0.25) {\tiny{[(10),0]}};
            % \node[left] at (2,0.5) {\tiny{[(00),1]}};
            % \node[left] at (2,-0.5) {\tiny{[(01),1]}};
            % \filldraw[red] (1,0.25) circle (1pt);
            % \filldraw[orange] (1,-0.25) circle (1pt);
            % \draw[-Latex] (0,-0.75) -- (0,0.75);
            % \draw[red] (0,0) -- (1,0.25) ;
            % \draw[-Latex][green] (1,0.25) -- (2,0.5) ;
            % \draw[orange] (0,0) -- (1,-0.25);
            % \draw[-Latex][blue] (1,-0.25) -- (2,-0.5);  
        \end{tikzpicture}
    }
    \caption{The construction of $\mathcal{Z}_0$ and $\mathcal{Z}_1$, the parts with the same color(or the same key) are attached.} % 添加图片标题
    \label{fig:example_image} % 为图片添加标签，用于引用
\end{figure}
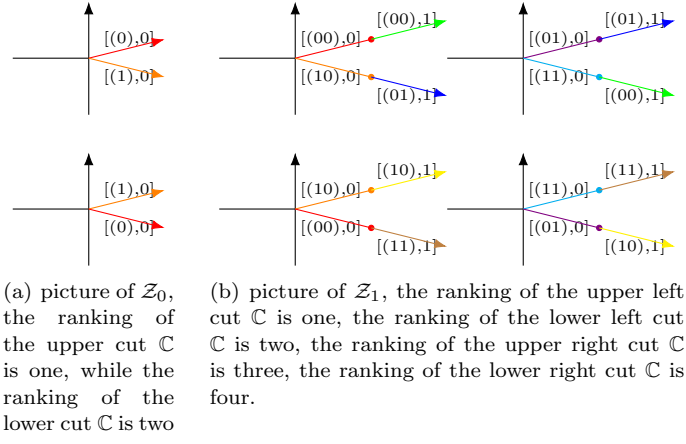

\begin{figure}[h] % [h] 参数表示“here”，即在当前位置插入图片
    \centering % 居中图片
    \begin{tikzpicture}
        \tikzset{->-/.style=
        {decoration={markings,mark=at position #1 with 
        {\arrow{latex}}},postaction={decorate}}}
        \tikzset{-<-/.style=
        {decoration={markings,mark=at position #1 with 
        {\arrow{latex reversed}}},postaction={decorate}}}
        
        \draw (-1,0) -- (0,0);
        \node[left] at (1.1,0.28) {\tiny{[(000),0]}};
        \node[left] at (1.1,-0.28) {\tiny{[(100),0]}};
        \node[left] at (2.1,0.56) {\tiny{[(000),1]}};
        \node[left] at (2.1,-0.56) {\tiny{[(010),1]}};
        \node[left] at (3.1,0.85) {\tiny{[(000),2]}};
        \node[left] at (3.1,-0.85) {\tiny{[(001),2]}};
        \filldraw[red] (1,0.25) circle (1pt);
        \filldraw[green] (2,0.5) circle (1pt);
        \filldraw[orange] (1,-0.25) circle (1pt);
        \filldraw[blue] (2,-0.5) circle (1pt);
        \draw[-Latex] (0,-1) -- (0,1);
        \draw[red] (0,0) -- (1,0.25) ;
        \draw[green] (1,0.25) -- (2,0.5) ;
        \draw[-Latex] (2,0.5) -- (3,0.75);
        \draw[orange] (0,0) -- (1,-0.25);
        \draw[blue] (1,-0.25) -- (2,-0.5);
        \draw[-Latex] (2,-0.5) -- (3,-0.75);
    
        \draw[shift={(0,-3)}] (-1,0) -- (0,0);
        \node[shift={(0,-3)}][left] at (1.1,0.28) {\tiny{[(100),0]}};
        \node[shift={(0,-3)}][left] at (1.1,-0.28) {\tiny{[(000),0]}};
        \node[shift={(0,-3)}][left] at (2.1,0.56) {\tiny{[(100),1]}};
        \node[shift={(0,-3)}][left] at (2.1,-0.56) {\tiny{[(110),1]}};
        \node[shift={(0,-3)}][left] at (3.1,0.85) {\tiny{[(010),2]}};
        \node[shift={(0,-3)}][left] at (3.1,-0.85) {\tiny{[(011),2]}};
        \filldraw[shift={(0,-3)}][orange] (1,0.25) circle (1pt);
        \filldraw[shift={(0,-3)}][red] (1,-0.25) circle (1pt);
        \filldraw[shift={(0,-3)}][yellow] (2,0.5) circle (1pt);
        \filldraw[shift={(0,-3)}][brown] (2,-0.5) circle (1pt);
        \draw[shift={(0,-3)}][-Latex] (0,-1) -- (0,1);
        \draw[shift={(0,-3)}][orange] (0,0) -- (1,0.25);
        \draw[shift={(0,-3)}][yellow] (1,0.25) -- (2,0.5) ;
        \draw[shift={(0,-3)}][red] (0,0) -- (1,-0.25);
        \draw[shift={(0,-3)}][brown] (1,-0.25) -- (2,-0.5);
        \draw[shift={(0,-3)}][-Latex] (2,0.5) -- (3,0.75);
        \draw[shift={(0,-3)}][-Latex] (2,-0.5) -- (3,-0.75);
    
        \draw[shift={(4.5,0)}] (-1,0) -- (0,0);
        \node[shift={(4.5,0)}][left] at (1.1,0.28) {\tiny{[(010),0]}};
        \node[shift={(4.5,0)}][left] at (1.1,-0.28) {\tiny{[(110),0]}};
        \node[shift={(4.5,0)}][left] at (2.1,0.56) {\tiny{[(010),1]}};
        \node[shift={(4.5,0)}][left] at (2.1,-0.56) {\tiny{[(000),1]}};
        \node[shift={(4.5,0)}][left] at (3.1,0.85) {\tiny{[(100),2]}};
        \node[shift={(4.5,0)}][left] at (3.1,-0.85) {\tiny{[(101),2]}};
        \filldraw[shift={(4.5,0)}][violet] (1,0.25) circle (1pt);
        \filldraw[shift={(4.5,0)}][cyan] (1,-0.25) circle (1pt);
        \filldraw[shift={(4.5,0)}][blue] (2,0.5) circle (1pt);
        \filldraw[shift={(4.5,0)}][green] (2,-0.5) circle (1pt);
        \draw[shift={(4.5,0)}][-Latex] (0,-1) -- (0,1);
        \draw[shift={(4.5,0)}][violet] (0,0) -- (1,0.25) ;
        \draw[shift={(4.5,0)}][blue] (1,0.25) -- (2,0.5) ;
        \draw[shift={(4.5,0)}][cyan] (0,0) -- (1,-0.25);
        \draw[shift={(4.5,0)}][green] (1,-0.25) -- (2,-0.5);
        \draw[shift={(4.5,0)}][-Latex] (2,0.5) -- (3,0.75);
        \draw[shift={(4.5,0)}][-Latex] (2,-0.5) -- (3,-0.75);
    
        \draw[shift={(4.5,-3)}] (-1,0) -- (0,0);
        \node[shift={(4.5,-3)}][left] at (1.1,0.28) {\tiny{[(110),0]}};
        \node[shift={(4.5,-3)}][left] at (1.1,-0.28) {\tiny{[(010),0]}};
        \node[shift={(4.5,-3)}][left] at (2.1,0.56) {\tiny{[(110),1]}};
        \node[shift={(4.5,-3)}][left] at (2.1,-0.56) {\tiny{[(100),1]}};
        \node[shift={(4.5,-3)}][left] at (3.1,0.85) {\tiny{[(110),2]}};
        \node[shift={(4.5,-3)}][left] at (3.1,-0.85) {\tiny{[(111),2]}};
        \filldraw[shift={(4.5,-3)}][cyan] (1,0.25) circle (1pt);
        \filldraw[shift={(4.5,-3)}][violet] (1,-0.25) circle (1pt);
        \filldraw[shift={(4.5,-3)}][brown] (2,0.5) circle (1pt);
        \filldraw[shift={(4.5,-3)}][yellow] (2,-0.5) circle (1pt);
        \draw[shift={(4.5,-3)}][-Latex] (0,-1) -- (0,1);
        \draw[shift={(4.5,-3)}][cyan] (0,0) -- (1,0.25) ;
        \draw[shift={(4.5,-3)}][brown] (1,0.25) -- (2,0.5) ;
        \draw[shift={(4.5,-3)}][violet] (0,0) -- (1,-0.25);
        \draw[shift={(4.5,-3)}][yellow] (1,-0.25) -- (2,-0.5);
        \draw[shift={(4.5,-3)}][-Latex] (2,0.5) -- (3,0.75);
        \draw[shift={(4.5,-3)}][-Latex] (2,-0.5) -- (3,-0.75);
    
        \draw[shift={(0,-7)}] (-1,0) -- (0,0);
        \node[shift={(0,-7)}][left] at (1.1,0.28) {\tiny{[(000),0]}};
        \node[shift={(0,-7)}][left] at (1.1,-0.28) {\tiny{[(100),0]}};
        \node[shift={(0,-7)}][left] at (2.1,0.56) {\tiny{[(000),1]}};
        \node[shift={(0,-7)}][left] at (2.1,-0.56) {\tiny{[(010),1]}};
        \node[shift={(0,-7)}][left] at (3.1,0.85) {\tiny{[(001),2]}};
        \node[shift={(0,-7)}][left] at (3.1,-0.85) {\tiny{[(000),2]}};
        \filldraw[shift={(0,-7)}][red] (1,0.25) circle (1pt);
        \filldraw[shift={(0,-7)}][green] (2,0.5) circle (1pt);
        \filldraw[shift={(0,-7)}][orange] (1,-0.25) circle (1pt);
        \filldraw[shift={(0,-7)}][blue] (2,-0.5) circle (1pt);
        \draw[shift={(0,-7)}][-Latex] (0,-1) -- (0,1);
        \draw[shift={(0,-7)}][red] (0,0) -- (1,0.25) ;
        \draw[shift={(0,-7)}][green] (1,0.25) -- (2,0.5) ;
        \draw[shift={(0,-7)}][-Latex] (2,0.5) -- (3,0.75);
        \draw[shift={(0,-7)}][orange] (0,0) -- (1,-0.25);
        \draw[shift={(0,-7)}][blue] (1,-0.25) -- (2,-0.5);
        \draw[shift={(0,-7)}][-Latex] (2,-0.5) -- (3,-0.75);
    
        \draw[shift={(0,-7)}][shift={(0,-3)}] (-1,0) -- (0,0);
        \node[shift={(0,-7)}][shift={(0,-3)}][left] at (1.1,0.28) {\tiny{[(100),0]}};
        \node[shift={(0,-7)}][shift={(0,-3)}][left] at (1.1,-0.28) {\tiny{[(000),0]}};
        \node[shift={(0,-7)}][shift={(0,-3)}][left] at (2.1,0.56) {\tiny{[(100),1]}};
        \node[shift={(0,-7)}][shift={(0,-3)}][left] at (2.1,-0.56) {\tiny{[(110),1]}};
        \node[shift={(0,-7)}][shift={(0,-3)}][left] at (3.1,0.85) {\tiny{[(011),2]}};
        \node[shift={(0,-7)}][shift={(0,-3)}][left] at (3.1,-0.85) {\tiny{[(010),2]}};
        \filldraw[shift={(0,-7)}][shift={(0,-3)}][orange] (1,0.25) circle (1pt);
        \filldraw[shift={(0,-7)}][shift={(0,-3)}][red] (1,-0.25) circle (1pt);
        \filldraw[shift={(0,-7)}][shift={(0,-3)}][yellow] (2,0.5) circle (1pt);
        \filldraw[shift={(0,-7)}][shift={(0,-3)}][brown] (2,-0.5) circle (1pt);
        \draw[shift={(0,-7)}][shift={(0,-3)}][-Latex] (0,-1) -- (0,1);
        \draw[shift={(0,-7)}][shift={(0,-3)}][orange] (0,0) -- (1,0.25);
        \draw[shift={(0,-7)}][shift={(0,-3)}][yellow] (1,0.25) -- (2,0.5) ;
        \draw[shift={(0,-7)}][shift={(0,-3)}][red] (0,0) -- (1,-0.25);
        \draw[shift={(0,-7)}][shift={(0,-3)}][brown] (1,-0.25) -- (2,-0.5);
        \draw[shift={(0,-7)}][shift={(0,-3)}][-Latex] (2,0.5) -- (3,0.75);
        \draw[shift={(0,-7)}][shift={(0,-3)}][-Latex] (2,-0.5) -- (3,-0.75);
    
        \draw[shift={(0,-7)}][shift={(4.5,0)}] (-1,0) -- (0,0);
        \node[shift={(0,-7)}][shift={(4.5,0)}][left] at (1.1,0.28) {\tiny{[(010),0]}};
        \node[shift={(0,-7)}][shift={(4.5,0)}][left] at (1.1,-0.28) {\tiny{[(110),0]}};
        \node[shift={(0,-7)}][shift={(4.5,0)}][left] at (2.1,0.56) {\tiny{[(010),1]}};
        \node[shift={(0,-7)}][shift={(4.5,0)}][left] at (2.1,-0.56) {\tiny{[(000),1]}};
        \node[shift={(0,-7)}][shift={(4.5,0)}][left] at (3.1,0.85) {\tiny{[(101),2]}};
        \node[shift={(0,-7)}][shift={(4.5,0)}][left] at (3.1,-0.85) {\tiny{[(100),2]}};
        \filldraw[shift={(0,-7)}][shift={(4.5,0)}][violet] (1,0.25) circle (1pt);
        \filldraw[shift={(0,-7)}][shift={(4.5,0)}][cyan] (1,-0.25) circle (1pt);
        \filldraw[shift={(0,-7)}][shift={(4.5,0)}][blue] (2,0.5) circle (1pt);
        \filldraw[shift={(0,-7)}][shift={(4.5,0)}][green] (2,-0.5) circle (1pt);
        \draw[shift={(0,-7)}][shift={(4.5,0)}][-Latex] (0,-1) -- (0,1);
        \draw[shift={(0,-7)}][shift={(4.5,0)}][violet] (0,0) -- (1,0.25) ;
        \draw[shift={(0,-7)}][shift={(4.5,0)}][blue] (1,0.25) -- (2,0.5) ;
        \draw[shift={(0,-7)}][shift={(4.5,0)}][cyan] (0,0) -- (1,-0.25);
        \draw[shift={(0,-7)}][shift={(4.5,0)}][green] (1,-0.25) -- (2,-0.5);
        \draw[shift={(0,-7)}][shift={(4.5,0)}][-Latex] (2,0.5) -- (3,0.75);
        \draw[shift={(0,-7)}][shift={(4.5,0)}][-Latex] (2,-0.5) -- (3,-0.75);
    
        \draw[shift={(0,-7)}][shift={(4.5,-3)}] (-1,0) -- (0,0);
        \node[shift={(0,-7)}][shift={(4.5,-3)}][left] at (1.1,0.28) {\tiny{[(110),0]}};
        \node[shift={(0,-7)}][shift={(4.5,-3)}][left] at (1.1,-0.28) {\tiny{[(010),0]}};
        \node[shift={(0,-7)}][shift={(4.5,-3)}][left] at (2.1,0.56) {\tiny{[(110),1]}};
        \node[shift={(0,-7)}][shift={(4.5,-3)}][left] at (2.1,-0.56) {\tiny{[(100),1]}};
        \node[shift={(0,-7)}][shift={(4.5,-3)}][left] at (3.1,0.85) {\tiny{[(111),2]}};
        \node[shift={(0,-7)}][shift={(4.5,-3)}][left] at (3.1,-0.85) {\tiny{[(110),2]}};
        \filldraw[shift={(0,-7)}][shift={(4.5,-3)}][cyan] (1,0.25) circle (1pt);
        \filldraw[shift={(0,-7)}][shift={(4.5,-3)}][violet] (1,-0.25) circle (1pt);
        \filldraw[shift={(0,-7)}][shift={(4.5,-3)}][brown] (2,0.5) circle (1pt);
        \filldraw[shift={(0,-7)}][shift={(4.5,-3)}][yellow] (2,-0.5) circle (1pt);
        \draw[shift={(0,-7)}][shift={(4.5,-3)}][-Latex] (0,-1) -- (0,1);
        \draw[shift={(0,-7)}][shift={(4.5,-3)}][cyan] (0,0) -- (1,0.25) ;
        \draw[shift={(0,-7)}][shift={(4.5,-3)}][brown] (1,0.25) -- (2,0.5) ;
        \draw[shift={(0,-7)}][shift={(4.5,-3)}][violet] (0,0) -- (1,-0.25);
        \draw[shift={(0,-7)}][shift={(4.5,-3)}][yellow] (1,-0.25) -- (2,-0.5);
        \draw[shift={(0,-7)}][shift={(4.5,-3)}][-Latex] (2,0.5) -- (3,0.75);
        \draw[shift={(0,-7)}][shift={(4.5,-3)}][-Latex] (2,-0.5) -- (3,-0.75);
        % \draw (-1,0) -- (0,0);
        % \node[left] at (1,0.25) {\tiny{[(00),0]}};
        % \node[left] at (1,-0.25) {\tiny{[(10),0]}};
        % \node[left] at (2,0.5) {\tiny{[(00),1]}};
        % \node[left] at (2,-0.5) {\tiny{[(01),1]}};
        % \filldraw[red] (1,0.25) circle (1pt);
        % \filldraw[orange] (1,-0.25) circle (1pt);
        % \draw[-Latex] (0,-0.75) -- (0,0.75);
        % \draw[red] (0,0) -- (1,0.25) ;
        % \draw[-Latex][green] (1,0.25) -- (2,0.5) ;
        % \draw[orange] (0,0) -- (1,-0.25);
        % \draw[-Latex][blue] (1,-0.25) -- (2,-0.5);
    
        % \draw (-1,0) -- (0,0);
        % \node[left] at (1,0.25) {\tiny{[(00),0]}};
        % \node[left] at (1,-0.25) {\tiny{[(10),0]}};
        % \node[left] at (2,0.5) {\tiny{[(00),1]}};
        % \node[left] at (2,-0.5) {\tiny{[(01),1]}};
        % \filldraw[red] (1,0.25) circle (1pt);
        % \filldraw[orange] (1,-0.25) circle (1pt);
        % \draw[-Latex] (0,-0.75) -- (0,0.75);
        % \draw[red] (0,0) -- (1,0.25) ;
        % \draw[-Latex][green] (1,0.25) -- (2,0.5) ;
        % \draw[orange] (0,0) -- (1,-0.25);
        % \draw[-Latex][blue] (1,-0.25) -- (2,-0.5);  
    \end{tikzpicture}
    \caption{The construction of $\mathcal{Z}_2$, the parts with the same key are attached. The author find it too difficult 
    to fill different colors for different parts, since there are too many parts.}
\end{figure}
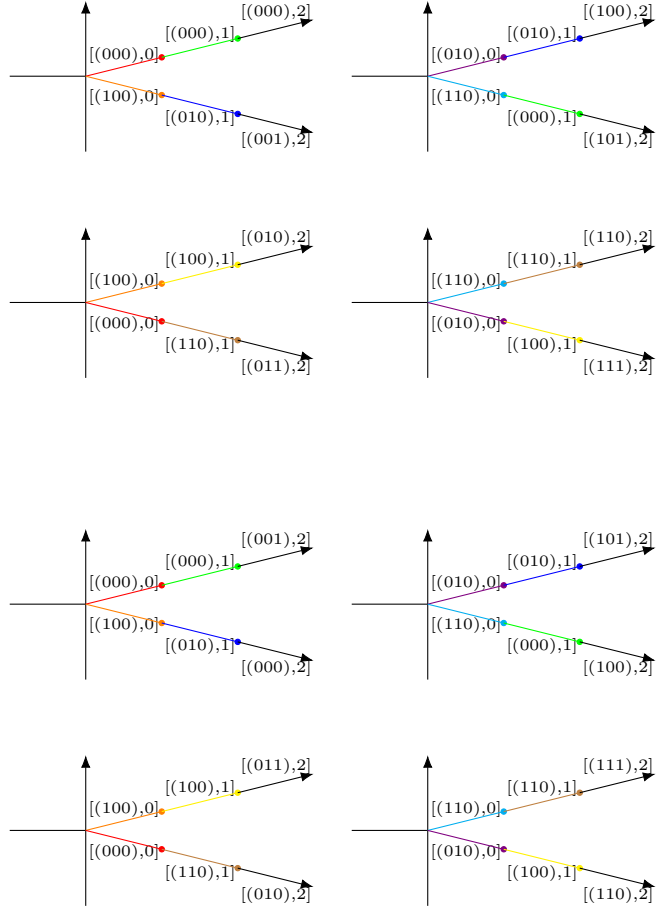

The preceding procedure determines the topological structure of $\mathcal{Z}_{k+1}$. As for the 
square root function $\sqrt{\bullet-\sigma_j^2}$ for $j=0,1\cdots k$, we define them on $\mathcal{Z}_{k+1}$ 
by assigning them the same values as on $\mathcal{Z}_{k}$ in the copy $\mathcal{Z}_{k,0}$,
and by taking the negatives of those values in $\mathcal{Z}_{k,1}$.
For $\sqrt{\bullet-\sigma_{k+1}^2}$, we specify that its argument lies in $[0, \pi]$ on $\mathcal{Z}_{k,0}$
and in $[\pi, 2\pi]$ on $\mathcal{Z}_{k,1}$.

Now for those $z$ whose image under the natural projection $\mathcal{Z}_k\to \mathbb{C}$ 
does not coincide with any branch points $\{\sigma_{j}^2\}_{j=0}^{k+1}$, the conformal structure near $z$ is the pull-back of the 
natural conformal structure in $\mathbb{C}$ through this projection. 
If $z$ maps to a branch point $\sigma_j^2$, then the conformal structure is instead defined as the pullback
through the local square root map $\sqrt{\bullet - \sigma_j^2}$.

The Riemann surface $\hat{\mathcal{Z}}$ is 
defined as the limit of $\mathcal{Z}_k$ 
in some sense. More precisely, 
consider the open subset $\tilde{\mathcal{Z}}_k$ of $\mathcal{Z}_k$ defined by 
\[
    \tilde{\mathcal{Z}}_k:=\text{preimage of }\mathbb{C}\setminus[\sigma_k^2,+\infty) \text{ under the projection } \tilde{\mathcal{Z}}_k\to \mathbb{C}
\]
The inclusion map 
$\tilde{\mathcal{Z}}_k\to \tilde{\mathcal{Z}}_{k+1}$ for $k\in \mathbb{N}_0$, is defined as the natural embedding
into the first copy $\mathcal{Z}_{k,0} \subset \mathcal{Z}_{k+1}$(recall that $\mathcal{Z}_{k+1} = \mathcal{Z}_{k,0} \cup \mathcal{Z}_{k,1}$ by construction).
This inclusion is holomorphic, so we define $\hat{\mathcal{Z}}$ 
as the inductive limit toplogical space of $\{\tilde{\mathcal{Z}}_k\}_{k\geq 0}$, with the complex structure inhertied from $\tilde{\mathcal{Z}}_k$ for each $k\geq 0$. 
The square root function $\tau_k$ on $\hat{\mathcal{Z}}$ is then a well-defined analytic function, since in each $\tilde{\mathcal{Z}}_j$ for 
$j\in \mathbb{N}_0$ we define an analytic square root(For those $j\leq k$ we simply take the argument of the square root in $(0,\pi)$, using the fact that 
we remove the inverse image of $[\sigma_k,\infty)$), which are all 
compatible.

The \textit{physical region}, or sometimes referred as \textit{physical space}, will mean the image of $\tilde{\mathcal{Z}}_0$ in $\hat{\mathcal{Z}}$, corresponding to 
the original $\mathbb{C}-\mathbb{R}_{\geq 0}$ where the free resolvent is initially defined. 
We will use parametrization $\lambda\mapsto z=\lambda^2$ where $\operatorname{Im} \lambda>0$ in the physical space, 
and it will be continuously extended to $\lambda\in \mathbb{R}$. Note that 
for $\lambda\in \mathbb{R}$ 
we have $\tau_k(\lambda)=\operatorname{sgn}(\lambda)
\sqrt{\lambda^2-\sigma_k^2}$ if $|\lambda|\geq \sigma_k$, 
and $\tau_k(\lambda)=i\sqrt{\sigma_k^2-\lambda^2}$ if $|\lambda|<\sigma_k$. 
For $\operatorname{Im} \lambda>0$, we have 
\[
    \tau_k(\lambda)=-\overline{\tau_{k}(-\bar{\lambda})}
\]
In the following exposition, we may use notation $\lambda$ where $\operatorname{Im} \lambda\geq 0$
to represent its image $z\in \hat{\mathcal{Z}}$ under this parametrization.

\begin{remark}
    We remark the conformal chart near thresholds, say $\sigma_{q}$. We may assume $\sigma_{q-1}<\sigma_{q}<\sigma_{q+1}$ 
    and assume $\sigma_q\neq 0$. Then the local chart near $\lambda=+\sigma_q\in \hat{\mathcal{Z}}$ is given by 
    $\zeta=\tau_k(z)$, where the physical region near $\lambda$ corresponds to the set $\operatorname{Im} \zeta>0,\operatorname{Re} \zeta>0$.
    Similarly the local chart near $\lambda=-\sigma_q\in \hat{\mathcal{Z}}$ is also given by 
    $\zeta=\tau_k(z)$, where the physical region near $\lambda$ corresponds to the set $\operatorname{Im} \zeta>0,\operatorname{Re} \zeta<0$.
    They are illustrated in the following figure.
\end{remark}

\begin{figure}[h]
    \centering
\begin{tikzpicture}
    \tikzset{->-/.style=
        {decoration={markings,mark=at position #1 with 
        {\arrow{latex}}},postaction={decorate}}}

\draw (-1,0) -- (0,0);
% \node[left] at (1.1,0.28) {\tiny{[(000),0]}};
% \node[left] at (1.1,-0.28) {\tiny{[(100),0]}};
% \node[left] at (2.1,0.56) {\tiny{[(000),1]}};
% \node[left] at (2.1,-0.56) {\tiny{[(010),1]}};
% \node[left] at (3.1,0.85) {\tiny{[(000),2]}};
% \node[left] at (3.1,-0.85) {\tiny{[(001),2]}};
\filldraw (1,0.25) circle (1pt) node[anchor=north] {$\sigma_{q-1}^2$};
\filldraw[green] (2,0.5) circle (1pt) node[anchor=north] {$\sigma_q^2$};
\filldraw (3,0.75) circle (1pt) node[anchor=north west] {$\sigma_{q+1}^2$};
\filldraw (1.5,0.375) circle (1pt) node[anchor=north] {A};
\filldraw (1.5+0.9659*0.512*2,0.375+0.512*2*0.2588) circle (1pt) node[anchor=north] {B};

\draw[-Latex] (0,-1.5) -- (0,1.5);
\draw[very thick][black] (0,0) -- (1,0.25) ;
\draw[green] (1,0.25) -- (2,0.5) ;
\draw[->-=0.5] (1.5,0.375) arc(180+11.25:15:0.512);
\draw[red] (2,0.5) -- (3,0.75);
\draw[-Latex][very thick][gray] (3,0.75) -- (4,1);
% \draw[-Latex] (0,0) -- (4,-1);

\node[anchor=south] at (6,0) {$\tau_q$};
\draw[->] (5,0)--(7,0);

\filldraw[shift={(10,-0.5)}] (0,1) circle (1pt) node[anchor=east] {$\sqrt{\sigma_{q}^2-\sigma_{q-1}^2}$};
\filldraw[shift={(10,-0.5)}][green] (0,0) circle (1pt) ;
\filldraw[shift={(10,-0.5)}] (1,0) circle (1pt) node[anchor=south west] {$\sqrt{\sigma_{q+1}^2-\sigma_{q}^2}$};
\filldraw[shift={(10,-0.5)}] (0,0.5) circle (1pt) node[anchor=east] {A};
\filldraw[shift={(10,-0.5)}] (0.5,0) circle (1pt) node[anchor=north] {B};

\draw[shift={(10,-0.5)}][-Latex] (0,-1) -- (0,2.5);
\draw[shift={(10,-0.5)}] (-1,0) -- (0,0); 
\draw[shift={(10,-0.5)}][very thick][black] (0,1) -- (0,2.5);
\draw[shift={(10,-0.5)}][green] (0,0) -- (0,1);
\draw[shift={(10,-0.5)}][red] (0,0) -- (1,0);
\draw[shift={(10,-0.5)}][very thick][gray][-Latex] (1,0) -- (2.5,0);
\draw[shift={(10,-0.5)}][->-=0.5] (0,0.5) arc(90:0:0.5);
% \draw[orange] (0,0) -- (1,-0.25);
% \draw[blue] (1,-0.25) -- (2,-0.5);
% \draw[-Latex] (2,-0.5) -- (3,-0.75);
\end{tikzpicture}
\caption{The conformal chart near $\lambda=+\sigma_q$. The thick black and gray lines are removed.}
\end{figure}

\begin{figure}[h]
    \centering
\begin{tikzpicture}
    \tikzset{->-/.style=
        {decoration={markings,mark=at position #1 with 
        {\arrow{latex}}},postaction={decorate}}}

\draw (-1,0) -- (0,0);
% \node[left] at (1.1,0.28) {\tiny{[(000),0]}};
% \node[left] at (1.1,-0.28) {\tiny{[(100),0]}};
% \node[left] at (2.1,0.56) {\tiny{[(000),1]}};
% \node[left] at (2.1,-0.56) {\tiny{[(010),1]}};
% \node[left] at (3.1,0.85) {\tiny{[(000),2]}};
% \node[left] at (3.1,-0.85) {\tiny{[(001),2]}};
\filldraw (1,-0.25) circle (1pt) node[anchor=south] {$\sigma_{q-1}^2$};
\filldraw[purple] (2,-0.5) circle (1pt) node[anchor=south] {$\sigma_q^2$};
\filldraw (3,-0.75) circle (1pt) node[anchor=south west] {$\sigma_{q+1}^2$};
\filldraw (1.5,-0.375) circle (1pt) node[anchor=south] {A};
\filldraw (1.5+0.9659*0.512*2,-0.375-0.512*2*0.2588) circle (1pt) node[anchor=south] {B};

\draw[-Latex] (0,-1.5) -- (0,1.5);
\draw[very thick][black] (0,0) -- (1,-0.25) ;
\draw[purple] (1,-0.25) -- (2,-0.5) ;
\draw[->-=0.5] (1.5,-0.375) arc(180-11.25:360-15:0.512);
\draw[blue] (2,-0.5) -- (3,-0.75);
\draw[-Latex][very thick][gray] (3,-0.75) -- (4,-1);
% \draw[-Latex] (0,0) -- (4,-1);

\node[anchor=south] at (6,0) {$\tau_q$};
\draw[->] (5,0)--(7,0);

\filldraw[shift={(10,-0.5)}] (0,1) circle (1pt) node[anchor=west] {$\sqrt{\sigma_{q}^2-\sigma_{q-1}^2}$};
\filldraw[shift={(10,-0.5)}][purple] (0,0) circle (1pt);
\filldraw[shift={(10,-0.5)}] (-1,0) circle (1pt) node[anchor=north east] {$\sqrt{\sigma_{q+1}^2-\sigma_{q}^2}$};
\filldraw[shift={(10,-0.5)}] (0,0.5) circle (1pt) node[anchor=west] {A};
\filldraw[shift={(10,-0.5)}] (-0.5,0) circle (1pt) node[anchor=north] {B};

\draw[shift={(10,-0.5)}][-Latex] (0,-1) -- (0,2.5);
\draw[shift={(10,-0.5)}][-Latex] (0,0) -- (1,0); 
\draw[shift={(10,-0.5)}][very thick][black] (0,1) -- (0,2.5);
\draw[shift={(10,-0.5)}][purple] (0,0) -- (0,1);
\draw[shift={(10,-0.5)}][blue] (0,0) -- (-1,0);
\draw[shift={(10,-0.5)}][very thick][gray] (-1,0) -- (-2.5,0);
\draw[shift={(10,-0.5)}][->-=0.5] (0,0.5) arc(90:180:0.5);
% \draw[orange] (0,0) -- (1,-0.25);
% \draw[blue] (1,-0.25) -- (2,-0.5);
% \draw[-Latex] (2,-0.5) -- (3,-0.75);
\end{tikzpicture}
\caption{The conformal chart near $\lambda=-\sigma_q$. The thick black and gray lines are removed.}
\end{figure}

\subsection{Resolvent for general $P_V$}
For general $V\in L^\infty_{\text{comp}}(X;\mathbb{R})$, we 
can use the standard method to define the resolvent $R_V(z)$ as \cite[Theorem 2.2]{MathematicalTheoryofScatteringResonances}.
\begin{proposition}\label{prop:Analytic continuation of RV}
    We can uniquely define
    \[
        R_V(z):L^2_{\operatorname{comp}}(X)\to H^2_{\operatorname{loc}}(X)
    \]
    as a family of operators depending meromorphically on $z\in \hat{\mathcal{Z}}$, so that when $-z$ lies in 
    the physical region and for sufficiently large $z\in \mathbb{R}$ 
    \[
        R_V(z)=(P_V+z)^{-1}:L^2(X)\to H^2(X)
    \] 
    as the usual resolvent of the self-adjoint operator $P_V:H^2(X)\subset L^2(X)\to L^2(X)$.
\end{proposition}

\begin{proof}
    Choose any $\rho\in C_c^\infty(\mathbb{R}^n)$ equals to one in a neighborhood of $\operatorname{supp} V$, we can define 
    \begin{equation}\label{eq:expression of RV}
        R_V(z):=R_0(z)(I+VR_0(z)\rho)^{-1}(I-VR_0(z)(1-\rho))
    \end{equation}
where we see $I+VR_0(z)\rho:L^2(X)\to L^2(X)$ is a Fredholm operator thanks to the Sobolev compact embedding in bounded regions, 
and the inverse exists as a family of operators $L^2(X)\to L^2(X)$ depending meromorphically on $z\in \hat{\mathcal{Z}}$. We only need to check two things.
\begin{itemize}
    \item The first thing is that $R_V$ is the true resolvent $(P_V+z)^{-1}$ when $-z$ lies in 
    the physical region and for sufficiently large $z\in \mathbb{R}$. Actually we have 
    \[
    \begin{aligned}
        P_V+z&=P_0+z+V=(I+VR_0(-z))(P_0+z)\\
        &=\left(I+VR_0(-z)(1-\rho)\right)\left(I+VR_0(-z)\rho\right)(P_0+z)
    \end{aligned}
    \]
    and it's easy to see that 
    \[
        \left(I+VR_0(-z)(1-\rho)\right)^{-1}=I-VR_0(-z)(1-\rho)
    \]
    And the expression \eqref{eq:Expression of R0 lambda} on $R_0$ in terms of $R_0^{\mathbb{R}^n}$, and the estimate on $R_0^{\mathbb{R}^n}$
    given by spectral theorem 
    \[
        ||R_0^{\mathbb{R}^n}(-z)||_{L^2(\mathbb{R}^n)\to L^2(\mathbb{R}^n)}\leq |z|^2
    \]
    implie that $||R_0(-z)||_{L^2\to L^2}\leq ||V||_{L^\infty}^{-1}/2$ for $z\in \mathbb{R}$ sufficiently large. Thus 
    we can take the inverse for both sides to obtain \eqref{eq:expression of RV}.
    \item The second thing is that $(I+VR_0(z)\rho)^{-1}(I-VR_0(z)(1-\rho))$ maps $L^2_{\operatorname{comp}}(X)$ to $L^2_{\operatorname{comp}}(X)$, 
    and it suffices to show 
    \[
        (I+VR_0(z)\rho)^{-1}:L^2_{\operatorname{comp}}(X)\to L^2_{\operatorname{comp}}(X)
    \]
    For $\chi_1,\chi_2\in C_c^\infty(\mathbb{R})$ so that $\chi_1=1$ in a neighborhood of $\operatorname{supp} \rho$  
    and $\chi_2=1$ in in a neighborhood of $\operatorname{supp} \chi_1$, we want to show 
    \[
        (1-\chi_2)(I+VR_0(z)\rho)^{-1}\chi_1=0
    \]
    Actually, given $f\in L^2(X)$, let $u=(I+VR_0(z)\rho)^{-1}(\chi_1f)$, then we see since 
    \[
        u=\chi_1f-VR_0(z)\rho u
    \]
    so $(1-\chi_2)u=0$. This completes the proof.
\end{itemize}

\end{proof}

\begin{remark}\label{rmk:2.2.15}
    It is also useful to express the operator $(I+VR_0(z)\rho)^{-1}$ in terms of $R_V(z)$ 
    \begin{equation}\label{eq:3.2.2}
        (I+VR_0(z)\rho)^{-1}=I-VR_V(\lambda)\rho
    \end{equation}
    This follows from direct calculation for $z$ lying in the physical space with $-z\gg 1$ 
    and analytic continuation
    \[
        (I+VR_0(z)\rho)(I-VR_V(\lambda)\rho)=I+V(R_0(z)-R_V(z))\rho-VR_0(z)VR_V(z)\rho=I
    \]
    where we use the resolvent identity
    \[
        R_0(z)-R_V(z)=R_0(z)VR_V(z)
    \]

    Moreover, we should notice that $R_V(z)$ is symmetric, that is 
    \[
        R_V(z,(x_1,y_1),(x_2,y_2))=R_V(z,(x_2,y_2),(x_1,y_1)) \quad z\in \hat{\mathcal{Z}};\;(x_1,y_1),(x_2,y_2)\in X
    \]
    Actually, this symmetry holds for those $-z\gg 1$ in the physical space by the property 
    of the resolvent, and thus it holds for any $z\in \hat{\mathcal{Z}}$ by analytic continuation.
\end{remark}

The following lemma concerning the Laurent expansion of $R_V$ near $\lambda\in \mathbb{R}$ is standard.
\begin{lemma}\label{lem:Laurent expansion of RV near real line}
    Let $\lambda_0\in \mathbb{R}$.
    \begin{itemize}
        % \item $\Pi_{\lambda_0}$ is a finite-rank operator.
        \item If $\lambda_0$ is not a threshold, then there exists 
        $
        B(z):L^2_{\operatorname{comp}}(X)\to H^2_{\operatorname{loc}}(X)
        $
        holomorphic for $z$ near $\lambda_0$ in $\hat{\mathcal{Z}}$, such that 
        \[
            R_V(z)=-\frac{\Pi_{\lambda_0}}{\tau_0(z)^2-
            \lambda_0^2}+B(z)
        \]
        for $z$ near $\lambda_0^2$ in $\hat{\mathcal{Z}}$.
        \item If $\lambda_0$ is a threshold, that is $|\lambda_0|=\sigma_k$ for some $k\geq 0$,
        then there exists $A_1,B(z):L^2_{\operatorname{comp}}(X)\to H^2_{\operatorname{loc}}(X)$
        with $A_1$ independent of $z$ and $B$ holomorphic for $z$ near $\lambda_0$ in $\hat{\mathcal{Z}}$, 
        such that 
        \[
            R_V(z)=-\frac{\Pi_{\lambda_0}}{\tau_j(z)^2}+
            \frac{A_1}{\tau_j(z)}+B(z)
        \]
        for $z$ near $\lambda_0$ in $\hat{\mathcal{Z}}$. And we have 
        $(P_V-\lambda_0^2)A_1=0$.
    \end{itemize}
\end{lemma}

\begin{proof}
    Our proof essentially follows \cite[Lemma 2.3]{christiansen2022wave}. For reader's convenience we give a 
    detailed exposition.
    \begin{itemize}
        % \item We need to show $\operatorname{ran} \Pi_{\lambda_0}$ is finite-dimensional. 
        % Actually we can write for any $u\in L^2(X)$
        % \[
        %     (P_V-\lambda)u=(P_0-\lambda)(I+R_0(\lambda)V)u=(P_0-\lambda)(I+\rho R_0(\lambda)V)(I+(1-\rho) R_0(\lambda)V)u
        % \]
        % where $\rho\in C_c^\infty(\mathbb{R}^n)$ equals one in a neighborhood of $\operatorname{supp} V$. Now since $P_0$ has 
        % no point spectrum, and the operator $(I+(1-\rho) R_0(\lambda)V)$ has inverse $(I-(1-\rho) R_0(\lambda)V)$, we know 
        % \[
        %     \dim \operatorname{ran} \Pi_{\lambda_0}=\dim \ker (I+\rho R_0(\lambda)V:L^2\to L^2)
        % \]
        % which is finite since $I+\rho R_0(\lambda)V$ is Fredholm on $L^2$.
        \item If $\lambda_0$ is not a threshold, then $\hat{\mathcal{Z}}\ni z \mapsto \tau_0(z)^2$ gives a local chart  
        of $\hat{\mathcal{Z}}$ near $\lambda_0$. Spectral theorem gives for $\tau_0(z)$ lying in the upper plane
        \begin{equation}\label{eq:estimate of RV in step 1 of proof of lem:Laurent expansion of RV near real line}
            ||R_V(z)||_{L^2\to L^2}\leq \frac{1}{|\operatorname{Im} \tau_0(z)^2|}
        \end{equation}
        Taking for example $\tau_0(z)^2=\lambda_0+i\operatorname{sgn}(\lambda)\epsilon$ where $\epsilon>0$ so that $z$ 
        lies in the physical space, and letting $\epsilon\to 0$, we see the Laurent expansion of $R_V$ near $\lambda_0$ must be of the form 
        \[
            R_V(z)=\frac{A_1}{\tau_0(z)^2-\lambda_0^2}+B(z)
        \]
        where $A_1,B(z):L^2_{\operatorname{comp}}\to H^2_{\operatorname{loc}}$ and $B$ depends holomorphically on $z$. 
        Then by estimate \eqref{eq:estimate of RV in step 1 of proof of lem:Laurent expansion of RV near real line} 
        the operator $A_1$ is actually bounded $L^2\to L^2$, since for any $u,v\in L^2_{\operatorname{comp}}$ we can for any $\epsilon>0$
        choose $z_{\epsilon}$ lying in the physical space so that $\tau_0(z_{\epsilon})^2=\lambda_0+i\operatorname{sgn}(\lambda)\epsilon$ and 
        \[
        \begin{aligned}
            \left|\langle A_1u,v\rangle_{L^2(X)}\right|&=\left|\lim_{\epsilon\to 0^+} \epsilon\langle R_V(\lambda)u,v\rangle\right|\\
            &\leq \epsilon ||R_V(z)||_{L^2\to L^2} ||u||_{L^2}||v||_{L^2}\leq ||u||_{L^2}||v||_{L^2}
        \end{aligned}
        \]
        And we use the identity $(P_V-\tau_0(z)^2)R_V(z)=\operatorname{id}$ to write 
        \[
            \begin{aligned}
                (P_V-\tau_0(z)^2)R_V(z)=\frac{(P_V-\lambda_0^2)A_1}{\tau_0(z)^2-\lambda_0^2}+\left((P_V-\tau_0(z)^2)B(z)-A_1\right)=\operatorname{id}
            \end{aligned}
        \]
        so by letting $\tau_0(z)\to \lambda_0$ we obtain 
        \[
            (P_V-\lambda_0^2)A_1=0
        \]
        and 
        \[
            (P_V-\lambda_0^2-i\operatorname{sgn}(\lambda)\epsilon)B(z_\epsilon)-A_1=\operatorname{id}
        \]
        The first formula implies that $\Pi_{\lambda_0}A_1=A_1$, 
        thus we can use the Laurent expansion and estimate \eqref{eq:estimate of RV in step 1 of proof of lem:Laurent expansion of RV near real line} to deduce
        \[
            \begin{aligned}
                 R_V(z_\epsilon)-\Pi_{\lambda_0}\frac{A_1}{\epsilon}=&B(z_\epsilon):L^2(X)\to H^2(X) \\
                &||B(z_\epsilon)||_{L^2\to L^2}\leq \frac{2}{\epsilon}
            \end{aligned}
        \]
        Now we can compose $\Pi_{\lambda_0}$ at left on the second formula, noting that $\operatorname{Ran} B(z_\epsilon)\subset H^2(X)$ 
        so we can swap $\Pi_{\lambda_0}$ and $P_V$, what is left is 
        \[
            \pm i \Pi_{\lambda_0}(B(z_\epsilon)\epsilon)-A_1=\Pi_{\lambda_0}
        \]
        We claim this leads to $A_1=-\Pi_{\lambda_0}$. In fact, given any $u\in C_c^\infty(X)$, we see $v_\epsilon:=B(z_\epsilon)\epsilon u$ 
        is bounded in $L^2(X)$ uniformly for $\epsilon>0$, thus $v_\epsilon$ converges weakly to some $w\in L^2(X)$ as $\epsilon\to 0$ up to some subsequence. 
        However, since $B(z):L^2_{\operatorname{comp}}\to H^2_{\operatorname{loc}}$ is continuous for $z$, we know $w$ must be zero. Letting $\epsilon\to 0$ 
        we know $-A_1u=\Pi_{\lambda_0}u$, and thus $A_1=-\Pi_{\lambda_0}$ since $u$ is arbitray.
        \item If $\lambda_0=\pm \sigma_k$ for some $\sigma_k$, then the conformal chart near $\lambda_0$ is given by $\mathbb{C}\ni \zeta\mapsto z\in \hat{\mathcal{Z}}$,  
        where $z$ is determined by $\tau_k(z)=\zeta$. Then we see $z=\lambda_0^2+\zeta^2$ 
        lying in the physical space corresponds to $\arg {\zeta}\in (0,\frac{\pi}{2})$ if $\lambda_0=+\sigma_k$, 
        while if $\lambda_0=-\sigma_k$ it corresponds to $\arg {\zeta}\in (\pi,\frac{3}{2}\pi)$. Then the spectral 
        theorem implies for $\zeta$ corresponding to the physical space 
        \begin{equation}\label{eq:estimate of RV in step 2 of proof of lem:Laurent expansion of RV near real line}
            ||R_V(z(\zeta))||_{L^2\to L^2}\leq \frac{1}{\operatorname{Im} \zeta^2}
        \end{equation}
        So the Laurent expansion of $R_V$ near $\lambda_0$, or equivalently near $\zeta=0$ must be of the form 
        \[
            R_V(z(\zeta))=\frac{A_2}{\zeta^2}+\frac{A_1}{\zeta}+B(z(\zeta))
        \]
        Using the identity $(P_V-\lambda_0^2-\zeta^2)R_V(z(\zeta))=\operatorname{id}$ we obtain for $\zeta$ corresponding to the physical space 
        \[
            \frac{(P_V-\lambda_0^2)A_2}{\zeta^2}+\frac{(P_V-\lambda_0^2)A_1}{\zeta}+\left(-A_2-\zeta A_1+(P_V-\lambda_0^2)B(z(\zeta))-\zeta^2 B(z(\zeta))\right)=\operatorname{id}
        \]
        The same argument as above shows that $A_2$ is $L^2\to L^2$ bounded and $\Pi_{\lambda_0} A_2=A_2$, and we have 
        \begin{equation}\label{eq:Consequence of PV-lambda0-zeta RV=id}
            \begin{aligned}
                (P_V-\lambda_0^2)A_2=0,\quad (P_V-\lambda_0^2)A_1=0\\
                -A_2-\zeta A_1+(P_V-\lambda_0^2-\zeta^2)B(z(\zeta))=\operatorname{id}
            \end{aligned}
        \end{equation}
        Next we define $C(\zeta)$ for $\zeta$ corresponding to the physical space via
        \begin{equation}\label{eq:the definition of Czeta}
            C(\zeta):=\zeta A_1+\zeta^2B(\zeta)=\zeta^2(R_V(z(\zeta))-\Pi_{\lambda_0} A_2):L^2(X)\to H^2(X)
        \end{equation}
        We can choose $\zeta_\epsilon$ for any $\epsilon>0$ corresponding to the physical space so that $\zeta_\epsilon=\pm i \epsilon$, 
        so we have by \eqref{eq:Consequence of PV-lambda0-zeta RV=id}
        \[
            ||C(\zeta_{\epsilon})||_{L^2(X)\to L^2(X)}\leq 2,\quad \forall \epsilon>0
        \]
        Returning to \eqref{eq:Consequence of PV-lambda0-zeta RV=id} we obtain 
        \[
        \begin{aligned}
            \operatorname{id}&=-A_2-\zeta_{\epsilon} A_1+(P_V-\lambda_0^2-\zeta_{\epsilon}^2)B(z(\zeta_{\epsilon}))\\
            &=-A_2-\zeta_{\epsilon} A_1+(P_V-\lambda_0^2-\zeta_{\epsilon}^2)\frac{C(\zeta_{\epsilon})-\zeta_{\epsilon} A_1}{\zeta_{\epsilon}^2}\\
            &=-A_2+(P_V-\lambda_0^2-\zeta_{\epsilon}^2)\frac{C(\zeta_{\epsilon})}{\zeta_{\epsilon}^2}
        \end{aligned}
        \]
        Since $\operatorname{Ran} C(\zeta_{\epsilon})\subset H^2(X)$, we can compose $\Pi_{\lambda_0}$ at left to obtain 
        \[
            \Pi_{\lambda_0}=-A_2-\Pi_{\lambda_0}C(\zeta_{\epsilon})
        \]
        The same argument as the case that $\lambda_0$ is not a threshold then leads to that 
        \[
            \Pi_{\lambda_0}=-A_2
        \]
        as desired.
    \end{itemize}
\end{proof}

\subsection{Rellich's uniqueness theorem}

As in the case of scattering in Euclidean space, we have the following form
of Rellich's uniqueness theorem.
\begin{theorem}[Rellich's uniqueness theorem]\label{thm:Rellich's uniqueness theorem}
    Suppose the potential $V$ is real-valued 
    with support contained in $B\times M$ 
    where $B\subset \mathbb{R}^n$ is a 
    ball centered at zero. Let $\lambda\in \mathbb{R}-\{0\}$. 
    Suppose $u\in H^2_{\operatorname{loc}}(X)$ has
    expansion with respect to the orthonormal basis $\{\varphi_k\}_{k\geq 0}$ of $L^2(M)$
    \[
        u(x,y)=\sum_{k\geq 0} u_k(x)\otimes \varphi_k(y):=\sum_{\sigma_k\leq |\lambda|} u_k(x)\otimes \varphi_k(y)+R(x,y)
    \]
    satisfying 
    \[
        (P_V-\lambda^2)u=0
    \]
    and the 
    following \textbf{outgoing condition}
    \begin{equation}\label{eq:Sommerfield radiation condition}
        \begin{aligned}
        &(\partial_r-i\tau_k(\lambda))u_k(x,y)=\mathcal{O}(|x|^{-\frac{n-1}{2}}), \quad |x|\to +\infty, \sigma_k\leq |\lambda|\\
        &|\nabla R(x,y)|, |R(x,y)|=\mathcal{O}(e^{-\epsilon|x|}), \quad |x|\to +\infty
        \end{aligned}
    \end{equation}
    for some $\epsilon>0$.
    Then $u_k$ vanishes outside $B$ for each $\sigma_k<|\lambda|$.
\end{theorem}

Before proving this theorem, we first show that 
functions lying in the range of $R_0$ satisfiy the outgoing condition \eqref{eq:Sommerfield radiation condition}.
\begin{lemma}\label{lem:Range of R0 is outgoing}
    If $u=R_0(\lambda)g$ for some $g\in L^2_{\text{comp}}(X)$, 
    then $u$ satisfies the outgoing condition
    \eqref{eq:Sommerfield radiation condition}.
\end{lemma}

\begin{proof}
    We can write 
    \[
        u(x,y)=\sum_{k\geq 0} u_k(x)\otimes \varphi_k(y),\quad g(x,y)=\sum_{k\geq 0} g_k(x)\otimes \varphi_k(y)
    \]
    with respect to the orthonormal basis $\{\varphi_k\}_{k\geq 0}$ of $L^2(M)$, then we have $u_k=R_0^{\mathbb{R}^n}(\tau_k(\lambda))g_k$.
    There are three cases for $k$:
    \begin{itemize}
        \item Suppose $\sigma_k<|\lambda|$, then it follows from the definition 
        of $R_0(\lambda)$, and the outgoing asymptotics of the free  
        resolvent $R_0^{\mathbb{R}^n}(\lambda)$ in $\mathbb{R}^n$, see Proposition \ref{thm:3.5}.
        \item Suppose $\sigma_k=|\lambda|$, then by the explicit expression of $R_0^{\mathbb{R}^n}(\lambda)$ in proposition 
        \ref{thm:3.3} we have 
        \[
            u_k(x)=C_n\int_{\mathbb{R}^n} \frac{g_k(y)}{|x-y|^{n-2}}dy=C_n\int_{\mathbb{R}^n} 
            \frac{g_k(y)}{|x|^{n-2}}\left(1-(n-2)\langle \frac{x}{|x|},y\rangle|x|^{-1}+\mathcal{O}(|x|^{-2})\right)dy
        \]
        the remaining term can be differentiated. By direct calculation we see 
        \[
            \partial_r u_k=\mathcal{O}(r^{1-n})
        \]
        which suffices since we have $n\geq 3$.
        \item Suppose $\sigma_k>|\lambda|$. 
        % We will actually shows that 
        % for some $\epsilon(\lambda)>0$
        % \[
        %     |\partial_{x}^\alpha \sum_{\sigma_k>|\lambda|} u_k(x)\otimes \varphi_k(y)|
        %     =\mathcal{O}_{\alpha}(e^{-\epsilon(\lambda)|x|}),\quad |x|>>1
        % \]
        By the explicit expression of $R_0^{\mathbb{R}^n}(\lambda)$ in proposition 
        \ref{thm:3.3} we have 
        we have when $|x|\gg 1$ 
        \[
            |u_k(x)|\leq C \int_{\mathbb{R}^n} \frac{e^{-\sqrt{\sigma_k^2-\lambda^2}|x-y|}}{|x-y|^{n-2}} |g_k(y)|dy
            \leq C_g e^{-\sqrt{\sigma_k^2-\lambda^2}|x|/2}
        \]
        and similar estimate holds for derivatives with respect to $x$. To do summation over $k$, 
        we note that according to Weyl's law
        \[
            |\{k\geq 0:\lambda\leq \sigma_k\leq \lambda+1\}|=\mathcal{O}(\lambda^{\operatorname{dim} M-1})
        \]
        and the fact that there exists $M_s>0$ for each $s\in \mathbb{N}$ so that 
        \[
            ||\varphi_k||_{C^s(M)}=\mathcal{O}(\sigma_k^{M_s})
        \]
        Thus for each $s\in \mathbb{N}$
        \[
            \begin{aligned}
            ||u||_{C^s}&\leq C \sum_{\sigma_k>\lambda} e^{-\sqrt{\sigma_k^2-\lambda^2}|x|/2}\sigma_k^{M_s}\\
            &\leq Ce^{-\epsilon(\lambda)|x|/2}
            +C\int_{\lambda+1}^\infty
            e^{-\sqrt{t^2-\lambda^2}|x|/2}t^{M_s}\left|\{k\geq 0:t\leq \sigma_k\leq t+1\}\right| dt\\
            &\leq C'e^{-\epsilon(\lambda)|x|/2}
            \end{aligned}
        \]
        as desired.
    \end{itemize}
\end{proof}

\begin{proof}[Proof of Theorem \ref{thm:Rellich's uniqueness theorem}]
    % Note 
    % \[
    %     Vu=(P_0-\lambda^2)u
    % \]
    % Let $w=-R_0(\lambda)g$, we have $Vw=g$ and 
    % \[
    % (P_V-\lambda^2)w=0,\quad w=-R_0(\lambda)g
    % \]
    % Write 
    % \[
    %     w=\sum_{k=0}^\infty w_k\otimes \phi_k
    % \]
    The proof is essentially the same as that in Euclidean space.
    Choose $\chi\in C_c^\infty(\mathbb{R}^n)$ 
    so that $\chi=1$ in a neighborhood of $B$. Define 
    \[
        f:=(-\Delta_X-\lambda^2)(1-\chi)u=[\Delta_{\mathbb{R}^n},\chi]u\in C_c^\infty(X)
    \]
    where we use elliptic regularity. Then we define 
    \[
        w:=(1-\chi)u-R_0(\lambda)f,\quad (-\Delta_X -\lambda^2)w=0
    \]
    % Since both $u$ and $R_0(\lambda)f$ satisfies the outgoing condition, 
    % we see that 
    % \begin{equation}
    %     \int_{\partial B(0,\rho)\times M} |G|^2=, G:=\frac{1}{2i\lambda}
    % \end{equation}
    For $\rho>0$, integrating over $B_{\mathbb{R}^n}(0,\rho)\times M$ and 
    applying Green's formula, we deduce
    \[
    \begin{aligned}
        0&=\int_{B(0,\rho)\times M} (w(-
        \Delta_{X}-\lambda^2)\bar{w}-(
            -\Delta_{X}-\lambda^2)w\bar{w})dxdy\\
        &=\int_{B(0,\rho)\times M} 
        (\bar{w}\Delta_{X}w-
        w\Delta_{X}\bar{w})dxdy\\
        &=
        \int_{\partial B(0,\rho)\times M} 
        (\partial_r w \bar{w}-w \partial_r \bar{w})
        dS dy\\
    \end{aligned}
    \]
    Using the outgoing condition for both $u$ and $R_0(\lambda)f$, we obtain 
    \[
    \begin{aligned}
        0&=\sum_{\sigma_k\leq |\lambda|}  \int_{\partial B(0,\rho)} (\partial_r w_k 
        \bar{w}_k-w_k \partial_r \bar{w}_k)dS
        +\mathcal{O}(e^{-\epsilon\rho})\\
        &=2i\sum_{\sigma_k\leq |\lambda|}\tau_k(\lambda)\int_{\partial B(0,\rho)} 
        |w_k|^2+\mathcal{O}(\rho^{-1})
    \end{aligned}
    \]
    Thus we have for each $|\sigma_k|<\lambda$
    \[
        \lim_{\rho\to \infty} \frac{1}{R}\int_{\partial B(0,R)} |w_k(x)|^2 dx=0, 
        \quad (-\Delta_X-\tau_k(\lambda)^2)w_k=0
    \]

    We now invoke \cite[Lemma 3.36]{MathematicalTheoryofScatteringResonances} 
    to obtain $w_k=0$ for each $|\sigma_k|<\lambda$, 
    and note the fact for each $|\sigma_k|<\lambda$
    \[
        (1-\chi)u_k=\mathcal{R}_0(\tau_k(\lambda))f,\quad f_k=[\Delta_{\mathbb{R}^n},\chi]u_k\in C_c^\infty(\mathbb{R}^n)
    \]
    Thus following the proof of in \cite[theorem 3.35]{MathematicalTheoryofScatteringResonances} 
    starting from Step 2, we can obtain the desired result.
\end{proof}

\begin{remark}\label{rmk:Eigenspace lies in the range of R0}
    Note for $\operatorname{Im} \lambda\geq 0$, 
    the resolvent identity holds 
    \[
        R_V(\lambda)=R_0(\lambda)(-V)R_V(\lambda)+R_0(\lambda)
    \]
    by analytically continuing $\lambda$ from the upper plane.
    Comparing the term with the highest 
    order in the Laurent expansion of $R_V$ by Lemma \ref{lem:Laurent expansion of RV near real line}, we see for any $\lambda_0\in \mathbb{R}$
    \[
        \Pi_{\lambda_0}=R_0(\lambda_0)(-V)\Pi_{\lambda_0}
    \]
    In particular we have $\operatorname{Ran}(\Pi_{\lambda_0})\subset 
    R_0(\lambda_0)(L^2_{\text{comp}})$. 
\end{remark}

This remark, and the fact that the function in the range of $R_0(\lambda_0)$ 
satisfies the outgoing condition, combined with the Rellich uniqueness theorem 
immediately implies the following:
\begin{corollary}
    If $\lambda\in \mathbb{R}$ and $u\in \Pi_{\lambda}$, 
    suppose $u=\sum_{k}u_k\otimes \varphi_k$ is the expansion with respect to $\varphi_k$, 
    then for each $\sigma_k<|\lambda|$, the support of $u_k$ lies in any open ball $B$ 
    so that $\operatorname{supp} V\subset B\times M$. 
\end{corollary}

% This seems that $\lambda$ is not need to be threshold, for 
% the Rellich uniqueness theorem can apply.

% \begin{remark}
%     A easy one 
% \end{remark}

\begin{proposition}[Boundary Pairing]\label{prop:boundary pairing}
    Suppose $u_l\in H^2_{\operatorname{loc}}(X),l=1,2$ satisfy 
    \[
    \begin{aligned}
        &(P_V-\lambda^2)u_l=F_l\in \mathscr{S}(X),\quad \lambda\in \mathbb{R}-\{0\}\\
        &u_l(r\theta,y)=r^{-\frac{n-1}{2}}\sum_{\sigma_k<|\lambda|} \left(e^{i\tau_k r}f_{l,k}(\theta)+
        e^{-i\tau_k r}g_{l,k}(\theta)\right)\otimes \varphi_k(y)+\mathcal{O}(r^{-\frac{n+1}{2}})
    \end{aligned}
    \]
    with $f_{l,k},g_{l,k}\in C^\infty(\mathbb{S}^{n-1})$, and the expansion is also valid for derivatives with respect to $rr$. Then 
    \[
        \sum_{\sigma_k<|\lambda|} (2i\tau_k(\lambda)) 
        \int_{\mathbb{S}^{n-1}} 
        (g_{1,k}\bar{g}_{2,k}-f_{1,k}\bar{f}_{2,k})d\omega
        =\int_{\mathbb{R}^n\times M} (F_1\bar{u}_2-u_1\bar{F}_2)
    \]
    even when $\lambda$ is a threshold.
\end{proposition}

\begin{proof}
    The proof is almost the same as the proof of Theorem \ref{thm:Rellich's uniqueness theorem}. 
    Integrating over $B_{\mathbb{R}^n}(0,\rho)\times M$ we obtain 
    \[
    \begin{aligned}
        \int_{\mathbb{R}^n\times M} F_1\bar{u}_2-u_1\bar{F}_2=& 
        \lim_{r\to \infty} \int_{B(0,r)\times M} (P_V-\lambda^2)u_1\bar{u}_2-
        u_1(P_V-\lambda^2)\bar{u}_2 dxdy\\
        =&\lim_{r\to \infty} \int_{B(0,r)\times M} (-\Delta_X u_1\bar{u}_2+
        u_1\Delta_X\bar{u}_2) dxdy\\
        =&\lim_{r\to \infty} \int_{\partial_B(0,r)\times M} 
        (-\partial_r u_1\bar{u}_2+
        u_1\partial_r\bar{u}_2) dxdy\\
        =&\lim_{r\to \infty} \sum_{\sigma_k<\lambda} 
        \int_{\mathbb{S}^{n-1}} 2i\tau_k(\lambda) \left(g_{1,k}(\theta)\bar{g}_{2,k}(\theta)-
        f_{1,k}(\theta)\bar{f}_{2,k}(\theta)\right) d\theta\\
        &+ \mathcal{O}(r^{-1})
    \end{aligned}
    \]
    which completes the proof.
\end{proof}
% \begin{remark}
%     In the region $\operatorname{Im} \lambda\geq 0$, 
%     if near $\sigma_k$ we have 
%     \[
%         R_V(\lambda)= -\frac{\Pi_{\sigma_k}}{\tau_k(\lambda)^2}+\frac{A_1}{\tau_k(\lambda)}+B(\lambda)
%     \]
%     Then near $-\sigma_k$ we have 
%     \[
%         R_V(\lambda)=R_V(-\bar{\lambda})^*=
%         -\frac{\Pi_{\sigma_k}}{\tau_k^2}-\frac{A_1^*}{\tau_k(\lambda)}+B(\lambda)^*
%     \]
%     since
%     $
%     \tau_k(\lambda)=-\overline{\tau_{k}(-\bar{\lambda})}
%     $.
% \end{remark}
\subsection{Resolvents near thresholds}

The following lemma will be used to 
characterize the resolvent near thresholds.
\begin{lemma}\label{lem:Induction formula for resonante state with expansion}
    Suppose $u(x,y)=\sum_{k} u_k(x)\varphi_k(y)\in H^2_{\operatorname{loc}}(X)$
    such that $(P_V-\lambda^2)u$ has compact support, satisfies 
    \[
        u_k(x)=e^{i\tau_k(\lambda) |x|}|x|^{-\frac{n-1}{2}}
        \sum_{j=-s_1}^{s_2} |x|^{-j} F_j^k(\frac{x}{|x|})+R^k(x)
    \]
    for some $s_1,s_2\geq 0$ and each $k$ with $\sigma_k<|\lambda|$, 
    where $F_j^k\in C^\infty(\mathbb{S}^{n-1})$ and $R^k$ 
    is smooth outside a compact set satisfying the estimate
    \[
        |\partial^\alpha R^k(x)|\leq C_\alpha|x|^{-s_2-1-\frac{n-1}{2}},\quad |x|\gg 1
    \]
    Then we have $F_{j}^k=0$ for every $j\leq {-1}$, and we have the following induction formula on $F_j^k$
    \begin{equation}\label{eq:3.7.12}
        F_{j+1}^k=\frac{1}{-2i \tau_k(\lambda)(j+1)}\left(-\Delta_{\mathbb{S}^{n-1}}+\frac{(n-1)(n-3)}{4}-j(j+1)\right)F_j^k
    \end{equation}
    for $0\leq j\leq s_2-2$.

    % If we assume in addition that $s_2\geq 0$ and $(P_V-\lambda^2)u=0$ then $u_k(x)$ has compact support.
\end{lemma}

\begin{proof}
    
    Writting the metric on Euclidean space via 
    \[
        g_{\mathbb{R}^n}=dr^2+r^{2}g_{\mathbb{S}^{n-1}}
    \]
    we see 
    \[
        \Delta_{\mathbb{R}^n}=\partial_r^2+\frac{n-1}{r}\partial_r +\frac{1}{r^2}\Delta_{\mathbb{S}^{n-1}}
    \]
    Using the fact 
    \[
        r^{\frac{n-1}{2}}\partial_r r^{-\frac{n-1}{2}}=\partial_r-\frac{n-1}{2r}
    \]
    we can compute
    \begin{equation}
        -r^{\frac{n-1}{2}}\Delta r^{-\frac{n-1}{2}}=-\partial_r^2+\frac{(n-1)(n-3)}{4r^2}-\frac{1}{r^2}\Delta_{\mathbb{S}^{n-1}}
    \end{equation}
    Using this formula, we directly compute for each  $j\in \mathbb{Z}$
    \[
    \begin{aligned}
        -&\Delta_{\mathbb{R}^n}\left(e^{i\tau_k(\lambda) r}r^{-\frac{n-1}{2}-j}
        F_j^k(\theta)\right)=r^{-\frac{n-1}{2}-j-2}e^{i\tau_k(\lambda) r}\\
        &
        \left(r^2\tau_k(\lambda)^2+r2i\tau_k(\lambda)j-
        j(j+1)+\frac{(n-1)(n-3)}{4}-\Delta_{\mathbb{S}^{n-1}}\right)F_j^k
    %    &-
    %    r^{-\frac{n-1}{2}}e^{i\tau_k(\lambda) r}r^{-j-2}\Delta_{\mathbb{S}^{n-1}}F_j^k
    \end{aligned}
    \]
    By the expansion of $u$ we have
    \[
        \left\langle (P_V-\lambda^2)u|\varphi_k \right\rangle_{L^2(M)}=(-\Delta_{\mathbb{R}^n}-\tau_k^2)u_k+\mathscr{E}'(\mathbb{R}^n)\in \mathscr{E}'(\mathbb{R}^n)
    \]
    And the previous computation shows that the leading term of $(-\Delta_{\mathbb{R}^n}-\tau_k^2)u_k
    $ is 
    \[
        2i\tau_k(\lambda)(-s_1)r^{-\frac{n-1}{2}+s_1-1}F^k_{s_1}(\theta)
    \]
    since all other terms are of $\mathcal{O}(r^{-\frac{n-1}{2}+s_1-2})$, 
    which implies that $F^k_{s_1}=0$. Thus inductively running $j$ from $-s_1$ to $-1$, 
    we see $F_j^k$ equals to zero for each $j\leq -1$, 
    by comparing the term $r^{j+1}$.
    And inductively running $j$ from $0$ to $s_2-2$, comparing the term $r^{j+2}$, we deduce the induction formula for $F_j$.
\end{proof}

The next proposition says that when $n\geq 5$, the first order term in the Laurent expansion 
of $R_V(\lambda)$ near thresholds is bounded $L^2\to L^2$, and actually it vanishes 
when $n\geq 7$, analogous to the the resolvent for potential scattering in Euclidean space $\mathbb{R}^{n}$ when $n\geq 5$.
See \cite[Theorem 6.2]{Jensen1980Spectral}.
\begin{proposition}\label{prop:Resolvent near threshold when n dayudengyu five}
    Suppose $n\geq 5$ the potential $V$ is real-valued.
    Then for $z$ near $\lambda_0=\pm\tau_k$ in $\hat{\mathcal{Z}}$, the Laurent expansion of $R_V(z)$ 
    \[
        R_V(z)=-\frac{\Pi_{\lambda_0}}{\tau_k(z)^2}+\frac{A_1}{\tau_k(z)}+B(z)
    \]
    satisfies $A_1$ is a bounded finite rank operator $L^2\to L^2$. When $n\geq 7$, we have $A_1=0$.
\end{proposition}

\begin{proof}
    We first show that $A_1:L^2_{\operatorname{comp}}\to \operatorname{Ran}(\Pi_{\lambda_0})$. It suffices to show $A_1(L^2_{\operatorname{comp}})
    \subset L^2$. Recall by \eqref{eq:expression of RV} and remark \ref{rmk:2.2.15}
    \[
    \begin{aligned}
       &R_V(z)=R_0(z)(I+VR_0(z)\rho)^{-1}(I-VR_0(z)(1-\rho))\\
        &(I+VR_0(z)\rho)^{-1}=I-VR_V(z)\rho
    \end{aligned}
    \]
    Hence the Laurent expansion for $(I+VR_0(z)\rho)^{-1}$ near $\lambda_0$ has order at most two 
    since $R_V$ does, we have
    \begin{equation}\label{eq:Laurent expansion of inverse of IVR0z}
        (I+VR_0(z)\rho)^{-1}(I-VR_0(z)(1-\rho))=\frac{\tilde{A}_2}{
            \zeta^2
        }+\frac{\tilde{A}_1}{\zeta}+\tilde{B}(\zeta)
    \end{equation}
    where $z$ near $\lambda_0$ satisfies $\tau_k(z)=\zeta$, 
    thus $\zeta$ is a local conformal coordinate near $\lambda_0$, 
    and $\tilde{A}_2,\tilde{A}_1:L^2_{\operatorname{comp}}\to L^2_{\operatorname{comp}}$ are both finite-rank operators.
    We define $\tilde{R}_0(\zeta):=R_0(z(\zeta))$ 
    and $\tilde{R}_V(\zeta):=R_V(z(\zeta))$, thus 
    $\tilde{R}_0$ and $\tilde{R}_V$ are holomorphic near zero.
    Then we have 
    \[
        A_1(L^2_{\text{comp}})=(\tilde{R}_0(0)\tilde{A}_1-\partial_{\zeta}
        \tilde{R}_0(0)\tilde{A}_2)(L^2_{\text{comp}})
    \]
    The expression \eqref{eq:Expression of R0 lambda} of free resolvent $R_0$ implies
    \[ 
    \begin{aligned}
        \tilde{R}_0(\zeta)u=&\sum_{\sigma_k=|\lambda_0|}R_0^{\mathbb{R}^n}(\zeta)u_k\otimes \varphi_k\\
        &+
        \sum_{\sigma_k<|\lambda_0|} R_0^{\mathbb{R}^n}(\tau_k(z(\zeta)))u_k\otimes \varphi_k+
        \sum_{\sigma_k>|\lambda_0|} R_0^{\mathbb{R}^n}(\tau_k(z(\zeta)))u_k\otimes \varphi_k
    \end{aligned}
    \]
    Let $v\in A_1(L^2_{\operatorname{comp}})$, thus $v=A_1u$ for some $u\in L^2_{\operatorname{comp}}$. 
    We next show that $v$ satisfies the outgoing condition and $v\in L^2(X)$.
    \begin{itemize}
        \item  By the asymptotic of $R_0^{\mathbb{R}^n}(\lambda)$ in proposition \ref{thm:3.5}
        for $\lambda\in \mathbb{R}-0$, for each $\sigma_k<|\lambda|$, 
        we have an asymptotic expansion of $v_k$ as in 
        Lemma \ref{lem:Induction formula for resonante state with expansion}
        , starting from $s_1=-1$ in its notation. 
        \item For those $k$ with $\sigma_k=|\lambda|$, 
        we know
        \[
            v_k\in R_0^{\mathbb{R}^n}(0)(L^2_{\operatorname{comp}})+\partial_{\zeta}R_0^{\mathbb{R}^n}(0)
            (L^2_{\operatorname{comp}})
        \]
        And we recall the explicit expression of $R_0^{\mathbb{R}^n}$ in proposition \ref{thm:3.3} 
        \[
            R_0^{\mathbb{R}^n}(0,x,y)=\frac{1}{|x-y|^{n-2}}P_n(0), 
            \quad \partial_{\zeta}R_0^{\mathbb{R}^n}(0,x,y)=\frac{1}{|x-y|^{n-3}}(iP_n(0)+P_n'(0))
        \]
        Hence when $n\geq 7$, we have $v_k(x)=\mathcal{O}(|x|^{3-n})$ for large $|x|$, 
        while when $n=5$ we have $v_k(x)=\mathcal{O}(|x|^{-3})$ for large $|x|$ since $iP_n(0)+P_n'(0)=0$ by the explicit expression. 
        And it's easy to see that 
        \[
            ||\partial_r v_k(x)||=\mathcal{O}(|x|^{2-n})
        \]
        \item For those $k$ with $\sigma_k>|\lambda|$, we can use the method in the proof of Lemma \ref{lem:Range of R0 is outgoing} 
        to show that they have exponential decay.
    \end{itemize}
    Thus we know $v$ satisfies the outgoing condition \eqref{eq:Sommerfield radiation condition}.
    The Rellich uniqueness theorem \ref{thm:Rellich's uniqueness theorem} then 
    implies that $v_k$ has compact support for each $\sigma_k<|\lambda|$ since $(P_V-\lambda_0^2)A_1=0$, also we know $v\in L^2(X)$.

    We note that $A_1$ is a finite-rank operator 
    continuous from $L^2_{\operatorname{comp}}(X)$ to $L^2(X)$, so $A_1$ 
    is of the form 
    \[
        A_1=\sum_{j=1}^J u_j\otimes v_j
    \]
    for some $u_1\cdots u_J\in L^2(X)$ linearly independent, and $v_j\in L^2_{\operatorname{loc}}(X)$ 
    since the dual of $L^2_{\operatorname{comp}}$ is $L^2_{\operatorname{loc}}$.
    On the other hand, we note that for $z$ near $\lambda=-\sigma_k$ in the physical region, we have 
    \[
        R_V(z)=(R_V(\bar{z})^*)=-\frac{\Pi_{\lambda_0}}{\tau_k(z)^2}-\frac{A_1^*}{\tau_k(z)}
        +B^*(\bar{z})
    \]
    So the same argument as above shows that $A_1^*:L^2_{\text{comp}}(X)\to L^2(X)$. However we see 
    \[
        A_1^*=\sum_{j=1}^J \bar{v}_j\otimes \bar{u}_j
    \]
    Since $\bar{u}_1\cdots \bar{u}_j$ are linearly independent viewed as elements in the dual space of $L^2_{\operatorname{comp}}$, 
    we know the $A_1^*(L^2_{\operatorname{comp}})$ is actually the span of $\bar{v}_j$, so we know $v_j\in L^2(X)$. This shows that $A_1$ 
    is actually $L^2$ to $L^2$.

    Next we show $A_1=0$ when $n\geq 7$. Composing $\Pi_{\lambda_0}$ on the left of $R_V(\zeta)$ 
    for $z(\zeta)$ lying in the physical region, using the Laurent expansion we have 
    \[
        \Pi_{\lambda_0}\tilde{R}_V(\zeta)=-\frac{\Pi_{\lambda_0}}{\zeta^2}=-\frac{\Pi_{\lambda_0}}{\zeta^2}
        +\frac{A_1}{\zeta}+\Pi_{\zeta}B(\zeta)
    \]
    which implies that $A_1+\zeta \Pi_{\zeta}B(\zeta)=0$ for $z(\zeta)$ in the physical region. We must 
    remark here that at present we do \textbf{NOT} know $\Pi_{\zeta}B(\zeta)$ tends to zero as $\zeta\to 0$, 
    since we only know $B:L^2_{\text{comp}}\to L^2_{\text{loc}}$ continuously and $\Pi_{\lambda}$ 
    is defined only on the $L^2$ space. To make this argument rigorous, 
    we need to view $B$ taking values in some weighted $L^2$ space. 
    Actually, we know
    \begin{itemize}
        \item $\Pi_{\zeta}$ maps $L^2_{\text{loc}}(\mathbb{R}^n,\mathbb{C}\phi_k)\subset L^2(X)$ for $\sigma_k<|\lambda_0|$ 
    continuously to $L^2(X)$ since if $u=\sum_{j} u_j\otimes \varphi_j\in \operatorname{ran} \Pi_{\lambda_0}$ then 
    $u_k$ is compactly supported.
        \item And we also note that $\tilde{R}_0(\zeta)$ is continuous at zero 
    as a map from $L^2(\mathbb{R},\oplus_{\sigma_k>|\lambda_0|}\mathbb{C}\phi_k)$ to $L^2(X)$.
    \end{itemize}
    So it remains only to analyze those $k$ with $\sigma_k=\lambda_0$. We recall 
    \eqref{eq:Laurent expansion of inverse of IVR0z} and the fact that $R_V$ is the composition of $R_0$ and \eqref{eq:Laurent expansion of inverse of IVR0z}, 
    if we write 
    \[
        \tilde{R}_0(\zeta)=\tilde{R}_0(0)+\partial_{\zeta}\tilde{R}_0(0)\zeta+\zeta^2 Q(\zeta)
    \]
    for some holomorphic operator $Q(\zeta):L^2_{\operatorname{comp}}\to L^2_{\operatorname{loc}}$, then it follows that 
    \begin{equation}\label{eq:Expression of Bzeta}
        B(\zeta)=
        Q(\zeta)\tilde{A}_2+\partial_{\zeta}\tilde{R}_0(\zeta)\tilde{A}_1+(\tilde{R}_0(0)+\partial_{\zeta}\tilde{R}_0\zeta)\tilde{B}(\zeta)+\zeta Q(\zeta)\tilde{A}_1+\zeta^2Q(\zeta)\tilde{B}(\zeta)
    \end{equation}
    For $u\in L^2_{\operatorname{comp}}$ and $|x|\gg 1$, using the explicit 
    expression of $\tilde{R}_0(\zeta)$ given in Theorem \ref{thm:3.3}, we apply Taylor's expansion with integral remainder to obtain
    \[
        \left|\left(Q(\zeta)(u)\right)_k(x)\right|\leq C_n(u) |\zeta| \int_{0}^1 \sum_{j=0}^{\frac{n-3}{2}} \frac{e^{-t\operatorname{Im}(\zeta)|x|/2}}{|x|^{n-4}}(|t\zeta| |x|)^{j} dt
    \]
    for some constant $C$ depending on $u\in L^2_{\operatorname{comp}}$.
    And for $v\in \operatorname{Ran} \Pi_{\lambda_0}$, we know $v_k(x)=\mathcal{O}(|x|^{2-n})$ for $|x|$ large 
    since it lies in $\operatorname{R}_0^{\mathbb{R}^n}(0)(L^2_{\operatorname{comp}})$, so 
    for some constants $C$ depending on $u,v$ and 
    for those $\zeta$ with $|\zeta|\sim \operatorname{Im} \zeta$ we estimate
    \[
    \begin{aligned}
        |\zeta \langle v_k,(Q(\zeta)(u))_k\rangle|&\leq C|\zeta|+
        C|\zeta|^2 \int_{|x|\geq 1} \frac{1}{|x|^{2-n}} \int_{0}^1 \sum_{j=0}^{\frac{n-3}{2}} \frac{e^{-t\operatorname{Im}(\zeta)|x|/2}}{|x|^{n-4}}(|t\zeta| |x|)^{j} 
        dx dt\\
        &\leq C\zeta+C\zeta^2 \int_{0}^1 \int_{1}^{+\infty} \sum_{j=0}^{\frac{n-3}{2}}e^{-t|\zeta| r} r^{5-n+j} (|t\zeta|)^{j} drdt\\
        &\leq C\zeta+C\zeta^2 \int_{0}^1 \int_{t|\zeta|}^{+\infty} \sum_{j\geq n-6}^{\frac{n-3}{2}}e^{-s} s^{5-n+j}(|t\zeta|)^{j-1+n-5-j} ds  dt\\
        &\leq C\zeta+C\zeta^2 \int_{0}^1 (-\ln t-\ln \zeta) dt
    \end{aligned}
    \]
    for $n\geq 7$.
    This inequality and the expression \eqref{eq:Expression of Bzeta} of $B(\zeta)$, shows that $\zeta \langle v_k,(B(\zeta)u)_k\rangle$ tends to 
    zero for $\sigma_k=\lambda_0$ and $u\in L^2_{\operatorname{comp}}$ as $\zeta$ tends to zero along 
    the line that $|\zeta|\sim \operatorname{Im} \zeta$. Together with the argument for $\sigma_k>\lambda_0$ 
    and $\sigma_{k}<\lambda_0$, we know $A_1=0$ for $n\geq 7$. 
    
    % And note for any $\psi\in L^2_{\operatorname{comp}}$, we know $\Pi_{\lambda_0}A_1\psi=\psi$. So we have 
    % \[
    %     0=\Pi_{\lambda_0}C(\zeta)/\zeta=A_1\psi+\zeta \Pi_{\lambda_0}B(\zeta)\psi
    % \]
    % We know $\zeta (\zeta)\psi$ converges weakly in $L^2$ to zero as $\zeta$ lying in the physical space 
    % goes to zero, thus we know $A_1\psi=0$.
\end{proof}

Define the space 
\[
\tilde{H}_{\pm \sigma_k}=\{u\in H^2_{\operatorname{loc}}(X):(P_V-\sigma_k^2)u=0,
u=R_0(\pm \sigma_k)(-Vu)\}
\]
Then by Remark \ref{rmk:Eigenspace lies in the range of R0} we know 
the range of $\Pi_{\sigma_k}$ is a subspace of $\tilde{H}_{\sigma_k}$. 
The next proposition shows that when $n=3$ the range of the first singular term 
of the Laurent expansion near $\sigma_k$ also lies in $\tilde{H}_{\sigma_k}$.

\begin{proposition}\label{prop:Resolvent near threshold when n dengyu three}
    Assume that $n=3$.
    \begin{itemize}
        \item If $v\in \operatorname{Ran} \Pi_{\lambda_0}$, 
        then $v=R_0(\sigma_{k_0})f$ where $f=-Vv=-\Delta_Xv\in L^2_{\operatorname{comp}}(X)$
        and for any $k$ with $\sigma_k=\sigma_{k_0}$ we have 
        \[
            \int_{\mathbb{R}^3} \langle f(x,\bullet),\varphi_k\rangle_{L^2(M)}dx=0
        \]
        and thus $v_k(x)=\mathcal{O}(\langle x\rangle^{-2})$ when $|x|\gg 1$.
        \item For $z$ naer $\lambda_0=\tau_k$ in $\hat{\mathcal{Z}}$, the Laurent expansion of $R_V(z)$
        has the form 
        \[
            R_V(z)=-\frac{\Pi_{\lambda_0}}{\tau_k(z)^2}+\frac{A_1}{\tau_k(z)}+B(z)
        \]
        The range of $A_1$ lies in $\tilde{H}_{\sigma_k}$.
        \item If $u\in \tilde{H}_{\sigma_k}$, and $u$ has the expansion with respect to $\varphi_k$ 
        \[
            u(x,y)=\sum_{j=0}^\infty u_j(x)\otimes \varphi_j(y):=\sum_{\sigma_j\leq \sigma_{k_0}} u_j(x)\otimes \varphi_j(y)+R(x,y)
        \]
        Then $u_j$ is of compact support for $\sigma_j<\sigma_k$ and $R$ is in $L^2(X)$.
    \end{itemize}
\end{proposition}

\begin{proof}
    For the first part, we see 
    \[
        u=R_0(\lambda)(-\Delta_X-\lambda^2)u=R_0(\lambda)(-Vu)
    \]
    This shows in particular that if we set $f_k(x)=\langle (-Vu)(x,\bullet),\varphi_k\rangle_{L^2(M)}$ then
    \[
        u_k=R_{0}^{\mathbb{R}^{3}}(\lambda)(f_k)\in H^2(\mathbb{R}^3)
    \]
    Now we can argue as the proof of \cite[Lemma 3.18]{MathematicalTheoryofScatteringResonances}. We recall that 
    $R_{0}^{\mathbb{R}^{3}}(0)(x,y)=\frac{1}{4\pi|x-y|}$ given in Proposition \ref{thm:3.3}, and we can write 
    \[
        \begin{aligned}
            u_k(y+r\theta)&=\frac{1}{4\pi}\int_{\mathbb{R}^3}\frac{f_k(x)}{|x-y-r\theta|}dx=\frac{1}{4\pi r}\int_{\mathbb{R}^3}\frac{f_k(x)}{|\theta-r^{-1}(x-y)|}dx\\
            &=\frac{1}{4\pi r}\int_{\mathbb{R}^3} f_k(x)\left(1-2r^{-1}\langle\theta,x-y \rangle+r^{-2}|x-y|^2\right)^{-1/2} dx
        \end{aligned}
    \]
    By taylor's expansion $(1+s)^{-1/2}=1-\frac{1}{2}s+\mathcal{O}(s^2)$ we set $y=0$ to obatin
    \[
        u_k(r\theta)=\frac{1}{4\pi r} \int_{\mathbb{R}^3} f_k+\mathcal{O}(r^{-2})
    \]
    Since $u_k\in L^2$ we must have $\int_{\mathbb{R}^3} f_k=0$, which is exactly
    \[
        \int_{\mathbb{R}^3} \langle f(x,\bullet),\varphi_k\rangle_{L^2(M)}dx=0
    \]
    as desired.

    For the second part, we can write by remark \ref{rmk:Eigenspace lies in the range of R0}
    \[
        A_1=\tilde{R}_0(0)\rho(-\Delta-\sigma_k^2)A_1+\partial_{\zeta}\tilde{R}_0(0)\rho(-\Delta-\sigma_k^2)(-\Pi_{\sigma_{k_0}})
    \]
    And note that if $u\in \operatorname{Ran} \Pi_{\sigma_{k_0}}$, 
    then $u=R_0(\sigma_k)f$ for 
    $f=-Vu\in L^2_{\operatorname{comp}}(X)$ and $\int_{\mathbb{R}^3} f_k(x)dx=0$ 
    if $\sigma_k=\sigma_{k_0}$. This shows that $\partial_{\zeta}\tilde{R}_0(0)\rho(-\Delta-\sigma_k^2)(-\Pi_{\sigma_{k_0}})$ 
    has zero $\varphi_k$ coefficient in the Fourier expansion, 
    since $\partial_{\zeta}R_0^{\mathbb{R}^3}|_{\zeta=0}$ is exactly the integration over $\mathbb{R}^3$.
    Thus using Lemma \ref{lem:Induction formula for resonante state with expansion} 
    as in the argument $n=5$, we know if $v\in \operatorname{Ran} A_1$ then $v$ satisfies the outgoing condition.
    By Rellich uniqueness theorem, we know $v_k$ is compactly supported for $\sigma_k<\sigma_{k_0}$, 
    and $v_k$ is in $L^2(X)$ for $\sigma_k>\sigma_{k_0}$ as usual. This shows that $v\in \tilde{H}_{\sigma_{k_0}}$.

    The third part follows from the second part and the Rellich uniqueness theroem.
\end{proof} 

By considering $-A_1^*$ as the first singular term in 
the Laurent expansion near $\lambda=-\sigma_{k_0}$, we can then write $A_1$ as 
\[
    A_1=\sum_{j=1}^J u_j\otimes v_j
\]
where $u_j,v_j$ are both elements in $\tilde{H}_{\sigma_j}$, and the range of $A_1$ 
is exactly the span of $\{u_1\cdots u_J\}$. Note that this expression is not canonical. 
Furthermore, we can write
\[
\begin{aligned}
    &u_j(x,y)=\sum_{k=0}^\infty u_{jk}(x)\otimes \varphi_k(y),\quad v_j(x,y)=\sum_{k=0}^\infty v_{jk}(x)\otimes \varphi_k(y)\\
    &u_{jk}(x)=c_{jk}\frac{1}{-4\pi |x|}+\mathcal{O}(|x|^{-2}), v_{jk}(x)=d_{jk}\frac{1}{-4\pi |x|}+\mathcal{O}(|x|^{-2}), \quad |x|\gg 1, \sigma_k=\sigma_{k_0}
\end{aligned}
\]
for some constants $c_{jk},d_{jk}\in \mathbb{C}$.
We shall define a multiplicity $\tilde{m}_V(\sigma_k)$ as 
\begin{equation}\label{eq:definition of tildemV}
    \tilde{m}_V(\sigma_k):=\sum_{j=1}^J \frac{c_{jk}d_{jk}}{4\pi i}
\end{equation}
It can be verified directly that the definition of $\tilde{m}_V(\sigma_k)$ is independent of the choice of $u_j,v_j$. 
Moreover, we will show in the proof of Birman Krein trace formula that,
$\tilde{m}_V(\sigma_k)$ is in fact real-valued. 
% (The author believes it should be directly proved that 
% this is a non-negative integer, but he can not prove it).
When $n\geq 5$, we will set $\tilde{m}_V(\sigma_k):=0$.

The following proposition shows that, the wave plane 
$e^{-i\tau_k(\lambda)\langle \bullet,\omega\rangle}\otimes \varphi_k$ 
is, in some sense, orthogonal to $\tilde{H}_{\sigma_k}$. This proposition will be used later to analyze the regularity of scattering matrix near poles of $R_V$.
\begin{proposition}\label{prop:Projection onto eigenspace of free wave}
    If $\lambda\in \mathbb{R}$,
    and $\sigma_k<|\lambda|$, $u\in \tilde{H}_{\pm \sigma_k}$
    Then 
    \[
        \langle Ve^{-i\tau_k(\lambda)\langle \bullet,\omega\rangle}\otimes \varphi_k,u \rangle_{L^2(X)}=0
    \]
    If we assume in addition that $n=3$, then this holds even when $\sigma_k=|\lambda|$ and $u\in L^2$.
\end{proposition}

\begin{proof}
    We know $u_k$ is compactly supported for $\sigma_k<|\lambda|$.
    Hence we have
    \[
    \begin{aligned}
        \left(Ve^{-i\tau_k(\lambda)\langle \bullet,\omega\rangle}\otimes \varphi_k|u\right)_{L^2(X)}&=
        \left(e^{-i\tau_k(\lambda)\langle \bullet,\omega\rangle}\otimes \varphi_k|Vu\right)_{L^2(X)}\\
        &=
        \left(e^{-i\tau_k(\lambda)\langle \bullet,\omega\rangle}\otimes \varphi_k|(-\Delta_X-\lambda^2)u\right)_{L^2(X)}\\
        &=\left(e^{-i\tau_k(\lambda)\langle \bullet,\omega\rangle}\otimes \varphi_k|(-\Delta_X-\tau_k(\lambda)^2)u_k\right)_{L^2(X)}\\
        &=\left((-\Delta_{\mathbb{R}^n}-\tau_k(\lambda)^2)e^{-i\tau_k(\lambda)\langle \bullet,\omega\rangle}|u_k\right)_{L^2(\mathbb{R}^n)}=0
    \end{aligned}
    \]
    The last step uses the fact that $u_k$ has good control at infinity. When $n=3$ and $\sigma_k=|\lambda|$, $u\in L^2$, we 
    know $u\in \operatorname{Ran} \Pi_{\lambda}$ so this is a restatement of the first part of Proposition \ref{prop:Resolvent near threshold when n dengyu three}.
\end{proof}

\section{Scattering Matrix}\label{sec:Scattering Matrix}

The scattering matrix is the operator mapping the \textit{incoming data} to 
the \textit{outgoing data} in classical scattering theory. In our setting, we can also define 
the scattering matrix by imitating its definition in Euclidean space. Actually our definition is 
essentially the same as that in \cite[Chapter 3.7]{MathematicalTheoryofScatteringResonances}.

For each $k\in \mathbb{Z}_{\geq 0}$ and $\lambda\in \mathbb{R}$, with $|\lambda|>\sigma_k$ so that $\lambda$ is neither a pole of $R_V$ 
nor a threshold, and for each $\omega\in \mathbb{S}^{n-1}$, we define $e(x,y;\lambda,\omega;k)$ and $u(x,y;\lambda,\omega;k)$ where $(x,y)\in \mathbb{R}^n\times M$ as
\begin{equation}\label{eq:3.7.1}
    \begin{aligned}
    e(x,y;\lambda,\omega;k):=&e^{-i\tau_k(\lambda) \langle x,\omega\rangle}\otimes \varphi_k(y)+u(x,y;\lambda,\omega;k)\\
    u(x,y;\lambda,\omega;k):=&-R_V(\lambda)(Ve^{-i\tau_k(\lambda) \langle \bullet,\omega\rangle}\otimes \varphi_k)
    \end{aligned}
\end{equation}
so that $(P_V-\lambda^2)e=0$. We remark that $e$ should be viewed as a modified plane wave, 
namely, $e^{-i\lambda\langle x,\omega\rangle}$, distorted by 
the potential $V$.
Now since 
\[
    u(\cdot,\cdot;\lambda,\omega,k)=R_0(\lambda)(I+VR_0(\lambda)\rho)^{-1}(Ve^{-i\tau_k(\lambda) \langle \bullet,\omega\rangle}\otimes \varphi_k)
\]
and our assumption that $\lambda$ is not a pole of $R_V$ means precisely that the term 
\[
    (I+VR_0(\lambda)\rho)^{-1}(Ve^{-i\tau_k(\lambda) \langle \bullet,\omega\rangle}\otimes \varphi_k)
\]
can be defined, so we can use the asymptotic behaviour of $R_0(\lambda)$, as in the proof of 
Lemma \ref{lem:Range of R0 is outgoing}, to analyze the behaviour of $u$ as $|x|\to +\infty$. 
Moreover, the asymptotic behaviour of $e^{-i\tau_k(\lambda) \langle x,\omega\rangle}$ as $|x|\to +\infty$ is given by 
Proposition \ref{prop:3.38}, so in the sense of distribution in $\theta\in \mathbb{S}^{n-1}$, 
we know as $r\to \infty$ 
\[
    \begin{aligned}
    e(r\theta,y;\lambda,\omega;k)
    \sim &\;c_n^+(\tau_k(\lambda)r)^{-\frac{n-1}{2}}
     \left(e^{-i \tau_k(\lambda)r}\delta_{\omega}(\theta)+
    e^{i\tau_k(\lambda)r}i^{1-n}\delta_{-\omega}(\theta)\right)\otimes \varphi_k(y)\\
    &+c_n^+(\tau_k(\lambda)r)^{-\frac{n-1}{2}}
    \sum_{\sigma_j<|\lambda|} 
    e^{i\tau_j(\lambda)r}b(\theta;\lambda,\omega;j,k)\otimes \varphi_j(y)
    \end{aligned}
\]
where the constant 
\[
    c_n^+=(2\pi)^\frac{n-1}{2}e^{ i\frac{\pi}{4}(n-1)}
\]
Here $b(\theta;\lambda,\omega;j,k)$ is the leading part of $\langle u,\varphi_j\rangle_{L^2(M)}$ as $r\to \infty$
\[
\begin{aligned}
    u(r\theta,y;\lambda,\omega;k)=&c_n^+(\tau_k(\lambda)r)^{-\frac{n-1}{2}}\sum_{\sigma_j<|\lambda|} 
    \left(e^{i\tau_j(\lambda)r}b(\theta;\lambda,\omega;j,k)\otimes \varphi_j(y)
    +\mathcal{O}(r^{-1})\right)\\
    & +\mathcal{O}(e^{-\epsilon(\lambda)r})
\end{aligned}
\]
where $\epsilon(\lambda)$ is a positive constant depending on $\lambda$, 
as in the proof of Lemma \ref{lem:Range of R0 is outgoing}.
The \textbf{absolute scattering matrx} $S_{\operatorname{abs},k}(\lambda)$, 
defined for $\sigma_k<|\lambda|$, maps
\[
\begin{aligned}
    &S_{\operatorname{abs},k}(\lambda):
    \delta_\omega(\theta) \mapsto
    &i^{1-n}\delta_{-\omega}(\theta)\otimes \varphi_k
    +\sum_{\sigma_j<|\lambda|} b(\theta;\lambda,\omega;j,k)\otimes \varphi_j
\end{aligned}
\]
We denote by $S_{\operatorname{abs},jk}(\lambda)$ the Fourier coefficient of $\varphi_j$ for each $\sigma_j<|\lambda|$. 
Thus 
\[
\begin{aligned}
    S_{\operatorname{abs},jk}(\lambda)&:C^\infty(\mathbb{S}_{\theta}^{n-1})\to 
    \mathscr{D}'(\mathbb{S}_{\theta}^{n-1})\\
    &\delta_\omega(\theta) \mapsto \delta_{j}^k i^{1-n}\delta_{-\omega}(\theta)
    + b(\theta;\lambda,\omega;j,k)
\end{aligned}
\]
Note when $V=0$ the 
absolute scattering matrix is defined as 
\[
    S_{\operatorname{abs},jk,V=0}
    (\lambda)f(\theta)=\delta_{j}^k i^{1-n}f(-\theta)
\]
Thus we define the \textbf{scattering matrix $S_{jk}(\lambda)$ with 
index $jk$}, sometimes simply referred as the \textbf{scattering matrix} when there is no ambiguity, by
\[
S_{jk}(\lambda):=i^{n-1}S_{\operatorname{abs},jk}(\lambda)J
\]
where $Jf(\theta):=f(-\theta)$.

Notice that we have now defined the scattering matrix $S_{jk}(\lambda)$ when 
\begin{itemize}
    \item $\lambda\in \mathbb{R}$ and $\lambda>\max(\sigma_j,\sigma_k)$
    \item $\lambda$ is neither a pole of $R_V$ nor a threshold.
\end{itemize}
The next proposition provides a definition of the scattering matrix 
as a meromorphic family of operators in $\hat{\mathcal{Z}}$, 
for any $j,k\in \mathbb{Z}_{\geq 0}$.
% We Do NOT require $V$ to be real-valued now.

% \begin{remark}
%     We can define the Normalized scattering matrix to make it unitary, in contrast to the 
%     absolute scattering matrix.
% %     We define the interwine operator $J$ and normalization $\mathcal{N}_K(\lambda)$ operator via 
% % \[
% % Jf(\theta,y)=f(-\theta,y),\quad \mathcal{N}_K(\lambda)(\sum_{k\leq K} f_k\otimes \varphi_k)=\sum_{k\leq K}
% % \frac{1}{\sqrt{\tau_k(\lambda)}} f_k\otimes \varphi_k 
% % \]

% % And we define the $K-$th scattering matrix 
% % \[
% %     S_{K}(\lambda):=
% %     i^{n-1}\mathcal{N}_K(\lambda)^{-1}S_{\operatorname{abs},K}(\lambda)\mathcal{N}_K(\lambda)J
% % \]
% \end{remark}

\begin{proposition}[Description of the Scattering matrix]
    \label{prop:Description of the Scattering matrix}
    The scattering matrix $S_{jk}$ defines an operator 
    \[
    S_{jk}(z)=\delta_{j}^k I+A_{jk}(z):L^2(\mathbb{S}^{n-1})
    \to L^2(\mathbb{S}^{n-1})
    \]
    where $A_{jk}(z):\mathscr{D}'(\mathbb{S}^{n-1})\to C^\infty(\mathbb{S}^{n-1})$ 
    is meromorphic for $z\in \hat{\mathcal{Z}}$, which 
    is given by 
    \[
    A_{jk}(z)=a_n\tau_j(z)^{\frac{n-3}{2}}\tau_k(z)^{\frac{n-1}{2}}
    E_{\rho,j}(z)(I+VR_0(z)\rho)^{-1}V  
    \tilde{E}_{\rho,l}(z)
    \]
    where $E_{\rho,l}(z):L^2(X)
    \to L^2(\mathbb{S}^{n-1})$ is defined by the Schwartz kernel
    \[
        E_{\rho,l}(z)(\omega,x,y):=\rho(x)e^{-i\tau_l(z)\langle x,\omega\rangle}\otimes 
        \varphi_l(y)
    \]
    and $\tilde{E}_{\rho,l}(z):L^2(\mathbb{S}^{n-1})\to L^2(X)$ is defined by the Schwartz kernel
    \[
        \tilde{E}_{\rho,l}(z)(x,y,\omega):=
        \rho(x)e^{i\tau_l(z)\langle x,\omega\rangle}\otimes 
        \varphi_l(y)
    \]
    Here the constant $a_n=(2\pi)^{1-n}/2i$, and $\rho\in C_c^\infty(\mathbb{R}^n)$ 
    equals to one in a neighborhood of $\operatorname{supp} V$.
\end{proposition}

\begin{proof}
    We only need to check that, this description coincides with the preceding definition when 
    $z$ is parametrized by $z=\lambda^2$, $\lambda\in \mathbb{R}$ 
    which is neither a pole of $R_V$ nor a threshold, and satisfies $\lambda>\max(\sigma_j,\sigma_k)$.
    By definition, for fixed $k\in \mathbb{N}_0,\lambda\in \mathbb{R},\omega\in \mathbb{S}^{n-1}$ with $|\lambda|>\sigma_k$, 
    we know $u=R_0(\lambda)f$ where 
    \[
    f=-(I+VR_0(\lambda)\rho)^{-1}(Ve^{-i\tau_k \langle 
        \bullet,\omega
    \rangle}\otimes \varphi_k)\]
    We can write $f$ as an expansion in terms of $\varphi_l$
    \[
        f(x,y)=\sum_{l\geq 0} f_l(x)\otimes \varphi_l(y)
    \]
    and we will from now on use $\tau_j$ to denote $\tau_j(\lambda)$ for notational simplicity.
    According to the expansion of $R_0(\lambda)$ and the 
    asymptotic behaviour of $R_0^{\mathbb{R}^n}(\lambda)$ given in Proposition \ref{thm:3.5}, we see 
    \[
    u(r\theta,y;\lambda,\omega;k)=
    \sum_{\sigma_l<|\lambda|} e^{i\tau_lr}r^{-\frac{n-1}{2}}\frac{1}{4\pi}
    \left(\frac{\tau_l}{2\pi i}\right)^{\frac{1}{2}(n-3)}\mathcal{F}(f_l)(\tau_l \theta)\otimes\varphi_l(y)
    +\mathcal{O}(r^{-\frac{n+1}{2}})
    \]
    where the part $\sigma_l>|\lambda|$ can be dealt with as in the proof of Lemma \ref{lem:Range of R0 is outgoing}. 
    Thus the function $b(\theta;\lambda,\omega;j,k)$ 
    is given by 
    \[
    \begin{aligned}
        b(\theta;\lambda,\omega;j,k)&=\frac{\tau_k^{\frac{n-1}{2}}}{4\pi c_n^+}
        \left(\frac{\tau_j}{2\pi i}\right)^{\frac{1}{2}(n-3)}\mathcal{F}(f_j)(\tau_j \theta)\\
        &=-\frac{\tau_j^{\frac{n-3}{2}}\tau_k^{\frac{n-1}{2}}}{2(2\pi)^{n-1}}i^{2-n}
        \int_{\mathbb{R}^n\times M} 
        e^{-i\tau_j \langle x,\theta \rangle}\\
        &\left((I+VR_0(\lambda)\rho)^{-1}(Ve^{-i\tau_k \langle 
            \bullet,\omega
        \rangle}\otimes \varphi_k)\right)(x,y)\varphi_j(y) dxdy
    \end{aligned}
    \]
    Taking into accout $i^{n-1}$ and $J$, we see 
    $\delta_\omega\in \mathscr{D}'(\mathbb{S}^{n-1})$ is mapped to 
    $\delta_{j}^k\delta_\omega+A_{jk}(\delta_\omega)$ via $S_{jk}$, this completes the proof. 
\end{proof}

% Define 
% \[
%     \operatorname{Eigen}(\lambda):=\sum_{|\sigma_k|<|\lambda|} \operatorname{Span}_{\mathbb{C}}(\varphi_k)
% \]

% \begin{remark}
%     Recall 
%     \[
%         R_V(\lambda)=R_0(\lambda)(I+VR_0(\lambda)\rho)^{-1}(I-VR_0(\lambda)\rho)
%     \]
%     Note $|R_0(\lambda,x,y)|\lesssim e^{ (\operatorname{Im}\lambda)_{-}|x-y|}|x-y|^{-1}$
%     and the other two terms will not change the property at $|x|\to \infty$.
%     So away from poles, $R_V(\lambda)$ could be continuosly extended to the larger space
%     \[
%         R_V(\lambda):e^{-\epsilon|x|}L^2(X)\to L^2(X) 
%     \]
%     when $(\operatorname{Im}\lambda)_{-}$ is small. 
% \end{remark}

Analogous to the Euclidean case, the scattering matrix can be 
defined as the operator mapping incoming part of 
a generalized eigenfunction to the outgoint part at infinity.
\begin{theorem}\label{thm:3.42}
    Suppose $V$ is real-valued, $\lambda\in \mathbb{R}$ is 
    neither a pole of $R_V$ nor a threshold. 
    Then for any 
    collection $\{g_k\}_{\sigma_k<|\lambda|}\subset C^\infty(\mathbb{S}^{n-1})$ 
    there exists unqiue $\{f_k\}_{\sigma_k<|\lambda|}\subset C^\infty(\mathbb{S}^{n-1})$ 
    and $v\in H^2_{\operatorname{loc}}(X)$ such that 
    \begin{equation}
        \begin{aligned}
        &(P_V-\lambda^2)v=0\\
        &v(r\theta,y)=r^{-\frac{n-1}{2}}
        \sum_{|\sigma_k|<|\lambda|}\left(
            e^{i\tau_k(\lambda)r}f_k(\theta)+e^{-i\tau_k(\lambda)r}g_k(\theta)+\mathcal{O}(r^{-1})
        \right)\otimes \varphi_k(y)
        \\
        &+\mathcal{O}(e^{-\varepsilon(\lambda)r})
        \end{aligned}
    \end{equation}
    where all the remaining terms can be differentiated.
    And for each $j,k$ with $|\lambda|>\max(\sigma_j,\sigma_k)$, we have
    \begin{equation}
        \begin{aligned}
        S_{\operatorname{abs},jk}(\lambda):g_k\mapsto f_j\\
        \end{aligned}
    \end{equation}
\end{theorem}

\begin{proof}
    Uniquesness follows from the Rellich uniqueness theorem \ref{thm:Rellich's uniqueness theorem}, 
    since $g_k=0$ implies that $v$ satisfies the outgoing condition.

    For the existence, define 
    \[
    u_0(x,y):=\frac{1}{c_n^+}\sum_{\sigma_k<|\lambda|} 
    (\tau_k(\lambda))^{\frac{n-1}{2}}\int_{\mathbb{S}^{n-1}}g_k(\omega)e^{-i\tau_k \langle x,\omega\rangle} d\omega
    \otimes \varphi_k(y)
    \]
    Then $(-\Delta_X-\lambda^2)u_0=0$ and the coefficient in 
    the asymptotic of $u_0$ of the part
    $e^{-i\tau_k(\lambda)r}r^{-\frac{n-1}{2}}$ 
    is exactly $g_k$, by the asymptotic of the plane wave 
    in Proposition \ref{prop:3.38}. Next 
    we define 
    \[
    \begin{aligned}
        \tilde{u}(x,y):=&R_V(\lambda)(Vu_0)\\
        =&R_0(\lambda)(I+VR_0(\lambda)\rho)^{-1}(Vu_0)\in H^2_{\operatorname{loc}}
    \end{aligned}
    \]
    and define
    \[  
        v(x,y):=u_0(x,y)-\tilde{u}(x,y)
    \]
    Then it's easy to see that $(P_V-\lambda^2)v=0$.
    To compute $f_j$, we write the expansion of $\tilde{u}$ with respect to $\varphi_k$
    \[
        \tilde{u}(x,y)=\sum_{k} \tilde{u}_k(x)\otimes \varphi_k(y):=\sum_{\sigma_k<|\lambda|}
        \tilde{u}_k(x)\otimes \varphi_k(y)+R(x,y)
    \]
    The remainder term $R$ is of expotentially decay on $|x|$ as in the proof of Lemma \ref{lem:Range of R0 is outgoing},
    and we can commute $R_V$ and the integration to see
    \[
    \begin{aligned}
    \tilde{u}=
    \frac{1}{c_n^+}\sum_{\sigma_k<|\lambda|}(\tau_k)^{\frac{n-1}{2}}
    \int_{\mathbb{S}^{n-1}}g_k(\omega)
    R_V(\lambda)(Ve^{-i\tau_k \langle \bullet,\omega\rangle}\otimes \varphi_k)
    d\omega+R
    \end{aligned}
    \]
    and thus for $\sigma_j<|\lambda|$
    \[
        f_j(\theta)=\sum_{\sigma_k<|\lambda|} 
        \int_{\mathbb{S}^{n-1}}b(\theta;\lambda,\omega;
        j,k)g_k(\omega)d\omega+i^{n-1}g_j(-\theta)
    \]
    as the definiton of $S_{\operatorname{abs},jk}$.
\end{proof}

The scattering matrix can be analytically continued along $\mathbb{R}$, as shown in 
the following proposition.
\begin{proposition}\label{prop:holomorphy of scattering matrix}
    The scattering matrix $S_{jk}(\lambda)$ is holomorphic 
    for $\lambda\in \mathbb{R}$ and $|\lambda| \geq \max(\sigma_k,\sigma_j)$. 
\end{proposition}

\begin{proof}
    By boundary pairing Proposition \ref{prop:boundary pairing}, 
    and the definition of scattering matrix given in Theorem \ref{thm:3.42}, 
    we know for $\lambda\in \mathbb{R}$ and $\lambda$ is 
    not a pole or a threshold, for any $g\in L^2(\mathbb{S}^{n-1})$ 
    we have
\[
    \sum_{\sigma_l<\lambda}
    \tau_l(\lambda)||S_{lk}(\lambda)(g)||^2_{L^2(\mathbb{S}^{n-1})}
    =\tau_k(\lambda)||g||^2_{\mathbb{S}^{n-1}}
\] 
note all $\tau_j(\lambda)$ have the same sign.
This shows that 
\[
||S_{jk}||_{L^2(\mathbb{S}^{n-1})\to L^2(\mathbb{S}^{n-1})}
=\mathcal{O}((|\tau_j(\lambda)|/|\tau_k(\lambda)|)^{\frac{1}{2}})
\]

Note when $\lambda$ is far away from $\pm \sigma_k$, this implies 
that $||S_{jk}(\lambda)||$ is bounded. When $\lambda$ is near $\pm \sigma_k$, 
this implies that $||S_{jk}(\lambda)||$ is of $\mathcal{O}(|\tau_k(\lambda)|^{-1/2})$, 
which rules out the possibility of a pole at $\sigma_k$, since any such pole will give 
a singularity $\mathcal{O}(\tau_k(\lambda)^{-1})$ by Laurent expansion.
\end{proof}

To make the scattering matrix a unitary operator, we define 
for each $\lambda \in \mathbb{R}-\{\pm \sigma_k\}$ 
and each $j,k$ with $|\lambda|>\max(\sigma_j,\sigma_k)$
the \textbf{normalized scattering matrix} 
\begin{equation}\label{eq:definition of normalized scattering matrix}
    S_{\operatorname{nor},jk}(\lambda):=\tau_j(\lambda)^{\frac{1}{2}}
    S_{jk}(\lambda)\tau_k(\lambda)^{-\frac{1}{2}}
\end{equation}
We remark here that $S_{\operatorname{nor},jk}(z)$ 
may not be defined as a meromorphic family of operators depending on $z
\in \hat{\mathcal{Z}}$, for the function $\tau_k(z)^{\frac{1}{2}}$ can not 
be globally defined. And we will use notation $S_{
    \operatorname{nor}
}(\lambda)$ to denote the matrix  
whose entries are elements in $\mathcal{L}(L^2(\mathbb{S}^{n-1}))$ 
via
\[
    S_{\operatorname{nor}}(\lambda)=
    \{S_{\operatorname{nor},jk}(\lambda)\}_{\max(\sigma_j,\sigma_k)<\lambda}
\]
Let $N_p(\lambda)$ 
denote the number of eigenvalues of $-\Delta_M$ less than $\lambda^2$, 
counted by multiplicities.
Then $S_{\operatorname{nor}}(\lambda)$ is a matrix of order $N_p(\lambda)$, and it's unitary, in the sense that 
\[
\begin{aligned}
    &S_{\operatorname{nor}}^*(\lambda)S_{\operatorname{nor}}(\lambda)
    =S_{\operatorname{nor}}(\lambda)S_{\operatorname{nor}}^*(\lambda)=
    \operatorname{Id}:L^2(\mathbb{S}^{n-1},\mathbb{C}^{N_p(\lambda)})\to L^2(\mathbb{S}^{n-1},\mathbb{C}^{N_p(\lambda)})
\end{aligned}
\]

\subsection{Regulairity and symmetry of scattering matrix}\label{subsection:Regulairity and symmetry of scattering matrix}

The next proposition shows that the kernel of $A_{jk}$ is analytic.
\begin{proposition}\label{prop:Kernel of Ajk is analytic n greater than 3}
    Suppose $n\geq 3$. The map 
    \[
         (\lambda,\theta,\omega)\mapsto A_{jk}(\lambda,\theta,\omega)
    \]
    is analytic for 
    \[
        \lambda\in \mathbb{R}, |\lambda|\geq \max(\sigma_j,\sigma_k),
        \quad (\theta,\omega)\in \mathbb{S}^{n-1}\times \mathbb{S}^{n-1}
    \]
    % Moreover, we have 
    % \begin{equation}\label{eq:estimate for Ajk for j}
    %     A_{jk}(\lambda)(\lambda,\theta,\omega)=\mathcal{O}(\tau_j(\lambda))
    % \end{equation}

    % While for all $n\geq 3$, the map 
    % \[
    %      (\lambda,\theta,\omega)\mapsto A_{jk}(\lambda)(\lambda,\theta,\omega)
    % \]
    % is analytic for $\lambda>\max(\sigma_j,\sigma_k)$ which is not a threshold, and $(\theta,\omega)\in \mathbb{S}^{n-1}\times \mathbb{S}^{n-1}$.
\end{proposition}

\begin{proof}
    We can use Remark \ref{rmk:2.2.15} to express $A_{jk}$ as
    \[
        \begin{aligned}
            A_{jk}(z,\omega,\theta)
            =-\frac{\tau_j^{\frac{n-3}{2}}\tau_k^{\frac{n-1}{2}}}{2i(2\pi)^{n-1}}
            \int_{\mathbb{R}^n\times M} 
            e^{-i\tau_j \langle x,\theta \rangle}(I-VR_V(z))(Ve^{i\tau_k \langle 
                \bullet,\omega
            \rangle}\otimes \varphi_k)(x,y)\bar{\varphi}_j(y) dxdy
        \end{aligned}
        \]
    The only singularity may occur when $z$ is at thresholds or poles of $R_V$. 
    We first assume $\lambda_0=\sigma_{k_0}$, and we next show 
    \[
        VR_V(z)(Ve^{-i\tau_k(z)\langle \bullet,\omega\rangle}\otimes \varphi_k)
    \]
    is uniformly bounded in $L^2$ for $z\in \hat{\mathcal{Z}}$ near $\lambda_0$ 
    and $\omega\in \mathbb{S}^{n-1}$. Then by applying Cauchy's integral formula on $z$ variable, 
    we can show $A_{jk}$ is analytic near $\lambda_0$.
    
    We recall the Laurent expansion of $R_V$ for $z\in \hat{\mathcal{Z}}$ near $\lambda_0$
    \[
        R_V(z)=-\frac{\Pi_{\lambda_0}}{\tau_{k_0}(z)^2}+\frac{A_1}{\tau_{k_0}(z)}+B(z)
    \]
    where 
    \[
        A_1,B(z):L^2_{\operatorname{comp}}\to L^2_{\operatorname{loc}}
    \]
    and $B(z)$ is holomorphic for $z\in \hat{\mathbb{Z}}$ near $\lambda_0$. 
    There are two cases for $\tau_k$.
    \begin{itemize}
        \item the first case is $\tau_k=\tau_{k_0}$. The singularity $\tau_{k_0}(\lambda)^{-2}$
        is mitigated by the coefficient $\tau_k(\lambda)^{\frac{n-1}{2}}$ if $n\geq 5$. 
        When $n=3$, there is a $\tau_{k_0}$ term to cancel one-order singularity, 
        and recall that Proposition \ref{prop:Projection onto eigenspace of free wave} 
        implies that $\Pi_{\lambda_0}(V(1\otimes \varphi_k))=0$. Thus 
        \[
        \begin{aligned}
            &\frac{\Pi_{\lambda_0}}{\tau_{k_0}(\lambda)}(Ve^{-i\tau_k(\lambda)\langle ,\omega\rangle}\otimes \varphi_k)
            \\
            =&\Pi_{\lambda_0}\left(V
            \left(
                \frac{e^{-i(\tau_k(\lambda))\langle \bullet,\omega\rangle}-
            1}{\tau_{k_0}(\lambda)}\otimes \varphi_k\right)\right)
        \end{aligned}
        \]
        which implies the boundedness.
        \item the second case is $\tau_k<\tau_{k_0}$ . Then we can apply
        Proposition \ref{prop:Projection onto eigenspace of free wave} 
        to obtain 
        \[
        \begin{aligned}
            &\frac{\Pi_{\lambda_0}}{\tau_{k_0}(\lambda)^2}(Ve^{-i\tau_k(\lambda)\langle ,\omega\rangle}\otimes \varphi_k)
            \\
            =&\Pi_{\lambda_0}\left(e^{-i\tau_k(\lambda_0)\langle \bullet,\omega\rangle}
            \left(
                V\frac{e^{-i(\tau_k(\lambda_0)-\tau_k(\lambda))\langle \bullet,\omega\rangle}-
            1}{\tau_{k_0}(\lambda)^2}\right)\otimes \varphi_k\right)
        \end{aligned}
        \]
        We note that since $\sigma_k<\sigma_{k_0}$
        we have
        \[
            \tau_k(\lambda_0)-\tau_k(\lambda)=\mathcal{O}(|\lambda_0-\lambda|)
            =\mathcal{O}(\tau_{k_0}(\lambda)^2)
        \]
        so this term is bounded. The analysis for $A_1$ is the same, since the operator $A_1$ is also 
        a summation of inner products with elements in $\tilde{H}_{\pm \sigma_k}$.
    \end{itemize}
    For $z$ near $\lambda_0$ where $\lambda_0$ is a pole but not a threshold, 
    we can write 
    \[
        R_V(z)=-\frac{\Pi_{\lambda_0}}{\tau_0(z)^2-\lambda_0^2}+A(z)
    \]
    where 
    \[
        A(z):L^2_{\operatorname{comp}}\to L^2_{\operatorname{loc}}
    \]
    is holomorphic for $z\in \hat{\mathbb{Z}}$ near $\lambda_0$.
    The same proof as in the second case then suffices.
\end{proof}

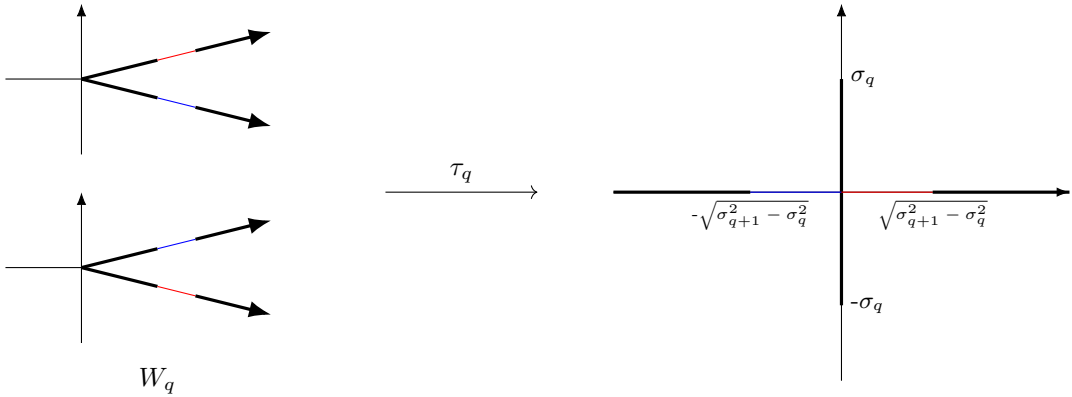
\begin{figure}[h] % [h] 参数表示“here”，即在当前位置插入图片
    \centering % 居中图片
\begin{tikzpicture}
        \tikzset{->-/.style=
        {decoration={markings,mark=at position #1 with 
        {\arrow{latex}}},postaction={decorate}}}
        \tikzset{-<-/.style=
        {decoration={markings,mark=at position #1 with 
        {\arrow{latex reversed}}},postaction={decorate}}}
        
        \draw (-1,0) -- (0,0);
        \draw[-Latex] (0,-1) -- (0,1);
        \draw[very thick] (0,0) -- (1,0.25) ;
        \draw[red] (1,0.25) -- (1.5,0.375);
        \draw[very thick][-Latex] (1.5,0.375) -- (2.5,0.625);
        \draw[very thick] (0,0) -- (1,-0.25) ;
        \draw[blue] (1,-0.25) -- (1.5,-0.375);
        \draw[very thick][-Latex] (1.5,-0.375) -- (2.5,-0.625);
    
        \draw[shift={(0,-2.5)}](-1,0) -- (0,0);
        \draw[shift={(0,-2.5)}][-Latex] (0,-1) -- (0,1);
        \draw[shift={(0,-2.5)}][very thick] (0,0) -- (1,0.25) ;
        \draw[shift={(0,-2.5)}][blue] (1,0.25) -- (1.5,0.375);
        \draw[shift={(0,-2.5)}][very thick][-Latex] (1.5,0.375) -- (2.5,0.625);
        \draw[shift={(0,-2.5)}][very thick] (0,0) -- (1,-0.25) ;
        \draw[shift={(0,-2.5)}][red] (1,-0.25) -- (1.5,-0.375);
        \draw[shift={(0,-2.5)}][very thick][-Latex] (1.5,-0.375) -- (2.5,-0.625);
        \node at (1,-4) {$W_q$};

        \node[anchor=south] at (5,-1.5) {$\tau_q$};
        \draw[->] (4,-1.5) -- (6,-1.5);

        \draw[shift={(10,-1.5)}][-Latex] (-3,0) -- (3,0);
    \draw[shift={(10,-1.5)}][-Latex] (0,-2.5) -- (0,2.5);
    \draw[shift={(10,-1.5)}][very thick] (0,-1.5) -- (0,1.5);
    \draw[shift={(10,-1.5)}][very thick] (-3,0) -- (-1.2,0);
    \draw[shift={(10,-1.5)}][very thick] (1.2,0)-- (3,0);
    \draw[shift={(10,-1.5)}][blue] (-1.2,0) -- (0,0);
    \draw[shift={(10,-1.5)}][red] (0,0) -- (1.2,0);
    \node[shift={(10,-1.5)}][circle] at (-1.2,-0.3) {\tiny{-$\sqrt{\sigma_{q+1}^2-\sigma_q^2}$}};
    \node[shift={(10,-1.5)}][circle] at (1.2,-0.3) {\tiny{$\sqrt{\sigma_{q+1}^2-\sigma_q^2}$}};
    \node[shift={(10,-1.5)}][anchor=west] at (0,-1.5) {\small{-$\sigma_q$}};
    \node[shift={(10,-1.5)}][anchor=west] at (0,1.5) {\small{$\sigma_q$}};
\end{tikzpicture}
\caption{$\tau_q(z)$ is a conformal chart on $W_q$. The upper $\mathbb{C}$ is the physical region, which is mapped to the upper half plane 
under $\tau_q$, while the lower $\mathbb{C}$ is mapped to the lower half plane 
under $\tau_q$. The bold lines are removed.}
\end{figure}

% \begin{tikzpicture}
%     \draw[-Latex] (-3,0) -- (3,0);
%     \draw[-Latex] (0,-2.5) -- (0,2.5);
%     \draw[very thick] (0,-1.5) -- (0,1.5);
%     \draw[very thick] (-3,0) -- (-1.2,0);
%     \draw[very thick] (1.2,0)-- (3,0);
%     \draw[blue] (-1.2,0) -- (0,0);
%     \draw[red] (0,0) -- (1.2,0);
%     \node[circle] at (-1.2,-0.3) {\tiny{-$\sqrt{\sigma_{k+1}^2-\sigma_k^2}$}};
%     \node[circle] at (1.2,-0.3) {\tiny{$\sqrt{\sigma_{k+1}^2-\sigma_k^2}$}};
% \end{tikzpicture}
    % \includegraphics[scale=0.3]{pictureofZ0.png} % 插入图片，example-image.png 应替换为你的图片文件名

Next we will examine the symmetry of the scattering matrix $S_{jk}$ between $\theta,\omega$ and $j,k$. Our method 
is essentially the same as \cite[Chapter 3.7]{MathematicalTheoryofScatteringResonances}.
We now define operators $\operatorname{Conj}_q(z)$ and 
$\operatorname{Oppo}_q(z)$ for each $q\in \mathbb{N}_0$ 
so that $\sigma_{q+1}>\sigma_q$, and for those $z$ lying in the
region $W_q\subset \hat{\mathcal{Z}}$, 
where $W_q$ is defined as the connected component containing the physical region 
in the subspace
\[
    \{z\in \hat{\mathcal{Z}}:\tau_0(z)^2\notin [0,\sigma_q^2]\cup [\sigma_{q+1}^2,\infty)\}
\]
of $\hat{\mathcal{Z}}$.
As a set, it is a disjoint union of two copies of $\mathbb{C}-\mathbb{R}_{\geq 0}$ and two real intervals $
(\sigma_q^2,\sigma_{q+1}^2)$. By our construction of $\hat{\mathcal{Z}}$ in subsection \ref{subsection:The construction of Riemann surface}, 
$W_q$ is actually the image of two copies of cut $\mathbb{C}$ with $[0,\sigma_q^2)\cup [\sigma_{q+1}^2,\infty)$ 
removed, one has ranking $1$ and the other has ranking $2^{q}+1$ in $\mathcal{Z}_{q}$, 
under the natural inclusion $\mathcal{Z}_q\to \hat{\mathcal{Z}}$. 

In $W_q$ we see 
\[
    W_q\ni z\mapsto \tau_q(z)\in \mathbb{C}-i[-\sigma_q,\sigma_q]
    -\left[\sqrt{\sigma_{q+1}^2-\sigma_q^2},\infty\right)-
    \left(-\infty,-\sqrt{\sigma_{q+1}^2-\sigma_q^2}\right]
\]
is a biholomorphic map, so we can use $\tau_q(z)$ as a coordinate in $W_q$. 
The operators $\operatorname{Conj}_q(z)$ and $\operatorname{Oppo}_q(z)$ are defined by 
\[
    \tau_q(\operatorname{Conj}_q(z))=\overline{\tau_q(z)},
    \quad \tau_q(\operatorname{Oppo}_q(z))=-\tau_q(z)
\]
for each $z\in W_q$.
Thus $\operatorname{Conj}_q$ is an anti-holomorphic map, 
while $\operatorname{Oppo}_q$ is a holomorphic map. 
In the region $W_q$ we define a matrix $S(z)$ of order $q+1$, 
with each entry an element in $\mathcal{L}(L^2(\mathbb{S}^{n-1}))$,  
via 
\begin{equation}\label{eq:definition of Sabsq}
    S_{\operatorname{abs}}^q(z):=\{S_{\operatorname{abs},jk}(z)\}_{0\leq j,k\leq q}
    :L^2(\mathbb{S}^{n-1},\mathbb{C}^{q+1})\to L^2(\mathbb{S}^{n-1},\mathbb{C}^{q+1})
\end{equation}
We also define the matrix $\mathcal{T}_q(z)$ of order $q+1$ by 
\[
    \mathcal{T}_q(z)=\operatorname{diag}(\tau_k(z))_{0\leq k\leq q}
\]

\begin{proposition}\label{prop:3.43}
    For $z\in W_q$, we have 
    \[
    \begin{aligned}
        &S_{\operatorname{abs}}^q(\operatorname{Oppo}_q(z))S_{\operatorname{abs}}^q(z)=
        S_{\operatorname{abs}}^q(z)S_{\operatorname{abs}}^q(\operatorname{Oppo}_q(z))=\operatorname{Id} \\
        &(S_{\operatorname{abs}}^q)^*(\operatorname{Conj}_q(z))\mathcal{T}_q(z)S_
        {\operatorname{abs}}^q(z)
        =\mathcal{T}_q(z)
    \end{aligned}
    \]
    whenever they are defined.
\end{proposition}

\begin{proof}
    When $\tau_q(z)\in \mathbb{R}$ and $z$ is not a pole, the statement for $\operatorname{Oppo_q}$ 
    follows from Theorem \ref{thm:3.42}, and 
    the statement for $\operatorname{Conj}_q$ follows from the boundary pairing. 
    For general $z$ it follows from analytic continuation.
\end{proof}

Recall that $b(\theta;\lambda,\omega;j,k)\in C^\infty(\mathbb{S}^{n-1}_{\theta}\times 
\mathbb{S}^{n-1}_{\omega})$ is defined as the Schwartz kernel of the operator 
\[
    S_{\operatorname{abs},jk}-\delta_j^k i^{1-n} J:L^2(\mathbb{S}^{n-1}_{\omega})\to L^2(\mathbb{S}^{n-1}_{\theta})
\]
The next proposition is an analogue of \cite[(3.7.7)]{MathematicalTheoryofScatteringResonances}.
\begin{proposition}\label{prop:symmetry of Scattering matrix}
    We have
    \[
        b(\theta;z,\omega;k,j)=b(\omega;z,\theta;j,k)\frac{\tau_j(z)}{\tau_k(z)}
    \]
    for all $z\in \hat{\mathcal{Z}}$ which is not a pole of $S_{abs,jk}$.
\end{proposition}

\begin{proof}
    For $t>0$ sufficiently large, we take $z_0=-t^2$ in the physical region.
    Now for any $q\in \mathbb{Z}_{\geq 0}$ with $q\geq \max(j,k)$ we have 
    \[
        \tau_q(z_0)=i\sqrt{t^2+\sigma_q^2}
    \]
    so we know
    \[
        \operatorname{Conj}_q(z_0)=\operatorname{Oppo}_q(z_0)
    \]
    and actually we see $\tau_j(z_0)=-\tau_j(\operatorname{Oppo}_q(z_0))$ for all $j\leq q$, by our construction of $\hat{\mathcal{Z}}$, 
    since $\operatorname{Oppo}_q(z_0)$ corresponds the $-q^2$ in the second copy of cut $\mathbb{C}$. (Recall $W_q$ is the union 
    of two copies cut $\mathbb{C}$ with something in half real line removed. See the figure above.)
    So we have by Proposition \ref{prop:3.43}
    \begin{equation}\label{eq:selfadjointness in the proof of prop:symmetry of Scattering matrix}
    \begin{aligned}
        S_{\operatorname{abs}}^q(z_0)&=
        S_{\operatorname{abs}}^q(
            \operatorname{Oppo}_q(z_0))^{-1}\\
        &=
        S_{\operatorname{abs}}^q(
                \operatorname{Conj}_q(z_0))^{-1}\\
            &=
        \mathcal{T}_q(z_0)^{-1}(S_{\operatorname{abs}}^q)^*(z_0
        )\mathcal{T}_q(z_0)
    \end{aligned}
    \end{equation}
    Now $\tau_j(z_0)i$ is real for all $j\in \mathbb{N}_0$,
    and we note that the Schwartz kernel of 
    \[
        (I+VR_0(z_0)\rho)^{-1}=I-VR_V(z_0)\rho
    \]
    is real since $R_V(z_0)=(P_V+t^2)^{-1}$ is the inverse of 
    an operator defined in $L^2(\mathbb{R}^n,\mathbb{R})$, 
    in addition we see the Schwartz kernel of $E_{\rho,j}(z_0)$ is real. Therefore we know 
    the Schwartz kernel of $A_{jk}(z_0)i^{n-1}$ is real by its expression 
    in Proposition \ref{prop:Description of the Scattering matrix}.
    Thus the Schwartz kernel of $S_{\operatorname{abs},jk}(z_0)$ is real, since 
    there is a coefficient $i^{n-1}$ cancelled. 
    % Note that if $K_{jk}(\theta,\omega)$ is the matrix Schwartz kernel of 
    % the operator, then under the inner product $\lambda$, 
    % the adjoint $\tilde{K}_{jk}(\theta,\omega)$ should satisfy
    % \[ 
    %     \sum_{jk} \tau_j(\lambda) \int K_{jk}(\theta,\omega)f_k(\omega)\overline{g_j(\theta)} d\omega d\theta 
    %     =\sum_{jk} \tau_k(\lambda) \int f_k(\theta)\overline{\tilde{K}_{kj}(\theta,\omega)}\overline{g_j(\omega)} d\omega d\theta 
    % \]
    % This implies 
    % \[
    %     \tilde{K}_{kj}(\theta,\omega) = \overline{K_{jk}(\omega,\theta)}\frac{\tau_j(\lambda)}{\tau_k(\lambda)}
    % \]
    Combining \eqref{eq:selfadjointness in the proof of prop:symmetry of Scattering matrix}
    and the $\mathbb{R}$-valued property, we see 
    \[
        b(\theta,z_0,\omega;k,j)=b(\omega,z_0,\theta;j,k)\frac{\tau_j(z_0)}{\tau_k(z_0)}
    \]
    since $b$ represents the Schwartz kernel of $S_{\operatorname{abs}}$.
    By analytic continuation we see
    \[
        b(\theta,z,\omega;k,j)=b(\omega,z,\theta;j,k)\frac{\tau_j(z)}{\tau_k(z)}
    \]
    for all $z\in \hat{\mathcal{Z}}$ which is not a pole of $S_{abs,jk}$.
\end{proof}

\subsection{Spectral measures in terms of distorted plane wave}

Recall in remark \ref{rmk:2.2.15} we have shown that $R_V(z)$ is symmetric for all $z\in \hat{\mathcal{Z}}$  in the sense that 
\[
    R_V(\lambda,x_1,y_1,x_2,y_2)=R_V(\lambda,x_2,y_2,x_1,y_1),\quad (x_1,y_1),(x_2,y_2)\in \mathbb{R}^n\times M
\]
The following lemma gives the asymptotic of $R_V(\lambda)$, analogous to \cite[Lemma 3.48]{MathematicalTheoryofScatteringResonances}.
\begin{lemma}\label{lem:3.48}
    Suppose $\lambda\in \mathbb{R}$ is 
    not a pole of $R_V$ or a threshold, $\varphi_k$ is real-valued.
    Then locally uniformly for $(x_2,y_2)\in X$ and $\omega\in \mathbb{S}^{n-1}$ as $r\to +\infty$, 
    we have
    \[
        \begin{aligned}
            &R_V(\lambda,r\omega,y,x_2,y_2)\\
            =&\sum_{\sigma_j<|\lambda|}
            \frac{e^{i\tau_j(\lambda)r}}
            {4\pi r^{\frac{n-1}{2}}}
            \left(\frac{\tau_j(\lambda)}{2\pi i}\right)^{\frac{n-3}{2}}
            \left(e(x_2,y_2;\lambda,\omega;j)+\mathcal{O}(r^{-1})\right)\otimes \varphi_j(y)
            +\mathcal{O}(e^{-\epsilon(\lambda)r})
        \end{aligned}
    \]
    for some $\epsilon(\lambda)>0$.
\end{lemma}

\begin{proof}
    We can write
    \[
        R_V(\lambda)=R_0(\lambda)-R_0(\lambda)VR_V(\lambda)
    \]
    Since 
    \[
        R_V(\lambda,r\omega,y,x_2,y_2)=R_V(\lambda)(\delta_{x_2}\otimes 
        \delta_{y_2})(r\omega,y_1)
    \]
    The asymptotics of $R_0^{\mathbb{R}^n}$ 
    in proposition \ref{thm:3.5} and the proof of 
    Lemma \ref{lem:Range of R0 is outgoing},
    implies that 
    for some $\epsilon(\lambda)>0$
    \[
    \begin{aligned}
        R_V(\lambda,r\omega,y,x_2,y_2)=&\sum_{\sigma_j<|\lambda|} 
        \frac{1}{4\pi}\left(\frac{\tau_j(\lambda)}{2\pi i}\right)^{\frac{n-3}{2}}
        \left(\hat{u}_j(\tau_j(\lambda)\omega)+\mathcal{O}(r^{-1})\right)\otimes \varphi_j(y)\\
        &+\mathcal{O}(e^{-\epsilon(\lambda)r})
    \end{aligned}
    \]
    where $u_j$ is defined as
    \[
        u_j(x_1)=\varphi_j(y_2)\delta_{x_2}(x_1)-
        \int_{M} V(x_1,y_1)R_V(\lambda,x_1,y_1,x_2,y_2)\varphi_j(y_1) dy_1
    \]
    Therefore we have
    \[
    \begin{aligned}
        \hat{u}_j(\tau_j(\lambda)\omega)&=\varphi_j(y_2)e^{-i\tau_j\langle x_2,\omega\rangle}-\\
        &\int_{\mathbb{R}^n\times M} e^{-i\tau_j \langle x_1,\omega\rangle}V(x_1,y_1)R_V(\lambda,x_1,y_1,x_2,y_2)\varphi_j(y_1) dy_1dx_1
    \end{aligned}
    \]
    By the symmetry of the Schwartz kernel of $R_V$, 
    the integral equals to 
    \[
    R_V(\lambda)(Ve^{-i\tau_j\langle \bullet,\omega\rangle}\otimes \varphi_j)(x_2,y_2)
    \]
    as the definition of $e$ in \eqref{eq:3.7.1}.
\end{proof}

The next theorem will represent the spectral measure of $P_V$ in terms of the distorted wave plane $e$, 
in view of Stone's formula.
\begin{proposition}\label{thm:3.47}
    Suppose $\lambda\in \mathbb{R}$ is not a threshold, 
    then 
    % \[
    %     \overline{e(x,y;\lambda,\omega;k)}=e(x,y;-\lambda,\omega;k)
    % \]
    % and 
    % \[
    %     b(\theta;-\lambda,\omega;k,j)=
    %     \overline{b(\theta;\lambda,\omega;j,k)}
    % \]
    % Moreover 
    we have
    \begin{equation}\label{eq:StoneFormulaForRV}
        \begin{aligned}
            &R_V(\lambda,x_1,y_1,x_2,y_2)-R_V(-\lambda,x_1,y_1,x_2,y_2)=\\
            &\frac{i}{2} \frac{1}{(2\pi)^{n-1}}\sum_{\sigma_j<|\lambda|}
            \tau_{j}(\lambda)^{n-2}\int_{\mathbb{S}^{n-1}} e(x_1,y_1;\lambda,\omega;j)\overline{e(x_2,y_2;\lambda,\omega;j)} d\omega
            \end{aligned}
    \end{equation}
\end{proposition}

\begin{proof}
    We can assume $\lambda$ is not a pole, 
    for both sides are continuous at $\lambda \in \mathbb{R}-\{\pm \sigma_k\}_{k\geq 0}$. 
    Indeed, the left side is continuous since the singularity of $R_V$ at $\lambda_0\in \mathbb{R}-\{\pm \sigma_k\}_{k\geq 0}$  
    are both the orthogonal projection onto the $L^2$-eigenspace of eigenvalue $\lambda_0^2$. And the right hand is 
    continuous since $A_{jk}$ is continuous in $\mathbb{R}-\{\pm \sigma_k\}_{k\geq 0}$, by proposition \ref{prop:Kernel of Ajk is analytic n greater than 3}.
%     When $\operatorname{Im} \lambda>0$
%     \[
%         \overline{R_V(\lambda)\bar{u}}=R_V(\lambda)^*u=R_V(-\bar{\lambda})u
%     \]
%     which holds also for $\lambda\in \mathbb{R}$, by continuity. Thus 
%     since $\tau_k(\lambda)=-\tau_k(-\lambda)$ when $\lambda\geq \sigma_k$ 
%     we have
%     \[
%     \begin{aligned}
%         \overline{e(x,y;\lambda,\omega;k)}&=e^{-i\tau_k(-\lambda)\langle x,
%         \omega\rangle}-\overline{R_V(\lambda)(\overline{Ve^{-i\tau_k(-\lambda)}\otimes\varphi_k})}\\
%         &=e^{-i\tau_k(-\lambda)\langle x,
%         \omega\rangle}-
%         {R_V(-\lambda)({Ve^{-i\tau_k(-\lambda)}\otimes\varphi_k})}
%         =e(x,y;-\lambda,\omega;k)
%     \end{aligned}  
%     \]
%     Doing expansion as $|x|\to \infty$, we see 
% \[
%     \begin{aligned}
%     &e(r\theta,y;-\lambda,\omega;k)
%     \sim \overline{c_n^+}(\tau_k(\lambda)r)^{-\frac{n-1}{2}}\\
%     & \left(e^{-i \tau_k(-\lambda)r}\delta_{\omega}(\theta)+
%     e^{-i\tau_k(-\lambda)r}i^{1-n}(\delta_{-\omega}(\theta))\right)\otimes \varphi_k(y)\\
%     &+\overline{c_n^+}(\tau_k(\lambda)r)^{-\frac{n-1}{2}}
%     \sum_{\sigma_j<|\lambda|} 
%     e^{ir\tau_j(-\lambda)}\overline{b(\theta;\lambda,\omega;k,j)}\otimes \varphi_j(y)
%     \end{aligned}
% \]
% since 
% \[
%     \overline{c_n^+}=c_n^+ e^{-i\pi\frac{n-1}{2}}=c_n^+(-1)^{(n-1)/2}
% \]
% we obtain 
% \[
%     b(\theta;-\lambda,\omega;k,j)=
%     \overline{b(\theta;\lambda,\omega;j,k)}
% \]
    By polarization identity, we see \eqref{eq:StoneFormulaForRV} is equivalent to that
    \[
    \begin{aligned}
        &\langle R_V(\lambda)u-R_V(-\lambda)u,u\rangle_{L^2(X)}\\
        =&
            \sum_{\sigma_j<|\lambda|} \frac{i}{2} \frac{1}{(2\pi)^{n-1}} \tau_j(\lambda)^{n-2}
            \int_{\mathbb{S}^{n-1}}\left|\int_{X} e(x_2,y_2;\lambda,\omega;k)u(x_2,y_2)dx_2dy_2 \right|^2 d\omega
    \end{aligned}
    \]
    holds for any $u\in C_c^\infty(X)$. For $\operatorname{Im} \lambda>0$ we have $R_V(-\bar{\lambda})^*=R_V(\lambda)$ thus by continuity we see
    \[
        R_V(\lambda)=R_V(-\lambda)^*,\quad \lambda\in \mathbb{R}
    \]
    Suppose $\operatorname{supp}u\subset B_R\times M$, where $B_R\subset \mathbb{R}^n$ 
    is the open ball of radius $R$ centered at zero, then we have 
    \[
        \begin{aligned}
        &\langle R_V(\lambda)-R_V(-\lambda)u,u\rangle_{L^2(X)}\\
        =&\langle R_V(\lambda)u,u\rangle_{L^2(X)}-\langle u,R_V(\lambda)u\rangle_{L^2(X)}\\
        =&\langle R_V(\lambda)u,u\rangle_{L^2(B_R\times M)}
    -\langle u,R_V(\lambda)u\rangle_{L^2(B_R\times M)}\\
    =&\left\langle R_V(\lambda)u,(P_V-\lambda^2)R_V(\lambda)u\right\rangle_{L^2(B_R\times M)}
    -\left\langle (P_V-\lambda^2)R_V(\lambda)u,R_V(\lambda)u\right\rangle_{L^2(B_R\times M)}\\
    =&\langle \Delta_X R_V(\lambda)u, R_V(\lambda)u\rangle_{L^2(B_R\times M)}
    -\langle R_V(\lambda)u,\Delta_X R_V(\lambda)u\rangle_{L^2(B_R\times M)}
        \end{aligned}
    \]
    Applying Green's formula, and using the fact that $R_V(\lambda)u(x)\in C^\infty(X-B_R\times M)$ by elliptic 
    regularity, this equals to 
    \begin{equation} \label{eq:3.8.8}
        2i \operatorname{Im} \int_{\partial B(0,R)\times M} \partial_r (R_V(\lambda)u)(x,y) \overline{R_V(\lambda)u(x,y)} dS(x) dy
    \end{equation}
    Using the asymptotic expansion of $R_V(\lambda)$ in the lemma \ref{lem:3.48}, we see 
    \[
        \begin{aligned}
            \partial_r &\left(R_V(\lambda)u\right)(R\omega,y) \overline{\left(R_V(\lambda)u\right)(R\omega,y)}=
            \sum_{\sigma_j,\sigma_k<|\lambda|} i\tau_j(\lambda)^{\frac{n-1}{2}}\tau_k(\lambda)^{\frac{n-3}{2}}\\
            &|c_n|^2 R^{-n+1}\left(\int_{X} e(x_2,y_2;\lambda,\omega;k)u(x_2,y_2)dx_2dy_2\right)\\
            &\left(\int_{X} \overline{e(x_2,y_2;\lambda,\omega;j)u(x_2,y_2)}dx_2dy_2\right)\varphi_j(y)\otimes \varphi_k(y)+\mathcal{O}(R^{-n})
        \end{aligned}
    \]
    where the constant
    \[
        c_n=\frac{1}{4\pi}\left(\frac{1}{2\pi i}\right)^{(n-3)/2}
    \]
    Upon integration over $\partial B(0,R)\times M$, the remaining term contributes $\mathcal{O}(R^{-1})$, and in the leading term 
    only the case $j=k$ will survive. Thus the formula \eqref{eq:3.8.8} becomes 
    \[
        \sum_{\sigma_j<|\lambda|} 2i|c_n|^2 \tau_j(\lambda)^{n-2}
        \int_{\mathbb{S}^{n-1}}\left|\int_{X} e(x_2,y_2;\lambda,\omega;k)u(x_2,y_2)dx_2dy_2 \right|^2 d\omega + \mathcal{O}(R^{-1})
    \]
    Letting $R\to \infty$ the proof is complete.
\end{proof}

\section{Proof of the Birman-Krein trace formula}

Before proving the Birman-Krein trace formula, 
we need first show that the operator $f(P_V)-f(P_0)$ is of trace class.

\begin{proposition}\label{thm:3.50}
    Suppose $V\in L^\infty_{\operatorname{comp}}(\mathbb{R}^n,\mathbb{R})$. 
    Then for $f\in \mathscr{S}(\mathbb{R})$
    \[
        f(P_V)-f(P_0)\in \mathcal{L}_1(L^2(\mathbb{R}^n))
    \]
    and the map 
    \[
        \mathscr{S}(\mathbb{R})\ni f\mapsto \operatorname{tr}(f(P_V)-f(P_0))
    \]
    defines a tempered distribution on $\mathbb{R}$.
    In addition, if $\mathbf{1}_{B_r\times M}$
    denotes the indicator function on $B(0,r)\times M$, we have 
    \[
        \mathbf{1}_{B_r\times M}f(P_V)\in \mathcal{L}_1(L^2(\mathbb{R}^n))
    \]
    and 
    \begin{equation}\label{eq:limiting trace formula}
        \operatorname{tr}\left(f(P_V)-f(P_0)\right)=\lim_{r\to \infty} \operatorname{tr}\left(
            \mathbf{1}_{B_r\times M}\left(f(P_V)-f(P_0)\right)\right)
    \end{equation}
    Moreover, we have the following trace norm estimate
    \[
    ||(P_V-z)^{-1}(P_V+M)^{-N}-
    (P_0-z)^{-1}(P_0+M)^{-N}||_{\mathcal{L}_1}\leq C|\operatorname{Im} z|^{-2}
    \]
    and the following singular value estimate for large $M>0$
    \begin{equation}\label{eq:3.9.7}
        s_j(\rho(P_V+M)^{-k})=s_j((P_V+M)^{-k}\rho)\leq Cj^{-2k/(n+\operatorname{dim} M)}
    \end{equation}
\end{proposition}

\begin{proof}
    The proof is the same as Theorem 3.50 in \cite{MathematicalTheoryofScatteringResonances}.
\end{proof}

Next we turn to the proof of Theorem \ref{thm:Birman-Krein formula on product space}, 
the Birman-Krein trace formula. 
We recall 
the definition of normalized scattering matrix in \eqref{eq:definition of normalized scattering matrix}. 
We remark here that for $\lambda\in \mathbb{R}_{\geq 0}-\{\sigma_k\}_{k\geq 0}$, 
the following identity holds
\[
\begin{aligned}
    \operatorname{tr}(S_{\operatorname{nor}}(\lambda)^{-1}\partial_{\lambda}S_{\operatorname{nor}}(\lambda))
    =\partial_{\lambda} \log \det(S_{\operatorname{nor}}(\lambda))=
    \partial_{\lambda} \log \det(S(\lambda))=\operatorname{tr}(S(\lambda)^{-1}\partial_\lambda S(\lambda))
\end{aligned}
\]
if we define $S(\lambda)$ as 
\[
    S(\lambda)=(S_{jk}(\lambda))_{0\leq j,k\leq N_p(\lambda)-1}:L^2(\mathbb{S}^{n-1},\mathbb{C}^{N_p(\lambda)})\to L^2(\mathbb{S}^{n-1},\mathbb{C}^{N_p(\lambda)})
\]
And we recall that $N_p(\lambda)$ is the number of eigenvalues of $-\Delta_M$ less than $\lambda^2$, 
counted by multiplicities. We also recall that $S(\lambda)$ 
defined here is exactly $S_{\operatorname{abs}}^q(\lambda)$ in \eqref{eq:definition of Sabsq} as $q=N_p(\lambda)-1$. 
We know $S_{\operatorname{abs}}^q(\lambda)$ is invertible as shown in 
Proposition \ref{prop:3.43}, 
and we also know that the kernel of $\partial_{\lambda}S(\lambda)$ is real analytic as shown in Proposition \ref{prop:Kernel of Ajk is analytic n greater than 3}, 
so the integrand in \eqref{eq:trace formula on product space}
\[
    \operatorname{tr}(S_{\operatorname{nor}}(\lambda)^{-1}\partial_{\lambda}S_{\operatorname{nor}}(\lambda))
\]
is a locally bounded function of $\lambda$ for $\lambda\geq 0$.

We make some remarks on the Birman-Krein trace formula \eqref{eq:trace formula on product space} we want to prove.
The first integral should be interpreted as a distributional pairing, which currently only makes sense 
for $f\in C_c^\infty(\mathbb{R})$. 
The second summation counts the eigenvalues with multiplicities, while the 
third summation ranges over the set of all thresholds, 
with each threshold counted only once.

In the proof of the Birman-Krein trace formula, we first 
show that this trace formula holds for $f\in C_c^\infty(\mathbb{R})$
such that $\operatorname{supp} f$ is a compact subset of $\mathbb{R}-\{\sigma_k\}_{k\geq 0}$.
Finally, we need to handle the contribution for $\lambda$ near the 
thresholds to complete the proof. 

\subsection{Proof of the Birman-Krein formula for $f$ has support far away from thresholds}

\begin{proof}[Proof of the Birman Krein formula for $f\in C_c^\infty(\mathbb{R}-\{\sigma_k^2\}_{k\geq 0})$.]

    We first assume that $f\in C_c^\infty(\mathbb{R}-\{\sigma_k^2\}_{k\geq 0})$.
    Recall 
    \[
        R_V(\lambda)=R_0(\lambda)-R_0(\lambda)VR_V(\lambda)
    \]
    Thus according to the definition of $e$ given in \eqref{eq:3.7.1}
    \[
        e(x,y;\lambda,\omega;k)=e^{-i\tau_k(\lambda) \langle x,\omega\rangle}\otimes \varphi_k(y)-R_V(\lambda)(Ve^{-i\tau_k(\lambda)\langle\bullet,\omega\rangle}\otimes \varphi_k)(x,y)
    \]
    depends analytically on $(\lambda,\omega)\in (\mathbb{R}-\{\pm \sigma_j\}_{j\geq 0})\times
    \mathbb{S}^{n-1}$, and has an asymptotic expansion as $|x|\to \infty$, which depends analytically 
    on $\lambda\in \mathbb{R}-\{\pm \sigma_j\}_{j\geq 0}$. That is 
    \[
    \begin{aligned}
        e(x,y;\lambda,\omega;k)\in C^\infty&((\mathbb{R}_{\lambda}-\{\pm \sigma_j\}_{j\geq 0})\times  
        \mathbb{S}^{n-1}_{\omega};H^2_{\operatorname{loc}}(X_{x,y}))\cap \\
        C^\infty&((\mathbb{R}_{\lambda}-\{\pm \sigma_j\}_{j\geq 0})\times  
        \mathbb{S}^{n-1}_{\omega}\times (\mathbb{R}^n-B(0,R))\times M)
    \end{aligned}
    \]
    More precisely, we define
    \[
        \tilde{e}(x,y;\lambda,\omega;k):=(c_n^{+})^{-1}\tau_k(\lambda)^{\frac{n-1}{2}} e(x,y;\lambda,\omega;k)
    \]
    where
    \[
        c_n^+=e^{\frac{\pi}{4}(n-1)i}(2\pi)^{\frac{n-1}{2}}
    \]
    Similarly, we can define $\tilde{e}_0$ for the corresponding $\tilde{e}$ in the case $V=0$. 
    Then we have an decomposition for $\tilde{e}_0$
    \[
        \tilde{e}_0(r\theta,y;\lambda,\omega;k)=\left(e^{i\tau_k r}a(r,\theta,\omega;\tau_k(\lambda))+
        e^{-i\tau_k r}\tilde{a}(r,\theta,\omega;\tau_k(\lambda))\right)\otimes \varphi_k(y)
    \]
    where $a$ and $\tilde{a}$ have asymptotic expansions as $r\to \infty$
    \[
        a(r,\theta,\omega;\tau_k(\lambda))\sim r^{-\frac{n-1}{2}}\sum_{l=0}^\infty a_l(\theta,\omega;\tau_k(\lambda)),\quad a_l\in C^\infty(\mathbb{S}^{n-1}_{\omega},\mathscr{D}'(\mathbb{S}^{n-1}_{\theta}))
    \]
    while the corresponding asymptotic expansion of $\tilde{a}$ is denoted by $\tilde{a}_j$. 
    This asymptotic expansion arises from the decomposition of plane wave $e^{-i\lambda \langle x,\omega \rangle}$ 
    in Proposition \ref{prop:3.38}. We see 
    \[
        a_0=\delta_{-\omega}(\theta)i^{1-n},\quad \tilde{a}_0=\delta_\omega(\theta)
    \]
    We also define the difference 
    \[
        \tilde{e}(x,y;\lambda,\omega;k)-\tilde{e}_0(x,y;\lambda,\omega;k):=
        \tilde{\eta}(x,y;
        \lambda,\omega;k)
    \]
    which is smooth for sufficiently large $|x|$, and we can write 
    \[
        \tilde{\eta}(r\theta,y;
        \lambda,\omega;k)=\sum_{\sigma_j<|\lambda|} e^{i\tau_j r}
        B(r,\omega,\theta;\lambda;j,k)\otimes \varphi_j(y)+\mathcal{O}(e^{-\epsilon(\lambda)r})
    \]
    where the remainder terms can be differentiated, and $B$ is smooth and has an asymptotic sum when $r$ is large 
    \[
        B(r,\omega,\theta;\lambda;j,k)\sim r^{-\frac{n-1}{2}}\sum_{l=0}^\infty r^{-l}
        B_l(\theta,\omega;\lambda;j,k)
    \]
    where
    \[
        B_l\in C^\infty(\mathbb{S}^{n-1}_{\omega}\times 
        \mathbb{S}^{n-1}_{\theta}\times (\mathbb{R}_{\lambda}-\{\pm \sigma_k\}_{k\geq 0}))
    \]
    and 
    \[
        B_0(\theta,\omega;\lambda;j,k)=b(\theta;\lambda,\omega;j,k),\quad 
        B_0(\theta,\omega;\lambda;j,k)=B_0(\omega,\theta;\lambda;k,j)\frac{\tau_k(\lambda)}{\tau_j(\lambda)}
    \]
    The asymptotic expansion of $B$ is derived from the asymptotic expansion of $R_0^{\mathbb{R}^n}(\lambda)$ in 
    Proposition \ref{thm:3.5}, and we should notice that this expansion remains valid for all $\lambda\in \mathbb{R}-\{\pm \sigma_j\}_{j\geq 0}$ 
    even when $\lambda$ might be a pole, since $R_V(\lambda)(Ve^{-i\tau_k(\lambda)\langle\bullet,\omega\rangle}\otimes \varphi_k)$ 
    is analytic for $\lambda$, as shown in Proposition \ref{prop:Kernel of Ajk is analytic n greater than 3}.
    The last identity uses the symmetry of $b$, Proposition \ref{prop:symmetry of Scattering matrix}. 
    Since we assume 
    the support $f$ is compact and does not intersect those thresholds, all the 
    expansion is uniform for $\lambda^2$ lying in the support of $f$.

    By Stone's formula, and since $f\in C_c^\infty$ has support away from the thresholds, 
    we have 
    \[
    \begin{aligned}
        f(P_V)&=\lim_{\varepsilon\to 0_+} 
        \frac{1}{2\pi i} \int_{\mathbb{R}}
        f(t) \left((P_V-t-i\epsilon)^{-1}-(P_V-t+i\epsilon)^{-1}\right) 
        dt \\
        &=\lim_{\varepsilon\to 0_+} 
        \frac{1}{2\pi i} \int_{0}^\infty 
        f(t) \left((P_V-t-i\epsilon)^{-1}-(P_V-t+i\epsilon)^{-1}\right) + \sum_{
            E_k\in \operatorname{Spec_{pp}}(P_V),E_k<0} f(E_k)\Pi_{E_k}\\
        &=\sum_{E_k\in \operatorname{Spec_{pp}}(P_V)} f(E_k)\Pi_{E_k}+
        \frac{1}{\pi i} \int_{0}^\infty 
        \lambda f(\lambda^2)\left(R_V(\lambda)-R_V(-\lambda)\right) d\lambda
    \end{aligned}
    \]
    with an analogous formula valid for $f(P_0)$, where $\Pi_{E_k}$ is the orthogonal projection 
    onto the $L^2$-eigenspace of $P_V$ associated with the eigenvalue $E_k$.
    Thus by the limiting trace formula \eqref{eq:limiting trace formula}  
    we have
    \[
        \begin{aligned}
        &\operatorname{tr}(f(P_V)-f(P_0))-\sum_{E_k\in \operatorname{Spec_{pp}}(P_V)} f(E_k)\\
        = &\lim_{r\to \infty} \mathbf{1}_{B_r\times M}\operatorname{tr}(f(P_V)-f(P_0))
        \mathbf{1}_{B_r\times M}-\sum_{E_k\in \operatorname{Spec_{pp}}(P_V)} f(E_k)\\
        = &\lim_{r\to \infty} \frac{1}{\pi i} \int_{0}^\infty 
        \lambda f(\lambda^2) 
        \operatorname{tr}\left(\mathbf{1}_{B_r\times M}(R_V(\lambda)-R_V(-\lambda))\mathbf{1}_{B_r\times M}\right)\\
        &-\operatorname{tr}\left(\mathbf{1}_{B_r\times M}(R_0(\lambda)-R_0(-\lambda))\mathbf{1}_{B_r\times M}\right) d\lambda
        \end{aligned}
    \]
    Therefore we apply Theorem 
    \ref{thm:3.47} to obtain
    \begin{equation}\label{eq:3.9.50}
        \begin{aligned}
        \operatorname{tr}&(f(P_V)-f(P_0))-\sum_{E_l\in \operatorname{Spec_{pp}}(P_V)} f(E_l)
        =\lim_{r\to \infty} \frac{1}{(2\pi)^n}\int_{0}^\infty \lambda f(\lambda^2) d\lambda \\
        &\sum_{\sigma_k<\lambda}  \tau_k(\lambda)^{n-2} \int_{\mathbb{S}^{n-1}} d\omega
        \int_{B_r\times M} (|e(x,y;\lambda,\omega;k)|^2-|e_0(x,y;\lambda,\omega;k)|^2)dxdy
        \end{aligned}
    \end{equation} 

    Next we apply the Maass-Selberg method. Recall that $\tilde{e}$ satisfies 
    \[
        (P_V-\lambda^2)\tilde{e}=0
    \]
    Differentiating with respect to $\lambda$, we obtain 
    \[
        (P_V-\lambda^2)\partial_\lambda \tilde{e}=2\lambda \tilde{e}
    \]
    and a similar identity holds for $P_0$ and $\tilde{e}_0$. 
    Using $\lambda\neq 0$ is real, we can then apply Green's formula to obtain
    \begin{equation}\label{eq:3.9.51}
        \begin{aligned}
        \int_{B_r\times M} |\tilde{e}|^2 dxdy &=\frac{1}{2\lambda}\int_{B_r\times M} (P_V-\lambda^2)\partial_{\lambda} 
        \tilde{e} \bar{\tilde{e}}dxdy \\
        &=\frac{1}{2\lambda}\int_{B_r\times M}\left((P_V-\lambda^2)\partial_{\lambda} 
        \tilde{e} \bar{\tilde{e}}-\partial_{\lambda}\tilde{e}(P_V-\lambda^2)\bar{\tilde{e}}\right) dxdy \\
        &=\frac{1}{2\lambda}\int_{B_r\times M} \left(-\Delta_X \partial_{\lambda}\tilde{e}\bar{\tilde{e}}+
        \partial_{\lambda}\tilde{e}\Delta_X \bar{\tilde{e}}\right) dxdy\\
        &=\frac{-1}{2\lambda}\int_{\mathbb{S}^{n-1}\times M} 
        \left((\partial_r\partial_{\lambda} 
        \tilde{e}) \bar{\tilde{e}}-\partial_{\lambda}\tilde{e}\partial_{r}\bar{\tilde{e}}\right) r^{n-1} d\theta dy,\quad x=r\theta
        \end{aligned}
    \end{equation}
    An analoguous formula holds for $P_0$ and $\tilde{e}_0$.

    We next insert \eqref{eq:3.9.51} into the integral
    \[
        \int_{\mathbb{S}^{n-1}} 
        \int_{B_r\times M} (|\tilde{e}(x,y;\lambda,\omega;j)|^2-|\tilde{e}_0(x,y;\lambda,\omega;j)|^2)dxdy
    \]
    Note all terms quadratic on $\tilde{e}_0$ vanish, so the expression becomes 
    \[
    \begin{aligned}
        \frac{(-1)r^{n-1}}{2\lambda }\int_{\mathbb{S}^{n-1}} d\omega
        \int_{\mathbb{S}^{n-1}\times M} &(\partial_r\partial_{\lambda} 
        \tilde{e}_0) \bar{\tilde{\eta}}
        +(\partial_r\partial_{\lambda} 
        \tilde{\eta}) \bar{\tilde{e}}_0+
        (\partial_r\partial_{\lambda} 
        \tilde{\eta}) \bar{\tilde{\eta}}\\
        -&
        \partial_{\lambda}\tilde{e}_0\partial_{r}\bar{\tilde{\eta}}-
        \partial_{\lambda}\tilde{\eta}\partial_{r}\bar{\tilde{e}}_0-
        \partial_{\lambda}\tilde{\eta}\partial_{r}\bar{\tilde{\eta}}  \;d\theta dy
    \end{aligned} 
    \]
    Before inserting this formula into 
    \eqref{eq:3.9.50}, we make the observations:
    \begin{itemize}
        \item When pairing $\tilde{e}_0$ with $\tilde{\eta}$, 
        the distributional expansion of $\tilde{e}_0$
        is valid, for $\tilde{\eta}$ depends smmothly on $\theta$ when $r$ is large.
        \item Thanks to
        the integration over 
        $M$, all terms involving $\varphi_j \varphi_k$ in the Fourier expansion 
        will vanish when $j\neq k$.
        \item The exponentially decaying remainder of $\tilde{\eta}$ 
        contributes nothing in the limit $r\to \infty$.
        \item All integrals with oscillatory term $e^{\pm 2i\tau_j(\lambda)r}$ will tends to 
        zero as $r\to \infty$, since they are paired with $f(\lambda^2)$, whose support is compact and is 
        away from thresholds, thus we can use integration by parts via 
        \[
            e^{\pm 2i\tau_j(\lambda)r}=\frac{\pm \sqrt{\lambda^2-\sigma_j^2}}{2i\lambda r}\partial_\lambda(e^{\pm 2i\tau_j(\lambda)r})
        \]
    \end{itemize}
    Using the notation $D=(1/i)\partial$, and the identities
    \[
        D_r\circ e^{i\tau_j(\lambda)r}=e^{i\tau_j(\lambda)r}(D_r+\tau_j(\lambda)),
        \quad D_\lambda\circ e^{i\tau_j(\lambda)r}=e^{i\tau_j(\lambda)r}(D_\lambda+\frac{r\lambda}{\tau_j(\lambda)})
    \]
    we obtain
    \begin{equation}\label{eq:3.9.52}
        \begin{aligned}
        \operatorname{tr}&(f(P_V)-f(P_0))-\sum_{k} f(E_k)
        =\frac{1}{2\pi}\lim_{r\to \infty} \int_{0}^\infty f(\lambda^2) d\lambda \\
        &\sum_{\sigma_k<\lambda} \frac{\lambda}{
            \tau_k(\lambda)
        } \int_{\mathbb{S}^{n-1}} d\omega
        \int_{B_r\times M} (|\tilde{e}(x,y;\lambda,\omega;k)|^2-|\tilde{e}_0(x,y;\lambda,\omega;k)|^2)dxdy\\
        &=\frac{1}{4\pi}\lim_{r\to \infty} \int_{0}^\infty \frac{f(\lambda^2)}{-\lambda} d\lambda 
        \sum_{\sigma_k<\lambda}\frac{\lambda}{
            \tau_k(\lambda)
        } r^{n-1}\int_{\mathbb{S}^{n-1}}  d\omega
        \int_{\mathbb{S}^{n-1}} C(r;\lambda,k)d\theta \\
        \end{aligned}
    \end{equation} 
    where the function $C(r;\lambda,k)$ is defined by 
    \[
        \begin{aligned}
            C(r;\lambda,k):&=-(D_r+\tau_k(\lambda))(D_\lambda+\frac{r\lambda}{\tau_k(\lambda)})a(k)\bar{B}
            (k,k)\\
            &-(D_r+\tau_k(\lambda))(D_\lambda+\frac{r\lambda}{\tau_k(\lambda)})B(k,k)\bar{a}(k)\\
            &-\sum_{\sigma_j<\lambda}(D_r+\tau_j(\lambda))(D_\lambda+\frac{r\lambda}{\tau_j(\lambda)})B(j,k)\bar{B}(j,k)
            \\
            &+(D_\lambda+\frac{r\lambda}{\tau_k(\lambda)})a(k)(D_r-\tau_k(\lambda))\bar{B}(k,k)\\
            &+(D_\lambda+\frac{r\lambda}{\tau_k(\lambda)})B(k,k) (D_r-\tau_k(\lambda))\bar{a}(k)\\
            &+\sum_{\sigma_j<\lambda}(D_\lambda+\frac{r\lambda}{\tau_j(\lambda)})B(j,k) (D_r-\tau_j(\lambda))\bar{B}(j,k)
        \end{aligned}
    \]
    Recall that both $a$ and $B$ are of order $r^{-\frac{n-1}{2}}$, and differntiation in $r$ 
    lowers the 
    order in $r$, while differentiation in $\lambda$ preserves it. Keeping in 
    mind that the terms of order $\mathcal{O}(r^{-n})$ 
    will vanish in the limit since we integrate over $\mathbb{S}^{n-1}$, we compute 
    \[
    \begin{aligned}
        C(r;\lambda,k)=&-\frac{\lambda}{\tau_k(\lambda)}(rD_r-i)a(k)\bar{B}(k,k)-\tau_k(\lambda)(D_\lambda+\frac{r\lambda}{\tau_k(\lambda)})
        a(k)\bar{B}(k,k)\\
        &-\frac{\lambda}{\tau_k(\lambda)}(rD_r-i)B(k,k)\bar{a}(k)-\tau_k(\lambda)(D_\lambda+\frac{r\lambda}{\tau_k(\lambda)})B(k,k)\bar{a}(k)\\
        &-\sum_{\sigma_j<\lambda}\left(\frac{\lambda}{\tau_j(\lambda)}(rD_r-i)B(j,k)\bar{B}(j,k)
        +\tau_j(\lambda)(D_\lambda+\frac{r\lambda}{\tau_j(\lambda)})B(j,k)\bar{B}(j,k)\right)\\
        &+(-\tau_k(\lambda))(D_\lambda+\frac{r\lambda}{\tau_k(\lambda)})a(k)\bar{B}(k,k)+
        \frac{r\lambda}{\tau_k(\lambda)}a(k)D_r\bar{B}(k,k)\\
        &+(-\tau_k(\lambda))(D_\lambda+\frac{r\lambda}{\tau_k(\lambda)})B(k,k)\bar{a}(k)+
        \frac{r\lambda}{\tau_k(\lambda)}B(k,k) D_r\bar{a}(k)\\
        &+
        \sum_{\sigma_j<\lambda} \left((-\tau_j(\lambda))
        (D_\lambda+\frac{r\lambda}{\tau_j(\lambda)})B(j,k)\bar{B}(j,k)+
        \frac{r\lambda}{\tau_j(\lambda)}B(j,k) D_r\bar{B}(j,k)\right) \\
        &+\mathcal{O}_{\mathscr{D}'(\mathbb{S}^{n-1}_{\theta})}(r^{-n})
    \end{aligned}
    \]

    The coefficient of $r^{-n+2}$ in the expression for $C(r;\lambda,k)$ is given by 
    \[
        -2\lambda \operatorname{Re}\left(2a_0(\theta,\omega;\tau_k(\lambda))
        \bar{B}_0(\theta,\omega;\lambda;k,k)+
        \sum_{\sigma_j<\lambda}|B_0(\theta,\omega;\lambda;j,k)|^2\right)
    \]
    We claim that this contributes nothing to the integral
    \eqref{eq:3.9.52}, that is 
    \begin{equation}\label{eq:claim that r -n+2 term vanish}
       \sum_{\sigma_k<\lambda} \frac{1}{
            \tau_k(\lambda)
        } \int_{\mathbb{S}^{n-1}} \int_{\mathbb{S}^{n-1}} 
        \operatorname{Re}\left(2a_0(\theta,\omega;\tau_k(\lambda))
        \bar{B}_0(\theta,\omega;\lambda;k,k)+
        \sum_{\sigma_j<\lambda}|B_0(\theta,\omega;\lambda;j,k)|^2\right)d\omega d\theta=0
    \end{equation} 
    for any $\lambda^2$ lying in the support of $f$. To prove this, 
    recall the boundary pairing identity
    \[
        \sum_{\sigma_j<\lambda} \tau_j(\lambda) S_{\operatorname{abs},jk}^*(\lambda)
        S_{\operatorname{abs},jk}(\lambda)
        =\tau_k(\lambda)\operatorname{Id}_{L^2(\mathbb{S}^{n-1})}
    \] 
    Writing out the Schwartz kernel as operators 
    $
        L^2(\mathbb{S}^{n-1}_{\gamma})\to L^2(\mathbb{S}^{n-1}_{\theta})
    $, we have
    \[
    \begin{aligned}
        &\tau_k(\lambda)\int_{\mathbb{S}^{n-1}} (i^{1-n}\delta_{-\omega}(\theta)+\bar{b}(\omega;\lambda,
        \theta;
        k,k))(i^{1-n}\delta_{-\gamma}(\omega)+b(\omega;\lambda,\gamma;
        k,k))d\omega \\
        +&\tau_j(\lambda)
        \sum_{\sigma_j<\lambda,j\neq k} 
        \int_{\mathbb{S}^{n-1}}
        \bar{b}(\omega;\lambda,\theta;
        j,k)b(\omega;\lambda,\gamma;
        j,k) d\omega=\tau_k(\lambda)\delta_{\gamma}(\theta)
    \end{aligned}
    \]
    Letting $\gamma=\theta$, we deduce 
    \begin{equation}\label{eq:3.9.53}
        \begin{aligned}
            2i^{1-n} \operatorname{Re}(b(-\theta;\lambda,\theta;k,k))
            +
             \frac{\tau_j(\lambda)}{\tau_k(\lambda)}\int_{\mathbb{S}^{n-1}} 
            \sum_{\sigma_j<\lambda} 
            |b(\omega;\lambda,\theta;
            j,k)|^2 d\omega=0
        \end{aligned}
    \end{equation}
    On the other hand, we note that 
    \[
    \int_{\mathbb{S}^{n-1}} \int_{\mathbb{S}^{n-1}} 
        \operatorname{Re}\left(2a_0(\theta,\omega;\tau_k(\lambda))\bar{B}_0(\theta,\omega;\lambda;k,k)\right) d\omega d\theta=
        2i^{1-n}\int_{\mathbb{S}^{n-1}}\operatorname{Re}\left(b_0(-\omega,\omega;\lambda;k,k)\right) d\omega
    \]
    Hence by integrating \eqref{eq:3.9.53} over $\theta\in \mathbb{S}^{n-1}$, 
    we know the left side of \eqref{eq:claim that r -n+2 term vanish} equals to
    \begin{equation}\label{eq:fuck you sleeping}
        \sum_{\sigma_k,\sigma_j<\lambda} \frac{1}{\tau_k}
        \int_{\mathbb{S}^{n-1}\times \mathbb{S}^{n-1}} 
        -\frac{\tau_j}{\tau_k} |b(\omega;\lambda,\theta;j,k)|^2+|b(\omega;\lambda,\theta;j,k)|^2
        d\omega d\theta
    \end{equation}
    We now recall by Proposition \ref{prop:symmetry of Scattering matrix}, the symmetry of 
    scattering matrix gives
    \[
        b(\omega;\lambda,\theta;k,j)=b(\theta;\lambda,\omega;j,k)\frac{\tau_j(\lambda)}{\tau_k(\lambda)}
    \]
    Substituting this into \eqref{eq:fuck you sleeping}, we know \eqref{eq:fuck you sleeping} becomes
    \[
        C(\lambda,n)\sum_{\sigma_k,\sigma_j<\lambda} 
        \left(\frac{-1}{\tau_j}
        \int_{\mathbb{S}^{n-1}\times \mathbb{S}^{n-1}} 
         |b(\theta;\lambda,\omega;k,j)|^2
        d\omega d\theta + \frac{1}{\tau_k}
        \int_{\mathbb{S}^{n-1}\times \mathbb{S}^{n-1}} 
         |b(\omega;\lambda,\theta;k,j)|^2
        d\omega d\theta\right)
    \]
    which equals to zero by the symmetry between $(\theta,\omega)$ and $(k,j)$. This completes 
    the proof of the claimed identity \eqref{eq:claim that r -n+2 term vanish}.

    To compute the coefficient of $r^{-n+1}$ in $C(r;\lambda,k)$ 
    we note that $D_\lambda a_0=0$, thus it equals 
    to 
    \begin{equation}\label{eq:first simplication of coefficient of order -n+1}
        \begin{aligned}
        &-\frac{\lambda}{\tau_k(\lambda)}(i\frac{n-1}{2}-i)a_0(k)\bar{B}_0(k,k)-\lambda(a_0(k)
        \bar{B}_1(k,k)+a_1(k)\bar{B}_0(k,k))\\
        &-\frac{\lambda}{\tau_k(\lambda)}(i\frac{n-1}{2}-i)B_0(k,k)\bar{a}_0(k)-
        \tau_k(\lambda)D_\lambda B_0(k,k)
        \bar{a}_0(k)-\lambda(B_0(k,k)\bar{a}_1(k)+B_1(k,k)\bar{a}_0(k))\\
        &-\sum_{\sigma_j<\lambda} \left(\frac{\lambda}{\tau_j(\lambda)}
        (i\frac{n-1}{2}-i)B_0(j,k)\bar{B}_0(j,k)\right)\\
        &-\sum_{\sigma_j<\lambda} \left(
            \tau_j(\lambda)D_\lambda B_0(j,k)\bar{B}_0(j,k)+
            \lambda\left(B_0(j,k)\bar{B}_1(j,k)+B_1(j,k)\bar{B}_0(j,k)\right)
        \right)\\
        &-\lambda(a_0(k)\bar{B}_1(k,k)+a_1(k)\bar{B}_0(k,k))-\frac{n-1}{2i}\frac{\lambda}{
            \tau_k(\lambda)}a_0(k)\bar{B}_0(k,k)\\
        &-\tau_k(\lambda)D_\lambda B_0(k,k)\bar{a}_0(k)-\lambda(B_0(k,k)\bar{a}_1(k)+
        B_1(k,k)\bar{a}_0(k))-\frac{n-1}{2i}\frac{\lambda}{\tau_k(\lambda)}B_0(k)\bar{a}_0(k,k)\\
        &-\sum_{\sigma_j<\lambda} \tau_j D_\lambda B_0(j,k)\bar{B}_0(j,k)-\sum_{\sigma_j<\lambda}
        \lambda (B_0(j,k)\bar{B}_1(j,k)+B_1(j,k)\bar{B}_0(j,k))\\
        &-\sum_{\sigma_j<\lambda} \frac{n-1}{2i}\frac{\lambda}{\tau_j(\lambda)}|B_0(j,k)|^2
        \end{aligned}
    \end{equation}
    We can now group all terms in \eqref{eq:first simplication of coefficient of order -n+1} into 
    three parts, denoted by $I_{1,2,3}(k,\theta,\omega)$, defined as follows
    \[
        \begin{aligned}
            I_1:=&-\frac{2i\lambda}{\tau_k(\lambda)}(\frac{n-1}{2}-1)
            \operatorname{Re}(a_0(k)\bar{B}_0(k,k))\\
            &-\sum_{\sigma_j<\lambda} \left(\frac{i \lambda}{\tau_j(\lambda)}
            (\frac{n-1}{2}-1)B_0(j,k)\bar{B}_0(j,k)\right)\\
            &-2\frac{n-1}{2i}\frac{\lambda}{
                \tau_k(\lambda)}\operatorname{Re}(a_0(k)\bar{B}_0(k,k))\\
            &-\sum_{\sigma_j<\lambda} \frac{n-1}{2i}\frac{\lambda}{\tau_j(\lambda)}|B_0(j,k)|^2\\
            &=i\lambda
            \operatorname{Re}\left(
                \frac{1}{\tau_k(\lambda)}2a_0(k)\bar{B}_0(k,k)+\sum_{\sigma_j<\lambda} \frac{1}{\tau_j(\lambda)}|B_0(j,k)|^2
            \right)
        \end{aligned}
    \]
    \[
        I_2:=-2\tau_k(\lambda)D_\lambda B_0(k,k)
        \bar{a}_0(\lambda)-2\sum_{\sigma_j<\lambda} \tau_j(\lambda) 
        D_\lambda B_0(j,k)\bar{B}_0(j,k)
    \]
    \[
    \begin{aligned}
        I_3:=&-4\lambda \operatorname{Re}\left(a_0(k,k)\bar{B}_1(k,k)+a_1(k,k)\bar{B}_0(k,k)
        +\sum_{\sigma_j<\lambda}B_0(j,k)\bar{B}_1(j,k)\right)
    \end{aligned}
    \]
    Then we know \eqref{eq:first simplication of coefficient of order -n+1}
    equals to $I_1+I_2+I_3$. Combining \eqref{eq:3.9.52}, we know in order 
    to prove the Birman-Krein formula for $f\in C_c^\infty(\mathbb{R}-\{\sigma_j\}_{j\geq 0})$, 
    it suffices to show the following three identities
    \begin{equation}\label{eq:integration of I_1}
        \begin{aligned}
            \int_{\mathbb{S}^{n-1}\times 
        \mathbb{S}^{n-1}} I_1(k,\theta,\omega)
        d\omega d\theta= 0
        \end{aligned}
    \end{equation}
    \begin{equation}\label{eq:integration of I_2}
        \begin{aligned}
            \sum_{\sigma_k<\lambda} \frac{1}{\tau_k(\lambda)}\int_{\mathbb{S}^{n-1}\times 
        \mathbb{S}^{n-1}} I_2(k,\theta,\omega)
        d\omega d\theta= 2i \operatorname{tr}
        \left(S(\lambda)^{-1}
        \partial_{\lambda } S(\lambda) \right)
        \end{aligned}
    \end{equation}
    and
    \begin{equation}\label{eq:integration of I_3}
        \int_{\mathbb{S}^{n-1}} I_3(\omega,\theta,k)d\theta =0 
    \end{equation}
    
    We first note that \eqref{eq:integration of I_1} follows immediately from 
    \eqref{eq:3.9.53}. And by direct calculation, we find the right hand side of \eqref{eq:integration of I_2} 
    equals to 
    \[
        \begin{aligned}  
        \operatorname{tr}(S(\lambda)^{-1}\partial_{\lambda}S(\lambda))&=\operatorname{tr}(S_{\operatorname{abs}}(\lambda)^{-1}
        \partial_{\lambda}S_{\operatorname{abs}}(\lambda))\\
        &=\sum_{\sigma_k,\sigma_j<\lambda} \frac{\tau_j(\lambda)}{\tau_k(\lambda)}
        \operatorname{tr}(S_{\operatorname{abs},jk}^*(\lambda)
        \partial_{\lambda}S_{\operatorname{abs},jk}(\lambda))\\
        &=\sum_{\sigma_k<\lambda}
        \int_{\mathbb{S}^{n-1}\times 
        \mathbb{S}^{n-1}}\bar{a}_0(\omega,\theta;\tau_k(\lambda))\partial_\lambda 
        B_0(\omega,\theta;\lambda;k,k) d\theta d\omega\\
        &+\sum_{\sigma_k,\sigma_j<\lambda} \frac{\tau_j(\lambda)}{\tau_k(\lambda)}
        \int_{\mathbb{S}^{n-1}\times \mathbb{S}^{n-1}}\bar{B}_0(\omega,\theta;\lambda;j,k)\partial_\lambda 
        B_0(\omega,\theta;\lambda;j,k) d\theta d\omega
        \end{aligned}
    \]
    Thus \eqref{eq:integration of I_2} can be verified directly.

    It remains to prove \eqref{eq:integration of I_3}.
    To prove this, we recall both $\tilde{e}_0$ and $\tilde{\eta}$ satisfies 
    \[
        (P_V-\lambda^2)\tilde{\eta}=0\quad \text{at infinity}
    \]
    Thus by formula \eqref{eq:3.7.12} we obtain 
    \[
    a_1(k)=\frac{1}{-2i\tau_k(\lambda)}(-\Delta_{\mathbb{S}_{\theta}^{n-1}}+b_n)a_0(k),
    \quad B_1(j,k)=\frac{1}{-2i\tau_j(\lambda)}(-\Delta_{\mathbb{S}_{\theta}^{n-1}}+b_n)B_0(j,k)
    \]
    where $b_n:=(n-1)(n-3)/4$ is real. Thus the integral of $I_3$ equals to 
    \[
    \begin{aligned}
       &-4\lambda \operatorname{Re} \int_{\mathbb{S}^{n-1}} 
        a_0(k)\bar{B}_1(k,k)+a_1(k)\bar{B}_0(k,k)
        +\sum_{\sigma_j<\lambda}B_0(j,k)\bar{B}_1(j,k)d\theta\\
        =&-2\lambda \operatorname{Re} i\int_{\mathbb{S}^{n-1}} 
        \frac{b_n}{\tau_k(\lambda)}(-a_0(k)\bar{B}_0(k,k)+
        a_0(k)\bar{B}_0(k,k))\\
        &+
        \frac{1}{\tau_k(\lambda)}\left(a_0(k)\Delta_{\mathbb{S}_{\theta}^{n-1}}
        \bar{B}_0(k,k)-\Delta_{\mathbb{S}_{\theta}^{n-1}}a_0(k)
        \bar{B}_0(k,k)\right)\\
        &-\sum_{\sigma_j<\lambda}\frac{b_n|B_0(k,j)|^2}{\tau_j(\lambda)}
        +\sum_{\sigma_j<\lambda}\frac{B_0(k,j)\Delta_{\mathbb{S}_{\theta}^{n-1}}\bar{B}_0(k,j)}{\tau_j(\lambda)} d\theta
    \end{aligned}
    \]
    The integrals of the last two terms are real, through integration by parts 
    \[
        \int_{\mathbb{S}^{n-1}} B_0\Delta_{\theta} \bar{B}_0 d\theta=\int_{\mathbb{S}^{n-1}} -|\nabla B_0|^2 d\theta
    \]
    while the middle two terms are equal since they are distributional pairing. This completes the proof of 
    \eqref{eq:integration of I_3}.

\end{proof}

\subsection{Behaviour of the trace formula near thresholds}
To deal with the behaviour near thresholds, we will use the following lemma which set 
\[
f(x)=e^{-t(x-\sigma_{k_0}^2)}(x+M)^{-N}
\]
in the trace formula, to jump from one threshold to the next along the real line. 
This method is essentially the same as \cite[Lemma 3.52]{MathematicalTheoryofScatteringResonances}.

\begin{lemma}\label{lem:3.52}
    For $k_0\in \mathbb{Z}_{\geq 0}$ with $\sigma_{k_0+1}>\sigma_{k_0}$,
    suppose the Birman-Krein trace formula \eqref{eq:trace formula on product space} 
    holds for functions $f\in C_c^\infty(-\infty,\sigma_{k_0}^2)$.
    Then for sufficiently large $M,N>0$ we have
    \[
    \begin{aligned}
    &\operatorname{tr}\left(e^{-t(P_V-\sigma_{k_0}^2)}
    (P_V+M)^{-N}-e^{-t(P_0-\sigma_{k_0}^2)}(P_0+M)^{-N}\right)\\
    =&\int_0^{\sigma_{k_0}} e^{t(\sigma_{k_0}^2-\lambda^2)}(\lambda^2+M)^{-N}
        \operatorname{tr}(S_{\operatorname{nor}}(\lambda)^{-1}\partial_{\lambda}S_{\operatorname{nor}}(\lambda))
        d\lambda \\
        &+ \sum_{E_k\in \operatorname{Spec_{pp}}(P_V),
        E_k\leq \sigma_{k_0}^2} e^{t(\sigma_{k_0}^2-E_k)}(E_k+M)^{-N}
        +\frac{\tilde{m}_V(\sigma_{k_0})}{2}(\sigma_{k_0}^2+M)^{-N}+o(1)
        % \sum_{k=0}^\infty \frac{1}{2}f(\sigma_k^2)(m_V(\sigma_k)-2\operatorname{rank} \Pi_{\sigma_k})
    \end{aligned}
    \]
    as $t\to +\infty$.
\end{lemma}
Assuming this lemma, we can complete the proof of the Birman-Krein formula \eqref{eq:trace formula on product space}.
\begin{proof}[Completion of the proof of Birman-Krein formula assuming Lemma \ref{lem:3.52} ]
    According to the structure theorem for distributions supported at a point, 
    we know the distribution $T_V\in \mathscr{S}'(\mathbb{R})$ 
    defined by $C_c^\infty(\mathbb{R}^n)\ni f\mapsto \operatorname{tr}(f(P_V)-f(P_0))$ 
    equals to 
    \begin{equation}\label{eq:trace formula in the proof of lem 3.52}
        \begin{aligned}
            &
            T_V(f)=\frac{1}{2\pi i} 
            \int_0^\infty f(\lambda^2)
            \operatorname{tr}(S_{\operatorname{nor}}(\lambda)^{-1}\partial_{\lambda}S_{\operatorname{nor}}(\lambda))
            d\lambda \\
            & +\sum_{E_k\in \operatorname{Spec_{pp}}(P_V),
            E_k\notin \{\sigma_k^2\}} f(E_k)
            + \sum_{\lambda\in \{\sigma_k\}} 
            \sum_{j=0}^{N_{\lambda}} c_{j,\lambda}f^{(j)}(\lambda^2)
            % + 
            % \sum_{ \lambda\in \{\sigma_k\}_{k\geq 0}} \frac{1}{2}f(
            %     \lambda^2
            % )(m_V(\lambda)-2\operatorname{rank} \Pi_{\lambda})
        \end{aligned}
    \end{equation}
    where $N_{\lambda}\in \mathbb{Z}_{\geq 0}$ and $c_{j,\lambda}$ are constants. 
    All we need to show is $N_{\lambda}=0$ and $c_{0,\lambda}=\operatorname{tr} \Pi_{\lambda}
    +\frac{\tilde{m}_V(\lambda)}{2}$
    for all $\lambda \in \{\sigma_k\}$.
    
    By induction, we may assume $\lambda=\sigma_{k_0}$ 
    for some $k_0$ with $\sigma_{k_0+1}>\sigma_{k_0}$, 
    and assume that $N_{\eta}=0$, $c_{0,\eta}=\operatorname{tr} \Pi_{\eta}
    +\frac{\tilde{m}_V(\eta)}{2}$ for all $\eta<\lambda$ 
    with $\eta\in \{\sigma_k\}$. 
    We next choose $\chi\in C^\infty(\mathbb{R})$, with $\chi=1$ in $(-\infty,\sigma_{k_0}^2]$ 
    and $\operatorname{supp} \chi \subset (-\infty,\sigma_{k_0+1}^2)$. For $t>0$, 
    define functions $f_t,f_{1,t},f_{2,t}$ on $\mathbb{R}$ by 
    \[
    \begin{aligned}
        &f_t(x)=e^{-t(x-\sigma_k^2)}(x+M)^{-N},\quad  f_t=f_{1,t}+f_{2,t}\\
        &f_{1,t}=\chi f_t,\quad f_{2,t}=(1-\chi)f_t
    \end{aligned}
    \]
    where $M,N$ are sufficiently large as in Lemma \ref{lem:3.52}.
    Since $f_{2,t}\to 0$ in $\mathscr{S}$ toplogy as $t\to +\infty$,
    it follows that 
    \[
        T_V(f)=T_V(f_{1,t})+T_V(f_{2,t})=T_V(f_{1,t})+o(1),\quad t\to +\infty
    \]
    Applying \eqref{eq:trace formula in the proof of lem 3.52} 
    on $f_{1,t}$ and comparing with Lemma \ref{lem:3.52}, 
    it follows immediately $N_{\lambda}=0$ 
    and 
    \[
    c_{0,\lambda}=\operatorname{tr} \Pi_{\sigma_{k_0}}+\frac{\tilde{m}_V(\sigma_{k_0})}{2}
    \]
    since $t$ can be taken 
    arbitrarily large. 
    This completes the proof by induction.
\end{proof}

Before proving Lemma \ref{lem:3.52}, we will first establish the following estimates, 
which will be used as key ingredients in the proof.

\begin{lemma}\label{lem:estimate needed in proof of lem:3.52}
    Let $\zeta$ denotes the conformal chart near $\lambda=\pm\sigma_{k_0}$, 
    that is $z(\zeta)=\zeta^2+\sigma_k^2\in \hat{\mathcal{Z}}$, and define 
    \[
        \tilde{R}_0(\zeta):=R_0(z(\zeta))
    \]
    Then when $N,M$ is sufficiently large, we have the estimate
    \begin{equation} \label{eq:3.1.26}
        ||\zeta\rho\tilde{R}_0(\zeta)(P_0+M)^{-N}\tilde{R}_0(\zeta)\rho||_{L^2\to H^{n+1+\operatorname{dim} M}}\leq C,
        \quad \operatorname{Im} \zeta\geq 0, \pm \operatorname{Re} \zeta\geq 0, |\zeta|\leq 10
    \end{equation}
    Moreover, the following weighted-$L^2$ estimate holds for the free-resolvent:
    \begin{equation}\label{eq:Weighted estimate for RV near threshold}
        ||\langle x\rangle^{-s} 
        \tilde{R}_0(\zeta)\langle x\rangle^{-s}
        ||_{L^2\to L^2}\leq C_s,\quad s>1+\frac{n-3}{2},\;
        \operatorname{Im}\zeta \geq 0, \pm\operatorname{Re} \zeta\geq 0, |\zeta|\leq 10
    \end{equation}
    In addition, we have the following estimate for singular values 
    for $P_V+M$
    \begin{equation}\label{eq:3.9.40}
        s_j(\langle x\rangle^r (P_V+M)^{-k} \rho),
        \quad s_j(
        \langle x\rangle^r 
        (P_0+M)^{-k} \rho )\leq C_{r,M}(1+j)^{-k/(n+\operatorname{dim} M+1)},\quad r>0
    \end{equation}
    The same estimate holds near $\lambda=-\sigma_{k_0}$.
\end{lemma}

\begin{proof}[Proof of estimate \eqref{eq:3.1.26}]
    This is the same as Lemma 3.6 in \cite{MathematicalTheoryofScatteringResonances}.
\end{proof}
\begin{proof}[Proof of estimate \eqref{eq:Weighted estimate for RV near threshold}]
    Recall the free resolvent $R_0^{\mathbb{R}^n}(\lambda)$ in $\mathbb{R}^{n}$ 
    is 
    \[
        R_0^{\mathbb{R}^n}(\lambda)=\frac{e^{i\lambda|x-y|}}{|x-y|^{n-2}}P_n(\lambda|x-y|)
    \]
    where $P_n$ is a polynomial of degree $\frac{n-3}{2}$. Then in view of the 
    orthonormal basis $\varphi_{k}$, we can identify
    \[
        \tilde{R}_0(\zeta):L^2(X)\simeq l^2(\mathbb{Z}_{\geq 0},L^2(\mathbb{R}^n))
        \to l^2(\mathbb{Z}_{\geq 0},L^2(\mathbb{R}^n))\simeq L^2(X)
    \]
    Then in this identification we have
    \[
        \tilde{R}_0(\zeta)=\left\{R_0^{\mathbb{R}^n}\left(\sqrt{\sigma_{k_0}^2+\zeta^2-\sigma_k^2}\right)\right\}_{k=0}^\infty
    \]
    where all square roots take values in the closed upper half plane for $\operatorname{Im} \zeta\geq 0$.
    Thus we see 
    \[
    \begin{aligned}
        &||\langle x\rangle^{-s} \tilde{R}_0(\zeta)\langle x\rangle^{-s}||_{L^2(X)\to L^2(X)}
        \\
        \leq& \sup_{k\geq 0} \left| \left|\langle x\rangle^{-s}R_0^{\mathbb{R}^n}
        \left(\sqrt{\sigma_{k_0}^2+\zeta^2-\sigma_k^2}\right)\langle x\rangle^{-s}\right|\right|_{L^2(\mathbb{R}^n) 
        \to L^2(\mathbb{R}^n)}
    \end{aligned}
    \]
    Since the set
    \[
        \{\operatorname{Re} \sqrt{\sigma_{k_0}^2+\zeta^2-\sigma_k^2}\}
    \]
    is bounded, it remains to show that the following estimate 
    \[
        ||\langle x\rangle^{-s}R_0^{\mathbb{R}^n}(\lambda)\langle x\rangle^{-s}||\leq C_{s,n}
    \]
    holds for some constant $C_{s,n}$, uniformly for $\operatorname{Re} \lambda$ in a 
    bounded set and $\operatorname{Im} \lambda\geq 0$.

    We have the following trivial estimate 
    \[
        |\mathcal{R}_0(\lambda,x,y)|\leq \sum_{j=0}^{\frac{n-3}{2}} C_k|\lambda|^k|x-y|^{2+k-n}
        e^{-\operatorname{Im} \lambda|x-y|}
    \]
    Peetre's inequality $2\langle x\rangle\langle y\rangle\geq \langle x-y \rangle$ implies 
    \[
        \langle x\rangle^{-s} |x-y|^{2+k-n}\langle y\rangle^{-s}\leq |x-y|^{2+k-n}\langle x-y \rangle^{-s}
    \]
    Finally we apply Schur's test on the integral kernel, 
    and use spherical coordinate on $|x-y|$
    \[
    \begin{aligned}
        \int_{\mathbb{R}^n} \langle x\rangle^{-s} |\mathcal{R}_0(\lambda,x,y)|\langle y\rangle^{-s}dx
        \leq C\sum_{k=0}^{\frac{n-3}{2}} |\lambda|^k \int_{1}^\infty r^{1+k-s} e^{-\operatorname{Im} \lambda r}
        \leq C_{s,n}
    \end{aligned}
    \]
    and similarly for integration in the $y$-variable. This completes the proof.
\end{proof}
\begin{proof}[Proof of estimate \eqref{eq:3.9.40}]
    Recall the case $r=0$ is already coverred by estimate \eqref{eq:3.9.7}, and we use 
    the fact that $s_j(A)=s_j(A^*A)^{1/2}$ to obtain
    \[ 
        \begin{aligned}
        s_j(\langle x\rangle^r (P_V+M)^k \rho)^2&=
        s_j(\rho (P_V+M)^{-k} \langle x\rangle^{2r} (P_V+M)^{-k} \rho)\\
        &\leq s_j(\rho (P_V+M)^{-k}) ||\langle x\rangle^{2r} (P_V+M)^{-k} \rho||_{L^2\to L^2}\\
        &\leq Cj^{-2k/(n+\operatorname{dim} M)}||\langle x\rangle^{2r} (P_V+M)^{-k} \langle x\rangle^{-2r}||_{L^2\to L^2}
        \end{aligned}
    \]
    It therefore suffices to bound the $\langle x\rangle^{2r}$-weighted $L^2$-norm of $(P_V+M)^{-1}$.
    To this end, 
    we set $\lambda_0=i\sqrt{M}$, and recall that 
    \[
        R_V(\lambda_0)=R_0(\lambda_0)(I+VR_0(\lambda_0)\rho)^{-1}(I-VR_0(\lambda_0)(1-\rho))
    \]
    The $\langle x\rangle^{2r}$-weighted $L^2$-norm of the term $(I-VR_0(\lambda_0)(1-\rho))$ 
    is trivial, and $\langle x\rangle^{2r}$-weighted $L^2$-norm of $R_0(\lambda_0)$ follows by 
    explicitly writing out the Schwartz kernel and applying Schur's test, as in the 
    proof of estimate \eqref{eq:Weighted estimate for RV near threshold}. For
    $\langle x\rangle^{2r}$-weighted $L^2$-norm of $(I+VR_0(\lambda)\rho)^{-1}$, 
    we first observe that by applying spectral theorem on $P_0$ 
    \[
        ||\langle x\rangle^{2r}VR_0(\lambda_0)\rho \langle x\rangle^{-2r}||_{L^2\to L^2}\leq 1/2
    \]
    once we take $M$ large. Hence the Neumman's series gives the bound on $(I+VR_0(\lambda)\rho)^{-1}$.
\end{proof}

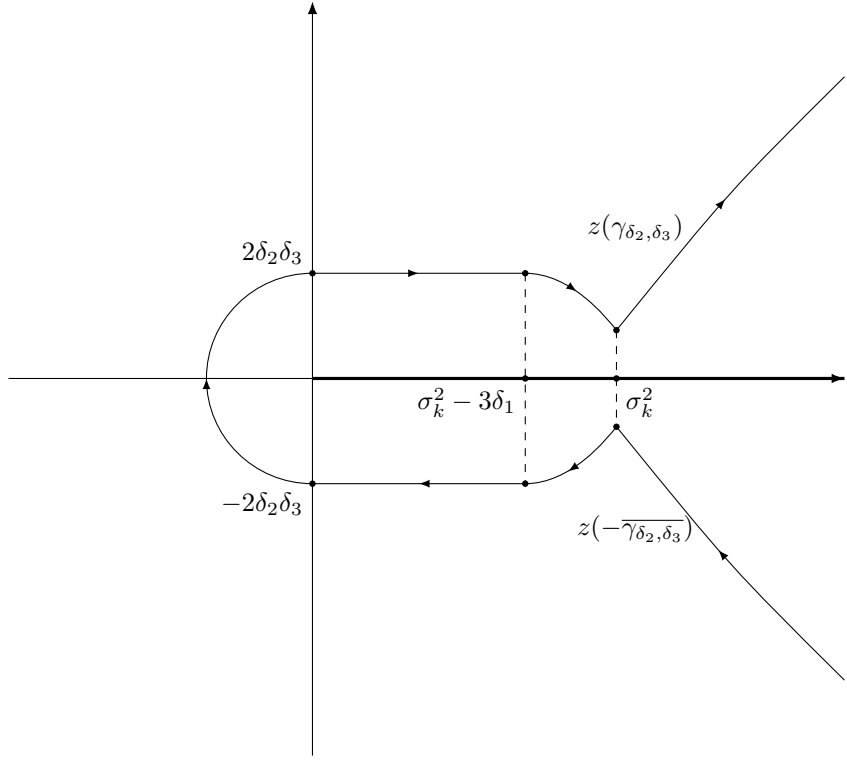
\begin{figure}
    \centering
\begin{tikzpicture}
    % [
    %     >={Kite[inset=0pt,length=0.32cm,bend]}, 
    %     decoration={markings, mark= at position .1 with {\arrow{>}}, 
    %     mark= at position .4 with {\arrow{>}}, 
    %     mark= at position 0.7 with {\arrow{>}}, }
    % ] 
    \def\radius{4} 
    \def\del2{0.4}
    \def\del1{0.2}
    \def\del3{0.8718}
    \def\del4{0.69744}
    \def\sigmak2{2}
    % contour 
    % \filldraw[postaction = {decorate}, thick ,fill=gray!40] 
    % (0:\radius) arc (0:180:\radius) node[below]{$-R$}-- cycle; 
    
    % \node at (30:\radius+0.3){$C_{R}$}; 
    % axes 、

\tikzset{->-/.style=
{decoration={markings,mark=at position #1 with 
{\arrow{latex}}},postaction={decorate}}}

\tikzset{-<-/.style=
{decoration={markings,mark=at position #1 with 
{\arrow{latex reversed}}},postaction={decorate}}}

    \filldraw[black] (0,\del4*2) circle (1pt) node[anchor=south east]{$2\delta_2\delta_3$};
    \filldraw[black] (0,-\del4*2) circle (1pt) node[anchor=north east]{$-2\delta_2\delta_3$};
    \filldraw[black] (1.4*2,\del4*2) circle (1pt);
    \filldraw[black] (1.4*2,0) circle (1pt) node[anchor=north east]{$\sigma_k^2-3\delta_1$};
    \filldraw[black] (2*2,0) circle (1pt) node[anchor=north west]{$\sigma_k^2$};
    \filldraw[black] (1.4*2,-\del4*2) circle (1pt);
    \filldraw[black] (2*2,2*0.4*0.4*2) circle (1pt);
    \filldraw[black] (2*2,-2*0.4*0.4*2) circle (1pt);
    \draw[->-=0.5] (0,-\del4*2) arc (270:90:\del4*2);
    % \draw (0,\del4*2) arc (90:135:\del4*2);
    \draw[very thick] (0,0) -- (7,0);

    \draw[->-=0.5] (0,\del4*2) -- (1.4*2,\del4*2);
    % \draw[-Latex] (0,\del4*2) -- (0.5*\sigmak2*2,\del4*2);
    % \draw (0.5*\sigmak2*2,\del4*2) -- (1.4*2,\del4*2);
    \draw[->-=0.5] (1.4*2,\del4*2) parabola (2*2,2*0.4*0.4*2);
    \draw[->-=0.5] (2*2,2*0.4*0.4*2) .. controls (5.5,2.5) .. (7,4);

    \draw[->-=0.5] (1.4*2,-\del4*2) -- (0,-\del4*2);
    % \draw[-latex reversed] (0,-\del4*2) -- (0.5*\sigmak2*2,-\del4*2);
    % \draw (0.5*\sigmak2*2,-\del4*2) -- (1.4*2,-\del4*2);
    
    \draw[-<-=0.5] (1.4*2,-\del4*2) parabola (2*2,-2*0.4*0.4*2);
    % \draw[->-=0.5] (2*2,-2*0.4*0.4*2) parabola (1.4*2,-\del4*2);
    \draw[-<-=0.5] (2*2,-2*0.4*0.4*2) .. controls (5.5,-2.5) .. (7,-4);

    \draw[dashed] (1.4*2,\del4*2) -- (1.4*2,-\del4*2);
    \draw[dashed] (2*2,2*0.4*0.4*2) -- (2*2,-2*0.4*0.4*2);
    % \filldraw[black] (-0.5*\radius,0) circle (2pt) node[below]{$-||V||_{L^\infty}$};
    % \filldraw[black] (-1*\radius,0) circle (2pt) node[below]{$c$};
    % \filldraw[black] (-1.5*\radius,0) circle (2pt) node[below]{$-1/t$};
    % \draw[dashed] (-0.5*\radius,0.5*\radius) -- (0.25*\radius,1.25*\radius);
    % \draw[dashed][-Latex]  (-1*\radius,0) -- (-0.5*\radius,0.5*\radius);
    % \draw[dashed][-Latex] (0.25*\radius,-1.25*\radius) node[below]{$\Gamma_c$} -- (-0.5*\radius,-0.5*\radius);
    % \draw[dashed] (-0.5*\radius,-0.5*\radius) -- (-1*\radius,0);
    % \draw[dashed][-Latex] (-1.5*\radius,-1.25*\radius) -- (-1.5*\radius,-0.75*\radius);
    % \draw[dashed] (-1.5*\radius,-0.75*\radius) -- (-1.5*\radius,0);
    % \draw[dashed][-Latex] (-1.5*\radius,0) -- (-1.5*\radius,0.75*\radius);
    % \draw[dashed] (-1.5*\radius,0.75*\radius) -- (-1.5*\radius,1.25*\radius);
    \draw[-Latex] (-4,0) -- (7,0); 
    \draw[-Latex] (0,-5) -- (0,5); 

    \node at (4.25,2) {$z(\gamma_{\delta_2,\delta_3})$};
    \node at (4.25,-2) {$z(-\overline{\gamma_{\delta_2,\delta_3}})$};
\end{tikzpicture}
\caption{The contour $\Gamma_{\delta_2,\delta_1}$. Note the spectrum of $P_V$ consists of  $\mathbb{R}$ 
and finitely many negative eigenvalues} \label{fig:M2}
\end{figure}

\begin{proof}[Proof of Lemma \ref{lem:3.52}]
    To prove the $o(1)$ remainder, 
    given $\varepsilon>0$, 
    we need to show the remainder has absolute value smaller than $\varepsilon$ when $t$ is large.
    Fix $\delta_{1},\delta_{2}>0$ small.
    Choose $\chi_{\delta_1}\in C^\infty(\mathbb{R})$ such that 
    \[
    \chi_{\delta_1} = 1 \quad \text{on } (-\infty, \sigma_k^2 - 2\delta_1), 
    \quad \text{and } \operatorname{supp} \chi_{\delta_1} \subset (-\infty, \sigma_k^2 - \delta_1).
    \]
    Define $\delta_3=\sqrt{3\delta_1+\delta_2^2}$, 
    and the contour $\gamma_{\eta_2,\eta_3}$ 
    for any $\eta_3>\eta_2>0$ via 
    \[
        \gamma_{\eta_2,\eta_
        3}:=\left([0,\infty)\ni s\mapsto \eta_2+i\eta_2+e^{i\frac{\pi}{8}}s\right)\bigcup 
        \left([0,\eta_3-\eta_2)\ni s \mapsto \eta_2+i(\eta_2+s)\right)
    \]
    oriented from left to right.
    Let $\hat{\mathcal{Z}}\ni z(\zeta):=\zeta^2+\sigma_{k_0}^2$ 
    be the conformal chart near $\lambda=\sigma_{k_0}$, and
    define the contours $\Gamma^\pm_{\delta_2,\delta_1}$ by
    \[
        \Gamma_{\delta_2,\delta_1}^+:=\left([\frac{\pi}{2},\pi]\ni s \mapsto 2\delta_2\delta_3 e^{is}
        \right)\bigcup 
        \left([0,\sigma_k^2-3\delta_1]\ni s \mapsto 2i\delta_2\delta_3+s\right)\bigcup z(\gamma_{\delta_2,\delta_3})
    \]
    Let $\Gamma_{\delta_2,\delta_1}^-$ be the image of $\Gamma_{\delta_2,\delta_1}^+$ 
    under reflection with across 
    the real axis. Then the full contour $\Gamma_{\delta_2,\delta_1}$ is given by
    \[
        \Gamma_{\delta_2,\delta_1}=\Gamma_{\delta_2,\delta_1}^+\cup \Gamma_{\delta_2,\delta_1}^-
    \]
    oriented from bottom to top. Note the choice 
    of $\delta_3$ ensures that $\Gamma_{\delta_2,\delta_1}$ a continuous path.
    Now, we can choose $\delta_2+\delta_1$ sufficiently small 
    so that all negative eigenvalues of $P_V$ 
    lie outside the contour $\Gamma_{\delta_2,\delta_1}$, then we have
    \[
    \begin{aligned}
    &e^{-t(P_V-\sigma_{k_0}^2)}(P_V+M)^{-N}-
    e^{-t(P_0-\sigma_{k_0}^2)}(P_0+M)^{-N}\\
    =&\sum_{E_k\in \operatorname{Spec_{pp}}(P_V), E_k<0} (E_k+M)^{-N}e^{-t(E_k-\sigma_{k_0}^2)}u_k\otimes \bar{u_k}\\
    &+\frac{1}{2\pi i}\int_{\Gamma_{\delta_2,\delta_1}} 
    ((P_V-z)^{-1}(P_V+M)^{-N}-
    (P_0-z)^{-1}(P_0+M)^{-N})e^{-t(z-\sigma_{k_0}^2)} dz
    \end{aligned}
    \]
    For simplicity of notation, we assume that $P_V$ has no negative eigenvalues. 
    Define a truncated almost analytic extension $\tilde{\chi}_{\delta_1}\in C^\infty(\mathbb{C})$
    of $\chi_{\delta_1}$ via
    \[
        \tilde{\chi}_{\delta_1}(x+iy)
        =\sum_{k=0}^{\tilde{N}-1} \partial^{k}_{x}\chi_{\delta_1}(x)\frac{(iy)^k}{k!}
    \]
    where $\tilde{N}$ is a large integer to be determined later. Then we have 
    \begin{equation}\label{eq:estimate for partialbar chidelta1}
        \partial_{\bar{z}}\tilde{\chi}_{\delta_1}(z)=\mathcal{O}_{\delta_1}(
        |\operatorname{Im} z|^{\tilde{N}})
    \end{equation}
    and the support satisfies
    \[
        \operatorname{supp} \tilde{\chi}_{\delta_1}\subset 
        \{\operatorname{Re} z\leq \sigma_k^2-\delta_1\}
        \quad 
        \operatorname{supp} d\tilde{\chi}_{\delta_1}\subset 
        \{\sigma_k^2-2\delta_1\leq \operatorname{Re} z\leq \sigma_k^2-\delta_1\}
    \]
    We recall that in Theorem \ref{thm:3.50} for large $N$
\begin{equation}\label{eq:3.9.27}
    ||(P_V-z)^{-1}(P_V+M)^{-N}-
    (P_0-z)^{-1}(P_0+M)^{-N}||_{\mathcal{L}_1}\leq C|\operatorname{Im} z|^{-2}
    \end{equation}
    Then we have 
    \begin{equation}\label{eq:}
    \begin{aligned}
    &e^{-t(P_V-\sigma_{k_0}^2)}(P_V+M)^{-N}-
    e^{-t(P_0-\sigma_{k_0}^2)}(P_0+M)^{-N}\\
    =&\frac{1}{2\pi i}\int_{\Gamma_{\delta_2,\delta_1}} 
    \left((P_V-z)^{-1}(P_V+M)^{-N}-
    (P_0-z)^{-1}(P_0+M)^{-N}\right)e^{-t(z-\sigma_{k_0}^2)}\tilde{\chi}_{\delta_1}(z)dz\\
    &+
    \frac{1}{2\pi i}\int_{\Gamma_{\delta_2,\delta_1}} 
    \left((P_V-z)^{-1}(P_V+M)^{-N}-
    (P_0-z)^{-1}(P_0+M)^{-N}\right)e^{-t(z-\sigma_{k_0}^2)}(1-\tilde{\chi}_{\delta_1}(z)) dz\\
    :=&I_1+I_2
    \end{aligned}
\end{equation}
where $I_1$ is the first integral, and $I_2$ is the second integral.
We can write
    \[
    \begin{aligned}
        I_1=&\frac{1}{2\pi i}\int_{0}^\mathbb{R} \int_{\Gamma_{\delta_2,\delta_1}} 
    \frac{e^{-t(z-\sigma_{k_0}^2)}}{s-z}\tilde{\chi}_{\delta_1}(z) dz 
    dE_V(s) (P_V+M)^{-N}\\
    &-\frac{1}{2\pi i}\int_{0}^\mathbb{R} \int_{\Gamma_{\delta_2,\delta_1}} 
    \frac{e^{-t(z-\sigma_{k_0}^2)}}{s-z}\tilde{\chi}_{\delta_1}(z) dz dE_0(s) (P_0+M)^{-N}
    \end{aligned}
    \]
    where $E_V,E_0$ denotes the spectral measure of $P_V,P_0$ respectively.
    By Cauchy-Green's formula we obtain 
    \[
    \begin{aligned}
        &\frac{1}{2\pi i}\int_{0}^\mathbb{R} \int_{\Gamma_{\delta_2,\delta_1}} 
    \frac{e^{-t(z-\sigma_{k_0}^2)}}{s-z}\tilde{\chi}_{\delta_1}(z) dz 
    dE_V(s)\\
    =&\int_0^{\mathbb{R}}
    e^{-t(s-\sigma_{k_0}^2)}\chi_{\delta_1}(s)dE_V(s)
    -\frac{1}{\pi}\int_{0}^\mathbb{R} \int_{\operatorname{Int} \Gamma_{\delta_2,\delta_1}} 
    \frac{e^{-t(z-\sigma_{k_0}^2)}}{s-z}\partial_{\bar{z}}
    \tilde{\chi}_{\delta_1}(z) dzdE_V(s)\\
    =&
    e^{-t(P_V-\sigma_{k_0}^2)}\chi_{\delta_1}(P_V)
    -\frac{1}{\pi}\int_{\operatorname{Int} \Gamma_{\delta_2,\delta_1}} 
    e^{-t(z-\sigma_{k_0}^2)}\partial_{\bar{z}}
    \tilde{\chi}_{\delta_1}(z)(P_V-z)^{-1} dz
    \end{aligned}
    \]
    where $\operatorname{Int} \Gamma_{\delta_2,\delta_1}$ is the 
    connected component containing the positive real axis in $\mathbb{C}-\Gamma_{\delta_2,\delta_1}$.
    Thus we have 
    \[
    \begin{aligned}
    I_1=&e^{-t(P_V-\sigma_{k_0}^2)}(P_V+M)^{-N}\chi_{\delta_1}(P_V)
    -e^{-t(P_0-\sigma_{k_0}^2)}(P_0+M)^{-N}\chi_{\delta_1}(P_0)\\
    &-\frac{1}{\pi}
    \int_{\operatorname{Int} \Gamma_{\delta_2}} 
    \left((P_V-z)^{-1}(P_V+M)^{-N}-
    (P_0-z)^{-1}(P_0+M)^{-N}\right)e^{-t(z-\sigma_{k_0}^2)}\partial_{\bar{z}}
    \tilde{\chi}_{\delta_1}(z) dz
    \end{aligned}
    \]
    Now by the support condition of $d\tilde{\chi}_{\delta_1}$
    we see in the integration region $\operatorname{Int} \Gamma_{\delta_2}$ 
    the imaginary part of $z$ is of $\mathcal{O}(\delta_2)$,
    thus by estimate \eqref{eq:3.9.27} and  
    \eqref{eq:estimate for partialbar chidelta1}, 
    we can take trace on both sides and obtain 
    \[
        \begin{aligned}
            \operatorname{tr} I_1=&
    \operatorname{tr}\left(e^{-t(P_V-\sigma_{k_0}^2)}(P_V+M)^{-N}\chi_{\delta_1}(P_V)
    -e^{-t(P_0-\sigma_{k_0}^2)}(P_0+M)^{-N}\chi_{\delta_1}(P_0)\right)\\
    &+\mathcal{O}_{\delta_1,t,\tilde{N}}(\delta_2^{\tilde{N}-2})
    \end{aligned}
    \]
    % \begin{equation} \label{eq:3.1.26}
    %     ||\zeta\rho\tilde{R}_0(\zeta)(P_0-M)^{-N}\tilde{R}_0(\zeta)\rho||_{L^2\to H^{n+1}}\leq C,
    %     \quad \operatorname{Re} \zeta\geq 0,  \operatorname{Im} \zeta\geq 0, |\zeta|\leq 10
    % \end{equation}
    % \begin{equation}\label{eq:Weighted estimate for RV near threshold}
    %     ||\langle x\rangle^{-s} 
    %     \tilde{R}_0(\zeta)\langle x\rangle^{-s}
    %     ||_{L^2\to L^2}\leq C_s,\quad s>1+\frac{n-3}{2},|\zeta|\leq 10, 
    %     \operatorname{Im}\geq 0, \operatorname{Re} \zeta\geq 0
    % \end{equation}
    % \begin{equation}\label{eq:3.9.40}
    %     s_j(\langle x\rangle^r (P_V+M)^k \rho),
    %     \quad s_j(
    %     \langle x\rangle^r 
    %     (P_0+M)^k \rho )\leq C_sj^{-k/(n+\operatorname{dim} M)},\quad r>0
    % \end{equation}

    Next we analyze $I_2$. We first consider the 
    integration region where $\operatorname{Im} z>0$. Recall Proposition \ref{prop:Resolvent near threshold when n dayudengyu five}, we have
    for $\zeta$ in a neighborhood of zero 
    \[
        \tilde{R}_V(\zeta):=R_V(z(\zeta))=-\frac{\Pi_{\sigma_{k_0}}}{\zeta^2}+\frac{A_1}{\zeta}+A(\zeta)
    \]
    where $A(\zeta):L^2_{\operatorname{comp}}\to L^2_{\operatorname{loc}}$ is holomorphic 
    in a neighborhood of zero in $\mathbb{C}$, $A_1$ 
    is characterized as 
    \[
        A_1=\sum_{j=1}^J u_j\otimes v_j
    \]
    and $z(\zeta)\in \hat{\mathcal{Z}}$.  
    We recall that $A_1=0$ when $n\geq 7$, and $u_j,v_j\in 
    \operatorname{ran} \Pi_{\sigma_{k_0}}$ when $n=5$. 
    In the case $n=3$, we can write 
    \[
        u_j=\sum_{\sigma_k<\sigma_{k_0}} u_j^k(x)\otimes \varphi_k(y)+\sum_{\sigma_k=\sigma_{k_0}} 
        u_j^k(x)\otimes \varphi_k(y)+\sum_{\sigma_k>\sigma_{k_0}} u_j^k(x)\otimes \varphi_k(y)
    \]
    where $u_j^k(x)$ is compactly supported for those $k$ with $\sigma_k<\sigma_{k_0}$.

    We note that $A$ is also holomorphic in $\zeta\in \{\operatorname{Im} \zeta>0, 
    \operatorname{Re} \zeta>0\}$. Similarly, we define $\tilde{R}_0(\zeta)=R_0(z(\zeta))$.
    And we can write for $\zeta\in \{\operatorname{Im} \zeta>0, 
    \operatorname{Re} \zeta>0\}$ (which corresponds to the physical region near $\lambda=+\sigma_{k_0}$),
    using remark \ref{rmk:2.2.15}
    \begin{equation}\label{eq:zetaRVV}
        \begin{aligned}
        \zeta \tilde{R}_V(\zeta)V&=\zeta \tilde{R}_0(\zeta)
        (I+V\tilde{R}_0(\zeta)\rho)^{-1}V
        =\zeta\tilde{R}_0(\zeta)
        (I-V\tilde{R}_V(\zeta))V\\
        &=\zeta^{-1}\tilde{R}_0(\zeta)
        V\Pi_{\sigma_{k_0}}V-\tilde{R}_0(\zeta)VA_1V+\tilde{R}_0(\zeta)\rho B(\zeta)V
        \end{aligned}
    \end{equation}
    where $B$ is holomorphic near zero defined by
    \[
        B(\zeta):=I-VA(\zeta)\rho:L^2(X)\to L^2(X)
    \]
    Using the identity
    \begin{equation}\label{eq:tildeR0zetaV}
        \tilde{R}_0(\zeta)V\Pi_{\sigma_{k_0}}
        =-\tilde{R}_0(\zeta)((-\Delta_{X}-\zeta^2-\sigma_{k_0}^2)+\zeta^2)\Pi_{\sigma_{k_0}}
        =-(I+\zeta^2\tilde{R}_0(\zeta))\Pi_{\sigma_{k_0}}
    \end{equation}
    and its adjoint, we have by \eqref{eq:zetaRVV} and \eqref{eq:tildeR0zetaV} for $\zeta\in \{\operatorname{Im} \zeta>0, 
    \operatorname{Re} \zeta>0\}$  
    \begin{equation}\label{eq:3.9.32}
        \begin{aligned}
            \zeta &(\tilde{R}_V(\zeta)-\tilde{R}_0(\zeta))
            =-\zeta \tilde{R}_V(\zeta)V\tilde{R}_0(\zeta)\\
            =&-\zeta \tilde{R}_0(\zeta)\rho B(\zeta)V\tilde{R}_0(\zeta) 
            +\tilde{R}_0(\zeta)VA_1V\tilde{R}_0(\zeta)
            -\zeta^{-1}\tilde{R}_0(\zeta)
            V\Pi_{\sigma_{k_0}}V\tilde{R}_0(\zeta) \\
            =&-\zeta \tilde{R}_0(\zeta)\rho B(\zeta)V\tilde{R}_0(\zeta) +\tilde{R}_0(\zeta)VA_1V\tilde{R}_0(\zeta)
            -\zeta^{-1}(I+\zeta^2\tilde{R}_0(\zeta))\Pi_{\sigma_{k_0}}(I+\zeta^2\tilde{R}_0(\zeta))\\
            =&-\zeta \tilde{R}_0(\zeta)\rho B(\zeta)V\tilde{R}_0(\zeta)-\zeta^{-1}\Pi_{\sigma_{k_0}}
            +\tilde{R}_0(\zeta)VA_1V\tilde{R}_0(\zeta)\\
            &-\zeta \tilde{R}_0(\zeta)\Pi_{\sigma_{k_0}}
            -\zeta\Pi_{\sigma_{k_0}}\tilde{R}_0(\zeta)
            -\zeta^3 \tilde{R}_0(\zeta)\Pi_{\sigma_{k_0}}\tilde{R}_0(\zeta)
        \end{aligned}
    \end{equation}

    We transform 
    the integral $I_2$ into the $\zeta$ coordinate 
    \[
        \begin{aligned}
        I_2&=\frac{1}{2\pi i}\int_{\gamma_{\delta_2,\delta_3}} 
        2\zeta\left(\tilde{R}_V(\zeta)(P_V+M)^{-N}-
        \tilde{R}_0(\zeta)(P_0+M)^{-N}\right)e^{-t\zeta^2}
        (1-\tilde{\chi}_{\delta_1}(\zeta^2+\sigma_{k_0}^2)) d\zeta\\
        &+\frac{1}{2\pi i}\int_{-\overline{\gamma_{\delta_2,\delta_3}}} 
        2\zeta\left((\tilde{R}_V(-\bar{\zeta}))^*(P_V+M)^{-N}-
        (\tilde{R}_0(-\bar{\zeta}))^*(P_0+M)^{-N}\right)e^{-t\zeta^2}
        (1-\tilde{\chi}_{\delta_1}(\zeta^2+\sigma_{k_0}^2)) d\zeta\\
        &:=I_2^+ + I_2^{-}
        \end{aligned}
    \]
    where we view $-\overline{\gamma_{\delta_2,\delta_3}}$ 
    as the image of $\gamma_{\delta_2,\delta_3}$ under the reflection across the imaginary axis, 
    with the orientation from left to right.
    We will first consider the integral $I_2^+$, 
    and it will be clear that the integral $I_2^-$
    can be tackled by the same method. 

    Since we want to take the trace, we define the following function 
    \[
    \begin{aligned}
        f(\zeta):=& \zeta\operatorname{tr}\left(\tilde{R}_V(\zeta)(P_V+M)^{-N}-
        \tilde{R}_0(\zeta)(P_0+M)^{-N}\right)\\
        =&\operatorname{tr}\left(\zeta(\tilde{R}_V(\zeta)-\tilde{R}_0(\zeta))(P_0+M)^{-N}
        +\zeta\tilde{R}_V(\zeta)((P_V+M)^{-N}-(P_0+M)^{-N})\right)\\
        =&\operatorname{tr}\left(\zeta(\tilde{R}_V(\zeta)-\tilde{R}_0(\zeta))(P_0+M)^{-N}
        +\zeta\tilde{R}_0(\zeta)((P_V+M)^{-N}-(P_0+M)^{-N})\right)\\
        &+\operatorname{tr}\left(\zeta(\tilde{R}_V(\zeta)-\tilde{R}_0(\zeta))((P_V+M)^{-N}-(P_0+M)^{-N})\right)
    \end{aligned}
    \]
    Using \eqref{eq:3.9.32} together and Laurent expansion of $\tilde{R}_V$, we obtain the following
    decomposition 
    \[  
    \begin{aligned}
    &f(\zeta)=a_1(\zeta)+a_2(\zeta)+b_1(\zeta)+b_2(\zeta)+c(\zeta)\\
    &a_1(\zeta):=-\zeta^{-1} \operatorname{tr} \Pi_{\sigma_{k_0}}(P_V+M)^{-N}\\
    &a_2(\zeta):=\operatorname{tr} \tilde{R}_0(\zeta)VA_1V\tilde{R}_0(\zeta)(P_0+M)^{-N}\\
    &b_1(\zeta):=-\zeta \operatorname{tr} \tilde{R}_0(\zeta)\rho B(\zeta)V\tilde{R}_0(\zeta)(P_0+M)^{-N}\\
    &b_2(\zeta):=-\zeta \operatorname{tr} \left((\tilde{R}_0(\zeta)\Pi_{\sigma_{k_0}}
    +\Pi_{\sigma_{k_0}}\tilde{R}_0(\zeta)
    +\zeta^2 \tilde{R}_0(\zeta)\Pi_{\sigma_{k_0}}\tilde{R}_0(\zeta))(P_0+M)^{-N}\right)\\
    &c(\zeta):=\zeta\operatorname{tr}\left((\zeta^{-1}\Pi_{\sigma_{k_0}}+
    \tilde{R}_0(\zeta)+\tilde{R}_V(\zeta)-\tilde{R}_0(\zeta))\left((P_V+M)^{-N}-(P_0+M)^{-N}\right)\right)\\ 
    &c(\zeta)=c_1(\zeta)+c_2(\zeta)\\
    &c_1(\zeta):= \operatorname{tr}\left((\zeta\tilde{R}_0(\zeta)-
    \tilde{R}_0(\zeta)\rho B(\zeta)V\tilde{R}_0(\zeta)+\tilde{R}_0(\zeta)VA_1V\tilde{R}_0(\zeta))\left((P_V+M)^{-N}-(P_0+M)^{-N}\right)\right)\\
    &c_2(\zeta):=-\zeta \operatorname{tr} \left((\tilde{R}_0(\zeta)\Pi_{\sigma_{k_0}}
    +\Pi_{\sigma_{k_0}}\tilde{R}_0(\zeta)
    +\zeta^2 \tilde{R}_0(\zeta)\Pi_{\sigma_{k_0}}\tilde{R}_0(\zeta))
    ((P_V+M)^{-N}-(P_0+M)^{-N})\right)
    \end{aligned}
    \]
    It's important to keep in mind that all the analysis is carried out
    either for $\zeta$ in a neighborhood of zero, 
    or for $\zeta$ far from both the imaginary axis and the real axis.
    In fact, $f$ behaves well away from these axes, thanks to estimate \eqref{eq:3.9.27}.
\begin{figure}
        \centering
\begin{tikzpicture}
    \tikzset{->-/.style=
    {decoration={markings,mark=at position #1 with 
    {\arrow{latex}}},postaction={decorate}}}
    \tikzset{-<-/.style=
    {decoration={markings,mark=at position #1 with 
    {\arrow{latex reversed}}},postaction={decorate}}}

    \def \del{1};

    \draw[->-=0.5][thick] (\del,1.5) -- (\del,\del);
    \draw[->-=0.5][thick] (\del,\del) -- (\del+2,0.82+\del);
    \draw[->-=0.5][thick][red] (\del,1.5) -- (0,1.5);
    % \draw[->-=0.5][dashed] (0.5,3) -- (0.5,0.5);
    % \draw[->-=0.5][dashed] (0.5,0.5) -- (3,1.5355);
    \draw[->-=0.5][thick][red] (0,1.5) -- (0,0);
    \draw[->-=0.5][thick][red] (0,0) -- (3,1.24);

    \draw[dashed] (0,0) circle(2);

    \draw[dashed] (0,0) -- (3,3);

    \draw[-Latex] (-2.5,0) -- (4,0); 
    \draw[-Latex] (0,-2.5) -- (0,3); 
    
    \node at (-0.5,1) {$\gamma_{0,+\infty}$};
    \node at (-0.2,1.5) {$\delta_3$};
    \node[anchor=west] at (2,2) {$\gamma_{\delta_2,+\infty}$};
    
\end{tikzpicture}
    \caption{Deformation from $\gamma_{\delta_2,+\infty}$ to $\gamma_{0,+\infty}$. 
    $A(\zeta)$ is holomorphic in this dashed circle and in the first quadrant.
    $\Omega_{\delta_2}$ is the region enclosed by the black and the red contours.}
\end{figure}
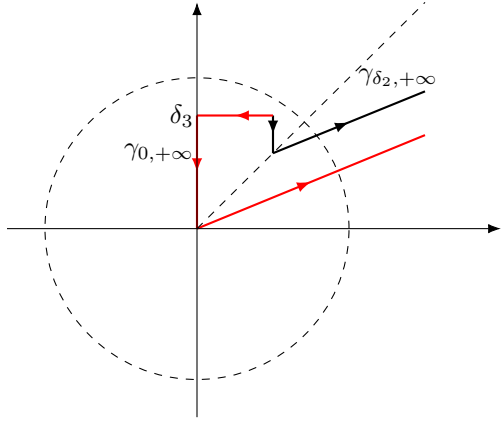

    \textbf{The analysis of $b_1$}. Since $\rho\tilde{R}_0(\zeta)(P_0+M)^{-N}$ is of trace class 
    when $\operatorname{Im} \zeta>0,\operatorname{Re} \zeta>0$, by the cyclic property of trace 
    we have 
    \[
        b_1(\zeta)=-\zeta \operatorname{tr} B(\zeta)V\tilde{R}_0(\zeta)(P_0+M)^{-N}\tilde{R}_0(\zeta)\rho 
    \]
    Let $\Omega_{\delta_2}$ denote the region enclosed by 
    $\gamma_{\delta_2,\delta_3}$, $\gamma_{0,\delta_3}$ and the horizontal segment $\operatorname{Re} z\in [0,\delta_2], \operatorname{Im} z=\delta_3$.
    By estimate \eqref{eq:3.1.26} we have 
    \[
        |b_1(\zeta)|\leq C ||B(\zeta)||_{L^2\to L^2} ||
        \zeta\rho\tilde{R}_0(\zeta)(P_0+M)^{-N}\tilde{R}_0(\zeta)\rho ||_{L^2\to H^{n+\dim M+1}}=\mathcal{O}(1)
    \]
    for $\zeta\in \Omega_{\delta_2}$. 
    We note that by our construction of the contour $\gamma_{\delta_2,\delta_3}$, 
    $\tilde{\chi}_{\delta_1}\equiv 1$ on the horizontal segment $\operatorname{Re} z\in [0,\delta_2], 
    \operatorname{Im} z=\delta_3$.
    Consequently we can deform $\gamma_{\delta_2,\delta_3}$ to $
    \gamma_{0,\delta_3}$ to obtain 
    \[
    \begin{aligned}
        &\int_{\gamma_{\delta_2,\delta_3}} 
        b_1(\zeta) e^{-t\zeta^2}
        (1-\tilde{\chi}_{\delta_1}(\zeta^2+\sigma_{k_0}^2)) d\zeta\\
        =&\int_{\gamma_{0,\delta_3}} 
        b_1(\zeta) e^{-t\zeta^2}
        (1-\tilde{\chi}_{\delta_1}(\zeta^2+\sigma_{k_0}^2)) d\zeta
        +2i\int_{\Omega_{\delta_2}} b_1(\zeta) e^{-t\zeta^2}
        \partial_{\bar{\zeta}}(\tilde{\chi}_{\delta_1}(\zeta^2+\sigma_{k_0}^2)) dm(\zeta)\\
    \end{aligned}
    \]
    Now we observe that, in the 
    first integral, for $\zeta\in i\mathbb{R}$ in the support of $1-\chi_{\delta_1}(\zeta^2+\sigma_{k_0}^2)$ 
    we must have $|\zeta|^2\leq 2\delta_1$; 
    and in the second integral, for $\zeta\in \Omega_{\delta_2}$ lying in the support of $d\tilde{\chi}(\zeta^2+\sigma_{k_0}^2)$
    we must have $\operatorname{Im} \zeta^2=\mathcal{O}(\delta_2)$. This implies that 
    \[
    \begin{aligned}
        &\int_{\gamma_{\delta_2,\delta_3}} 
        b_1(\zeta) e^{-t\zeta^2}
        (1-\tilde{\chi}_{\delta_1}(\zeta^2+\sigma_{k_0}^2)) d\zeta\\
        =&\int_{e^{i\frac{\pi}{8}}[0,\infty)} b_1(\zeta)
        e^{-t\zeta^2} d\zeta+\mathcal{O}(\sqrt{\delta_1} e^{3t\delta_1})+\mathcal{O}_{\delta_1,t}(\delta_2^{\tilde{N}})\\
        =&\mathcal{O}(t^{-1/2})+\mathcal{O}(\sqrt{\delta_1} e^{3t\delta_1})+\mathcal{O}_{\delta_1,t}(\delta_2^{\tilde{N}})
    \end{aligned}
    \]

    \textbf{The analysis of $b_2$.} 
    We write the projection $\Pi_{\sigma_{k_0}}=\sum_{j=1}^J u_j\otimes \bar{u}_j$, 
    and express $u_j$ as the Fourier expansion with respect to $\varphi_k$
    \[
        u_j(x,y)=\sum_{k=0}^\infty u_{jk}(x)\otimes \varphi_k(y)
    \]
    Recall remark \ref{rmk:Eigenspace lies in the range of R0}, 
    we know $u_j\in R_0(L^2_{\operatorname{comp}})\cap L^2(\mathbb{R}^n)$.
    To compute the trace, we decompose $L^2(X)$ into three subspaces 
    \[
    \begin{aligned}
        L^2(X)=\mathcal{H}_-\oplus\mathcal{H}_0\oplus \mathcal{H}_+\\
        \mathcal{H}_-:=L^2(\mathbb{R}^n,\oplus_{\sigma_j<\sigma_{k_0}} \mathbb{C}\varphi_j)\\
        \mathcal{H}_0:=L^2(\mathbb{R}^n,\oplus_{\sigma_j=\sigma_{k_0}} \mathbb{C}\varphi_j)\\
        \mathcal{H}_+:=L^2(\mathbb{R}^n,\oplus_{\sigma_j>\sigma_{k_0}} \mathbb{C}\varphi_j)
    \end{aligned}
    \]
    \begin{itemize}
        \item The analysis in $\mathcal{H}_-$. We note that $u_{jk}$ is compactly supported for any $\sigma_k<\sigma_{k_0}$ 
        by Rellich uniqueness theorem. Thus, 
        choosing $\rho_1\in C_c^\infty(\mathbb{R}^n)$ which equals one in a sufficiently 
        large set, we have
        \[
        \begin{aligned}
            &\operatorname{tr}_{\mathcal{H}_-} \left((\tilde{R}_0(\zeta)\Pi_{\sigma_{k_0}}
        +\Pi_{\sigma_{k_0}}\tilde{R}_0(\zeta)
        +\zeta^2 \tilde{R}_0(\zeta)\Pi_{\sigma_{k_0}}\tilde{R}_0(\zeta))(P_0+M)^{-N}\right)\\
        =&\operatorname{tr}_{\mathcal{H}_-}\left(
        \rho_1(P_0+M)^{-N}\tilde{R}_0(\zeta)\rho_1\Pi_{\sigma_{k_0}}\right)+\operatorname{tr}_{\mathcal{H}_-}\left(\Pi_{\sigma_{k_0}}
        \rho_1\tilde{R}_0(\zeta)(P_0+M)^{-N}\rho_1\right)\\
        &+\operatorname{tr}_{\mathcal{H}_-}\left(\zeta^2\Pi_{\sigma_{k_0}}
        \rho_1\tilde{R}_0(\zeta)(P_0+M)^{-N}\tilde{R}_0(\zeta)\rho_1\right)
        \end{aligned}
        \]
        Note that for large $r>0$, the weighted estimates 
        \eqref{eq:Weighted estimate for RV near threshold} 
        and \eqref{eq:3.9.40} imply
        \[
            ||\rho_1(P_0+M)^{-N}\tilde{R}_0(\zeta)\rho_1||_{L^2\to L^2}
            \leq ||\rho_1(P_0+M)^{-N}\langle x\rangle^r||_{L^2\to L^2}
            ||\langle x\rangle^{-r}|\tilde{R}_0(\zeta)\rho_1||=\mathcal{O}(1)
        \]
        The first and the second terms are $\mathcal{O}(1)$, and the third term is also $\mathcal{O}(1)$ by 
        estimate \eqref{eq:3.1.26}. 
        \item The analysis in $\mathcal{H}_+$. Since 
        \[
            \tilde{R}_0(\zeta)=\mathcal{O}(1)_{\mathcal{H}_+\to \mathcal{H}_+}
        \]
        we know the trace in $\mathcal{H}_+$ 
        is of $\mathcal{O}(1)$ since $\Pi_{\sigma_{k_0}}$ 
        is of finite-rank.
        \item The analysis in $\mathcal{H}_0$. Since 
        $u_j\in R_0(L^2_{\operatorname{comp}})\cap L^2$, 
        and by the kernel of the free-resolvent in $\mathbb{R}^n$ at zero 
        together with the first part of Proposition \ref{prop:Resolvent near threshold when n dengyu three}, we have 
        \[
            u_{jk}(x)=
            \left\{ 
                \begin{aligned}
                    &\mathcal{O}(\langle x\rangle^{2-n})\quad n\geq 5,\sigma_k=\sigma_{k_0}, |x|\gg1\\
                    &\mathcal{O}(\langle x\rangle^{-2})\quad n=3,\sigma_k=\sigma_{k_0}, |x|\gg1
                \end{aligned}
            \right.
        \]
        Thus $u_{jk}(x)\in L^p$ for some $p<2$.
        In the region $\zeta\in \Omega_{\delta_2}$ we have $
            \operatorname{Im}{\zeta}\geq \frac{1}{10} \operatorname{Re}{\zeta}
        $, 
        while the operator $\tilde{R}_0(\zeta):\mathcal{H}_0\to \mathcal{H}_0$ behaves 
        like the convolution with the function $g(x)$ defined by 
        \[
            g(x)=\frac{e^{i\zeta x}}{|x|^{n-2}},\quad x\in \mathbb{R}^n
        \]
        By direct computation in spherical coordinate we obtain
        \[ 
            ||g||_{L^q}=\mathcal{O}((\operatorname{Im} \zeta)^{-2+n(q-1)/q}),\quad 1\leq q<\frac{n}{n-2}
        \]
        By Young's inequality 
        on convolution, we can choose $q$ a little larger than $1$, 
        and then choose $p$ a litte smaller than $2$ so that 
        \[
            \frac{3}{2}=\frac{1}{p}+\frac{1}{q}
        \] 
        we then know for some $\delta>0$
        \begin{equation}\label{eq:Convolution estimate for Ro in H0 space}
            ||\tilde{R_0}(\zeta)\Pi_{\sigma_{k_0}}||_{L^2(X)\to \mathcal{H}_0}=\mathcal{O}(|\zeta|^{-2+\delta}),\quad \zeta\in \Omega_{\delta_2}
        \end{equation}
    \end{itemize}
    In summary we obtain for some $\delta>0$
    \[
        |b_2(\zeta)|=\mathcal{O}(|\zeta|^{-1+\delta}),\quad \zeta\in \Omega_{\delta_2}
    \]
    Thus, we can deform $\gamma_{\delta_2,\delta_3}$ to $
    \gamma_{0,\delta_3}$. Since $|\zeta|^2\geq \operatorname{Im} \zeta^2$, and 
    using the support condition of $\tilde{\chi}$ as in the analysis of $a_1$, we obtain
    \[
    \begin{aligned}
        &\int_{\gamma_{\delta_2,\delta_3}} 
        b_2(\zeta) e^{-t\zeta^2}
        (1-\tilde{\chi}_{\delta_1}(\zeta^2+\sigma_{k_0}^2)) d\zeta\\
        =&\int_{\gamma_{0,\delta_3}} 
        b_2(\zeta) e^{-t\zeta^2}
        (1-\tilde{\chi}_{\delta_1}(\zeta^2+\sigma_{k_0}^2)) d\zeta
        +2i\int_{\Omega_{\delta_2}} b_2(\zeta) e^{-t\zeta^2}
        \partial_{\bar{\zeta}}(\tilde{\chi}_{\delta_1}(\zeta^2+\sigma_{k_0}^2)) dm(\zeta)\\
        =&\int_{e^{i\frac{\pi}{8}}[0,\infty)} b_2(\zeta)
        e^{-t\zeta^2} d\zeta+
        \int_{0}^\infty \mathcal{O}(|s|^{-1+\delta})e^{ts^2}(1-\chi_{\delta_1}(\sigma_{k_0}^2-s^2))
        ds
        +
        \mathcal{O}_{\delta_1,t}(\delta_2^{\tilde{N}-1})\\
        =&\int_{e^{i\frac{\pi}{8}}[0,\infty)} \mathcal{O}(|\zeta|^{-1+\delta})
        e^{-t\zeta^2} d\zeta+
        \mathcal{O}(e^{2\delta_1 t}) \int_{0}^{\sqrt{2\delta_1}} \mathcal{O}(|s|^{-1+\delta})
        ds+\mathcal{O}_{\delta_1,t}(\delta_2^{\tilde{N}-1})\\
        =&\mathcal{O}(t^{-\delta/2})+\mathcal{O}(e^{2\delta_1t}\frac{(\delta_1)^{\delta/2}}{\delta})+
        \mathcal{O}_{\delta_1,t}(\delta_2^{\tilde{N}-1})
    \end{aligned}
    \]

    \textbf{The analysis of $c_1$.}
    We can rewrite $c_1$ using the resolvent identity inductively as 
    \[
        c_1(\zeta)=
        \sum_{k=1}^N  \operatorname{tr}\left(T(\zeta)(P_V+M)^{-N+k-1}V(P_0+M)^{-k}\right)
    \]
    where 
    \[
        T(\zeta):=\zeta\tilde{R}_0(\zeta)-\zeta\tilde{R}_0(\zeta)\rho B(\zeta)V\tilde{R}_0(\zeta)
        +\tilde{R}_0(\zeta)VA_1V\tilde{R}_0(\zeta)
    \]
    Then the weighted estimate \eqref{eq:Weighted estimate for RV near threshold} implies that 
    for $r>\frac{n-1}{2}$ 
    \[
        ||\langle x\rangle^{-r} T(\zeta) \langle x\rangle^{-r}||_{L^2(X)\to L^2(X)}=\mathcal{O}(1),\quad \operatorname{Re} \zeta\geq 0,
        \operatorname{Im} \zeta\geq 0, |\zeta|\leq 10
    \]
    By applying estimate \eqref{eq:3.9.40} with the larger of $k$ and $N-k+1$, we obtain
    \[
        s_j(\langle x\rangle^{r}(P_V+M)^{-N+k-1}V(P_0+M)^{-k}\langle x\rangle^{r})
        \leq Cj^{-(N-1)/2(n+\operatorname{dim} M)}
    \]
    Hence, for $r>\frac{n-1}{2}$ and $N>2(n+\operatorname{dim} M)+10$ we have 
    \[
    \begin{aligned}
        &||T(\zeta)\left((P_V+M)^{-N}-(P_0+M)^{-N}\right)||_{\mathcal{L}_1(\langle x\rangle^{r}L^2(X))}\\
        \leq& ||\langle x\rangle^{-r} T(\zeta) \langle x\rangle^{-r}||_{L^2(X)\to L^2(X)}
        ||\langle x\rangle^{r}\left((P_V+M)^{-N}-(P_0+M)^{-N}\right) \langle x\rangle^{r}||_
        {\mathcal{L}_1(L^2(X))}=\mathcal{O}(1)
    \end{aligned}
    \]
    for $\zeta\in \Omega_{\delta_2}$.
    Therefore, the same proof of \cite[Lemma B.33]{MathematicalTheoryofScatteringResonances} implies that for $\zeta\in \Omega_{\delta_2} $
    \[
        c_1(\zeta)=\operatorname{tr}_{\langle x\rangle^{r}L^2(X)} T(\zeta)\left((P_V+M)^{-N}-(P_0+M)^{-N}\right)
        =\mathcal{O}(1)
    \]
    Hence, we can deform $\gamma_{\delta_2,\delta_3}$ into $\gamma_{0,\delta_3}$
    and proceed as in the analysis of $b_1$.

    \textbf{The analysis of $c_2$}. This case is almost the same as $b_2$. In fact, we can first 
    factorize $(P_V+M)^{-N}-(P_0+M)^{-N}$ into a sum of terms of the form $(P_V+M)^{-k}V(P_0+M)^{-k+N-1}$.
    By decomposing into the three subspaces $\mathcal{H}_0,\mathcal{H}_{-},\mathcal{H}_+$, we can still 
    use the weighted estimate for $(P_V+M)^{-k}$ and the argument 
    in the analysis of $b_2$ to deduce that for any $k\geq 0$
    \[
        ||\Pi_{\sigma_{k_0}}\tilde{R}_0(\zeta)(P_V+M)^{-k}V||_{L^2\to L^2}=\mathcal{O}(\delta^{-2+\delta})
    \]
    For the term $\tilde{R}_0(\zeta)\Pi_{\sigma_{k_0}}$, the contributions from $\mathcal{H}_{0}$ and $\mathcal{H}_+$
    are the same as in the analysis $b_2$; while the contribution from $\mathcal{H}_{-}$ can be still treated
    using the cyclic property together with weighted estimates. Therefore it follows that 
    \[
        |c_2(\zeta)|=\mathcal{O}(|\zeta|^{-1+\delta}),\quad \zeta\in \Omega_{\delta_2}
    \]
    Hence, we deform $\gamma_{\delta_2,\delta_3}$ into $\gamma_{0,\delta_3}$ and proceed as in the analysis of $b_2$.
\begin{figure}
        \centering
\begin{tikzpicture}
    \tikzset{->-/.style=
    {decoration={markings,mark=at position #1 with 
    {\arrow{latex}}},postaction={decorate}}}
    \tikzset{-<-/.style=
    {decoration={markings,mark=at position #1 with 
    {\arrow{latex reversed}}},postaction={decorate}}}

    \def \del{1};

    \draw[->-=0.5][thick] (\del,3) -- (\del,\del);
    \draw[->-=0.5][thick] (\del,\del) -- (\del+2,0.82+\del);
    % \draw[->-=0.5][dashed] (0.5,3) -- (0.5,0.5);
    % \draw[->-=0.5][dashed] (0.5,0.5) -- (3,1.5355);

    \draw[-<-=0.5][thick][red] (-\del,3) -- (-\del,\del);
    \draw[-<-=0.5][thick][red] (-\del,\del) -- (-\del-2,0.82+\del);

    \draw[->-=0.5][thick] (0,3) -- (0,0.5);
    \draw[->-=0.5][thick] (0,0.5) arc(90:0:0.5);
    \draw[->-=0.5][thick] (0.5,0) -- (4,0);

    \draw[-<-=0.5][thick][red] (0,3) -- (0,0.5);
    \draw[-<-=0.5][thick][red] (0,0.5) arc(90:180:0.5);
    \draw[-<-=0.5][thick][red] (-0.5,0) -- (-4,0);

    \draw[dashed] (0,0) -- (3,3);
    \draw[dashed] (0,0) -- (-3,3);

    \draw (-1,0) -- (4,0); 
    \draw (0,-0.5) -- (0,3); 
    
    \node at (-1,0.3) {$-\overline{\tilde{\gamma}_{\delta_4}}$};
    \node at (1,0.3) {$\tilde{\gamma}_{\delta_4}$};
    \node[anchor=west] at (1.0,2.5) {$\gamma_{\delta_2,+\infty}$};
    \node[anchor=east] at (-1.0,2.5) {$-\overline{\gamma_{\delta_2,+\infty}}$};
    \node[anchor=north] at (-0.5,0) {$-\delta_4$};
    \node[anchor=north] at (0.5,0) {$\delta_4$};
\end{tikzpicture}
    \caption{Deformation from $\gamma_{\delta_2,+\infty}$ to $\tilde{\gamma}_{\delta_4}$ and  
    deformation from $-\overline{\gamma_{\delta_2,+\infty}}$ to $-\overline{\tilde{\gamma}_{\delta_4}}$.} 
    $\Omega_{\delta_2,\delta_4}$ is the region enclosed by $\gamma_{\delta_2,+\infty}$ and $\tilde{\gamma}_{\delta_4}$.
\end{figure}

\textbf{The analysis of $a_1$.} Note that the singularity of $a_1(\zeta)$ 
can only occur at $\zeta=0$ for $\zeta\in \mathbb{C}$. By the support condition 
of $\chi_{\delta_1}$, we can extend the integration contour from $\gamma_{\delta_2,\delta_3}$ 
to $\gamma_{\delta_2,+\infty}$. We note that 
\[
    \Pi_{\sigma_{k_0}}(P_V+M)^{-N}=(\sigma_{k_0}^2+M)^{-N}\Pi_{\sigma_{k_0}}
\]
Choose $\delta_4\ll\min(\sqrt{\delta_1},\delta_2)$ sufficiently small.
By deforming $\gamma_{\delta_2,\infty}$ into the new contour $\tilde{\gamma}_{\delta_4}$ defined as 
\[
    \tilde{\gamma}_{\delta_4}:=\left(i[\delta_4,\infty)\right)\bigcup 
    \left(\left\{\delta_4e^{is}:s\in [0,\pi/2]\right\}\right) \bigcup \left([\delta_4,\infty)\right)
\]
oriented from top to bottom and then from left to right, we can obtain the contribution of $a_1$
\[
\begin{aligned}
    &\int_{\gamma_{\delta_2,\infty}} 
    a_1(\zeta) e^{-t\zeta^2}(1-\tilde{\chi}_{\delta_1}(\zeta^2+\sigma_{k_0}^2))
     d\zeta\\
    =&-(\sigma_{k_0}^2+M)^{-N}\operatorname{tr} \Pi_{\sigma_{k_0}}  \int_{
        \delta_4 \exp(i(\pi/2\to 0))
    } 
    \frac{e^{-t\zeta^2}}{\zeta}
     d\zeta\\
    &-(\sigma_{k_0}^2+M)^{-N}\operatorname{tr} \Pi_{\sigma_{k_0}}  \int_{
        i\infty\to i\delta_4
    } 
    \frac{e^{-t\zeta^2}}{\zeta}
    (1-{\chi}_{\delta_1}(\zeta^2+\sigma_{k_0}^2)) d\zeta\\
    &-(\sigma_{k_0}^2+M)^{-N}\operatorname{tr} \Pi_{\sigma_{k_0}}  \int_{
        \delta_4\to \infty
    } 
    \frac{e^{-t\zeta^2}}{\zeta}
    (1-{\chi}_{\delta_1}(\zeta^2+\sigma_{k_0}^2)) d\zeta\\
    &+2i\int_{\tilde{\Omega}_{\delta_2,\delta_4}} a_1(\zeta) e^{-t\zeta^2}
    \partial_{\bar{\zeta}}(\tilde{\chi}_{\delta_1}(\zeta^2+\sigma_{k_0}^2)) dm(\zeta)
\end{aligned}
\]
where $\Omega_{\delta_2,\delta_4}$ is the region enclosed by $\gamma_{\delta_2,+\infty}$ and $\tilde{\gamma}_{\delta_4}$.
The first integral equals
\[
    \frac{\pi i}{2}(\sigma_{k_0}^2+M)^{-N}\operatorname{tr} \Pi_{\sigma_{k_0}}
    +\mathcal{O}_t(\delta_4)
\]
while the last integral is $\mathcal{O}_{\delta_1,t}(\delta_2^{\tilde{N}-1})$.
When computing the integral $I_2^-$, we shall 
decompose the trace of the integrand 
into $a_1,a_2,b_1,b_2,c_1,c_2$ in the same manner. 
Actually we have 
\[
    (R_V(-\bar{\zeta}))^*=-\frac{\Pi_{\sigma_{k_0}}}{\zeta^2}-\frac{A_1^*}{\zeta}+A(-\bar{\zeta})^*
\]
The computations of $b_{1,2},c_{1,2}$ in $I_2^-$ 
are unchanged, 
while the contribution of $a_1$ in $I_2^-$ equals 
\[
\begin{aligned}
    &-(\sigma_{k_0}^2+M)^{-N}\operatorname{tr} \Pi_{\sigma_{k_0}}  \int_{
        \delta_4 \exp(i(\pi\to \frac{\pi}{2}))
    } 
    \frac{e^{-t\zeta^2}}{\zeta}
     d\zeta\\
    &-(\sigma_{k_0}^2+M)^{-N}\operatorname{tr} \Pi_{\sigma_{k_0}}  \int_{
        i\delta_4\to i\infty
    } 
    \frac{e^{-t\zeta^2}}{\zeta}
    (1-{\chi}_{\delta_1}(\zeta^2+\sigma_{k_0}^2)) d\zeta\\
    &-(\sigma_{k_0}^2+M)^{-N}\operatorname{tr} \Pi_{\sigma_{k_0}}  \int_{
        -\infty\to -\delta_4
    } 
    \frac{e^{-t\zeta^2}}{\zeta}
    (1-{\chi}_{\delta_1}(\zeta^2+\sigma_{k_0}^2)) d\zeta\\
    &+\mathcal{O}_{\delta_1,t}(\delta_2^{\tilde{N}-1})
\end{aligned}
\]
Thus the second 
and the third integrals in the $a_1$-term contributions
of $I_2^{\pm}$ cancel, while the two first integrals yields
\[
    \pi i(\sigma_{k_0}^2+M)^{-N}\operatorname{tr} \Pi_{\sigma_{k_0}}
    +\mathcal{O}(\delta_4)
\]

    \textbf{The analysis of $a_2$.} This is the most delicate part. We first consider 
    the case that $A_1=u\otimes v$ with $u,v\in \tilde{H}_{\sigma_k}$, 
    since generally $A_1$ is a sum of such terms. 
    We still set
    \[  
        \begin{aligned}
        L^2(X)=\mathcal{H}_-\oplus\mathcal{H}_0\oplus \mathcal{H}_+\\
        \mathcal{H}_-:=L^2(\mathbb{R}^n,\oplus_{\sigma_j<\sigma_{k_0}} \mathbb{C}\varphi_j)\\
        \mathcal{H}_0:=L^2(\mathbb{R}^n,\oplus_{\sigma_j=\sigma_{k_0}} \mathbb{C}\varphi_j)\\
        \mathcal{H}_+:=L^2(\mathbb{R}^n,\oplus_{\sigma_j>\sigma_{k_0}} \mathbb{C}\varphi_j)
        \end{aligned}
    \]
    as before.
    We note that 
    \begin{itemize}
        \item if $u\in L^2$, then $u\in \operatorname{ran} \Pi_{\sigma_{k_0}}$ and 
        \[
            \tilde{R}_0(\zeta)Vu=-\tilde{R}_0(\zeta)(-\Delta_X-\sigma_{k_0}^2-\zeta^2+\zeta^2)u=
            -u-\zeta^2\tilde{R}_0(\zeta)u
        \]
        \item if $u\notin L^2(X)$, which can only occur when $n=3$, 
        then we can decompose $u=u_- +u_0+u_+$, where 
        \[
        \begin{aligned}
            &u_-\in L^2_{\text{loc}}(\mathbb{R}^n,\oplus_{\sigma_j<\sigma_{k_0}} \mathbb{C}\varphi_j)\\
            &u_0\in L^2_{\text{loc}}(\mathbb{R}^n,\oplus_{\sigma_j=\sigma_{k_0}} \mathbb{C}\varphi_j)\\
            &u_+\in L^2_{\text{loc}}(\mathbb{R}^n,\oplus_{\sigma_j>\sigma_{k_0}} \mathbb{C}\varphi_j)
        \end{aligned}
        \]
        By the characterization of $A_1$ in Proposition \ref{prop:Resolvent near threshold when n dengyu three}, 
        we know $u_-$ is compactly supported, $u_+$ 
        is actually in $L^2(X)$(hence lies in $\mathcal{H}_+$), and $u_0$ is of the form 
        \[
            u_0=\sum_{\sigma_k=\sigma_{k_0}} u_{k}\otimes \varphi_k
        \]
        where $u_k\in R_0^{\mathbb{R}^n}(0)(L^2_{\operatorname{comp}})$ satisfies
        \[
            u_k(x)=\frac{c_k}{-4\pi|x|}+\mathcal{O}(\frac{1}{|x|^2}),\quad |x|\to +\infty
        \]
        for some constant $c_k\in \mathbb{C}$. Note that $-\Delta \frac{1}{-4\pi |x|}=\delta_0$. 
        Thus we obtain
        \[
            \begin{aligned}
                \tilde{R}_0(\zeta)Vu&=-\tilde{R}_0(\zeta)\left((-\Delta_X-\sigma_{k_0}^2-\zeta^2+\zeta^2)u\right)\\
                    &=-(u_+ +u_-)-\zeta^2\tilde{R}_0(\zeta)(u_++u_-)-\sum_{\sigma_k=\sigma_{k_0}} 
                    R_0^{\mathbb{R}^3}(\zeta)(-\Delta_{\mathbb{R}^3}u_k)\otimes \varphi_k
            \end{aligned}
        \]
        We further write 
        \[
            u_k(x):=\frac{c_k}{-4\pi|x|}+w_k^1(x)+w_k^2(x):=\frac{c_k}{-4\pi|x|}+w_k(x)\\
        \] 
        with
        \[
        \begin{aligned}
            &w_k^1\in \mathscr{E}'(\mathbb{R}^3), w_k^2\in H^2(\mathbb{R}^3)\\
            &w_k^1+w_k^2=w_k\in L^2(\mathbb{R}^3),\quad w_k(x)=\mathcal{O}(\langle x\rangle^{-2}),\ x\gg 1
        \end{aligned}
        \]
        Thus 
        \[
            \begin{aligned}
                R_0^{\mathbb{R}^3}(\zeta)(-\Delta_{\mathbb{R}^3}u_k)=&
            c_kR_0^{\mathbb{R}^3}(\zeta)(\delta_0)+R_0^{\mathbb{R}^3}(\zeta)((-\Delta_{\mathbb{R}^3}-\zeta^2+\zeta^2)(w_k^1+w_k^2))\\
            =&c_kR_0^{\mathbb{R}^3}(\zeta)(\delta_0)+w_k+\zeta^2 R_0^{\mathbb{R}^3}(\zeta)(w_k)
            \end{aligned}
        \]
        where we use $R_0(\zeta)(-\Delta-\zeta^2)=\operatorname{Id}$ holds both on $\mathscr{E}'$ and also $H^2$.
        Moreover, by the case $n=3$ of \eqref{eq:Convolution estimate for Ro in H0 space}, 
        we obtain the estimate 
        \begin{equation}\label{eq:analysis of a2 trace on H0 w term}
            ||w_k||_{L^2}=\mathcal{O}(1),\quad 
            ||R_0^{\mathbb{R}^3}(\zeta)(w_k)||_{L^2}=\mathcal{O}(|\zeta|^{-2+\delta})
        \end{equation}
        We will see that $R_0^{\mathbb{R}^3}(\zeta)(\delta_0)$ is the only term that will 
        eventually contribute to the trace formula.
    \end{itemize}
    The computation of the trace on $\mathcal{H}_-$ and $\mathcal{H}_+$ is the same as 
    in the analysis of $b_2$. For the trace on $\mathcal{H}_0$, we have 
    \begin{equation}\label{eq:analysis of a2 trace on H0 main term}
        \begin{aligned}
        &\operatorname{tr}_{\mathcal{H}_0} \tilde{R}_0(\zeta)VA_1V\tilde{R}_0(\zeta)(P_0+M)^{-N}
        =\\
        &\sum_{\sigma_k=\sigma_{k_0}} 
        \operatorname{tr}_{L^2(\mathbb{R}^n)} 
        \left(R_0^{\mathbb{R}^n}(\zeta)(-\Delta_{\mathbb{R}^n})u_k\right)\otimes 
        \left(R_0^{\mathbb{R}^n}(\zeta)(-\Delta_{\mathbb{R}^n})v_k\right)(-\Delta_{\mathbb{R}^n}+\sigma_{k_0}^2+M)^{-N}
        \end{aligned}
    \end{equation}
    Direct calculation shows that 
    \[
        ||R_0^{\mathbb{R}^3}(\zeta)(\delta_0)||_{L^2(\mathbb{R}^3)}=
        \left\|\frac{e^{-\zeta x}}{x}\right\|_{L^2_x(\mathbb{R}^3)}=\mathcal{O}(|\operatorname{Im} \zeta|^{-1/2}), 
        \quad 
    \]
    Moreover, the estimate \eqref{eq:analysis of a2 trace on H0 w term}(which also applies when $u\in L^2$), 
    implies that \eqref{eq:analysis of a2 trace on H0 main term} equals
    \[
        Q(\zeta)+c_kd_k \int \left((-\Delta_{\mathbb{R}^n}+\sigma_{k_0}^2+M)^{-N/2}(R_0^{\mathbb{R}^n}(\zeta)(\delta_0))\right)^2(x) dx
    \]
    for some holomorphic $Q(\zeta)$ which is $\mathcal{O}(|\zeta|^{-1/2})$ for $\zeta \in \Omega_{\delta_2}$.  
    Note that in the case $u\in L^2$ (resp. $v\in L^2$), we simply set $c_k=0$ (resp. $d_k=0$).
    For the integral term, by the Plancherel identity we have(note that it's only nonzero when $n=3$)
\[
\begin{aligned}
    &c_k d_k \int_{\mathbb{R}^n} 
    \Bigl((-\Delta_{\mathbb{R}^n}+\sigma_{k_0}^2+M)^{-N/2}
    (R_0^{\mathbb{R}^n}(\zeta)(\delta_0))\Bigr)^2(x)\, dx \\
    &= c_k d_k \frac{1}{(2\pi)^3} \int_{\mathbb{R}^3} 
    \frac{1}{(|\xi|^2+\sigma_{k_0}^2+M)^N} \frac{1}{(|\xi|^2-\zeta^2)^2}\, d\xi \\
    &= c_k d_k \frac{4\pi}{2(2\pi)^3} \int_{-\infty}^\infty 
    \frac{r^2}{(r-\zeta)^2 (r+\zeta)^2} 
    \frac{1}{(r-i\sqrt{\sigma_{k_0}^2+M})^N} 
    \frac{1}{(r+i\sqrt{\sigma_{k_0}^2+M})^N}\, dr \\
\end{aligned}
\]
Applying Cauchy integral formula to the last integral, we deduce 
\[
\begin{aligned}
    &c_k d_k \int_{\mathbb{R}^n} 
    \Bigl((-\Delta_{\mathbb{R}^n}+\sigma_{k_0}^2+M)^{-N/2}
    (R_0^{\mathbb{R}^n}(\zeta)(\delta_0))\Bigr)^2(x)\, dx \\
    &= c_k d_k \frac{4\pi}{2(2\pi)^3} (2\pi i) \Bigg(
    \left(\frac{d}{dz}\right)_{z=\zeta} \frac{z^2}{(z+\zeta)^2 (z^2+\sigma_{k_0}^2+M)^N} \\
    &\quad + \frac{1}{(N-1)!}\left(\frac{d}{dz}\right)^{N-1}_{z=i\sqrt{\sigma_{k_0}^2+M}} 
    \frac{z^2}{(z^2-\zeta^2)^2 (z+i\sqrt{\sigma_{k_0}^2+M})^N} \Bigg) \\
    &= c_k d_k \frac{4\pi}{2(2\pi)^3} (2\pi i) \frac{1}{4\zeta (\zeta^2 + \sigma_{k_0}^2 + M^2)}
    + \tilde{Q}(\zeta),
\end{aligned}
\]
    where $\tilde{Q}(\zeta)$ is a holomorphic function on $\zeta$ which is bounded for $\zeta\in \Omega_{\delta_2}$.
    So we obtain 
    \[
        \operatorname{tr} \tilde{R}_0(\zeta)V(u\otimes v) V \tilde{R}_0(\zeta)
        (P_0+M)^{-N}=\hat{Q}(\zeta)+\frac{i\sum_{\sigma_k=\sigma_{k_0} }c_kd_k}{8\pi \zeta} 
        \frac{1}{(\sigma_{k_0}^2+M)^N}
    \]
    where $\hat{Q}(\zeta)$ is a holomorphic function and is of $\mathcal{O}(|\zeta|^{-1/2})$ 
    for $\zeta\in \Omega_{\delta_2}$. 
    The inetgration of $\hat{Q}(\zeta)$ term can be calculated by deforming $\gamma_{\delta_2,\delta_3}$ 
    to $\gamma_{0,\delta_3}$ as before. And the integration of the second term can be calculated by deforming 
    $\gamma_{\delta_2,\delta_3}$ to $\tilde{\gamma}_{\delta_4}$. So if 
    we write 
    \[
        A_1=\sum_{j=1}^J u_j\otimes v_j
    \]
    where $u_j,v_j\in \operatorname{ran} A_1$ satisfy,
    for some constants $c_{jk},d_{jk}\in \mathbb{C}$
    \[
        u_j=\sum_{\sigma_k=\sigma_{k_0}} 
        \frac{c_{jk}}{-4\pi |x|}\otimes \varphi_k(y) + L^2(X),\quad 
        v_j=\sum_{\sigma_k=\sigma_{k_0}} 
        \frac{d_{jk}}{-4\pi |x|}\otimes \varphi_k(y) + L^2(X)
    \]
    then by the definition of $\tilde{m}_V(\sigma_k)$ in \eqref{eq:definition of tildemV}
    \[
        \tilde{m}_V(\sigma_k)=\sum_{j=1}^J \sum_{\sigma_k=\sigma_{k_0}} \frac{c_{jk}d_{jk}}{4\pi i}
    \]
    Hence we have 
    \[
    \begin{aligned}
        &\int_{\gamma_{\delta_2,\delta_3}} 
        a_2(\zeta) e^{-t\zeta^2}(1-\tilde{\chi}_{\delta_1}(\zeta^2+\sigma_{k_0}^2))
         d\zeta\\
        =&-(\sigma_{k_0}^2+M)^{-N}\frac{\tilde{m}_V(\sigma_k)}{2}  \int_{
            \delta_4 \exp(i(\pi/2\to 0))
        } 
        \frac{e^{-t\zeta^2}}{\zeta}
         d\zeta\\
        &-(\sigma_{k_0}^2+M)^{-N}\frac{\tilde{m}_V(\sigma_k)}{2}  \int_{
            i\infty\to i\delta_4
        } 
        \frac{e^{-t\zeta^2}}{\zeta}
        (1-{\chi}_{\delta_1}(\zeta^2+\sigma_{k_0}^2)) d\zeta\\
        &-(\sigma_{k_0}^2+M)^{-N}\frac{\tilde{m}_V(\sigma_k)}{2}  \int_{
            \delta_4\to \infty
        } 
        \frac{e^{-t\zeta^2}}{\zeta} d\zeta\\
        &+\mathcal{O}(t^{-\delta/2})+\mathcal{O}(e^{3\delta_1t}\frac{(\delta_1)^{\delta/2}}{\delta})+
        \mathcal{O}_{\delta_1,t}(\delta_2^{\tilde{N}-1})
    \end{aligned}
    \]
    We need to recall that the contribution of the $a_2$ term in the integral $I_2^{-}$
    is given by 
    \[
    \begin{aligned}
        &\int_{\gamma_{\delta_2,\delta_3}} 
        a_2(\zeta) e^{-t\zeta^2}(1-\tilde{\chi}_{\delta_1}(\zeta^2+\sigma_{k_0}^2))
         d\zeta\\
        =&-(\sigma_{k_0}^2+M)^{-N}\frac{\tilde{m}_V(-\sigma_k)}{2}  \int_{
            \delta_4 \exp(i(\pi/2\to 0))
        } 
        \frac{e^{-t\zeta^2}}{\zeta}
         d\zeta\\
        &-(\sigma_{k_0}^2+M)^{-N}\frac{\tilde{m}_V(-\sigma_k)}{2}  \int_{
            i\delta_4\to i\infty
        } 
        \frac{e^{-t\zeta^2}}{\zeta}
        (1-{\chi}_{\delta_1}(\zeta^2+\sigma_{k_0}^2)) d\zeta\\
        &-(\sigma_{k_0}^2+M)^{-N}\frac{\tilde{m}_V(-\sigma_k)}{2}  \int_{
            -\infty\to -\delta_4
        } 
        \frac{e^{-t\zeta^2}}{\zeta} d\zeta\\
        &+\mathcal{O}(t^{-\delta/2})+\mathcal{O}(e^{2\delta_1t}\frac{(\delta_1)^{\delta/2}}{\delta})+
        \mathcal{O}_{\delta_1,t}(\delta_2^{\tilde{N}-1})
    \end{aligned}
    \]
    where $\tilde{m}_V(-\sigma_{k_0})$ is defined via $-A_1^*$ as 
    \[
        \begin{aligned}
           & -A_1^*=-\sum_{j=1}^J \bar{v}_j\otimes \bar{u}_j\\
            &\tilde{m}_V(-\sigma_{k_0}):=
            \sum_{j=1}^J \sum_{\sigma_k=\sigma_{k_0}} \frac{-\bar{c}_{jk}\bar{d}_{jk}}{4\pi i}=\overline{\tilde{m}_V(\sigma_k)}
        \end{aligned}
    \]
    So the sum of the $a_2$ terms in integrals $I_2^+$ and $I_2^-$ is given by 
    \[
    \begin{aligned}
        &\pi i(\sigma_{k_0}^2+M)^{-N} \frac{\operatorname{Re}(\tilde{m}_V(\sigma_{k_0}))}{2}+\\
        &(\sigma_{k_0}^2+M)^{-N} 
        i\operatorname{Im}(\tilde{m}_V(\sigma_{k_0}))\int_{i\delta_4\to i \infty} 
        \frac{e^{-t\zeta^2}}{\zeta}(1-{\chi}_{\delta_1}(\zeta^2+\sigma_{k_0}^2))d\zeta+\\
        &(\sigma_{k_0}^2+M)^{-N} (-i)\operatorname{Im}(\tilde{m}_V(\sigma_{k_0}))
        \int_{\delta_4}^{+\infty} \frac{e^{-t\zeta^2}}{\zeta} d\zeta+\\
        &\mathcal{O}(t^{-\delta/2})+\mathcal{O}(e^{3\delta_1t}\frac{(\delta_1)^{\delta/2}}{\delta})+
        \mathcal{O}_{\delta_1,t}(\delta_2^{\tilde{N}-1})+\mathcal{O}_t(\delta_4)
    \end{aligned}
    \]
    We thus define the following two integrals $J_{1,2}$, which appears as terms with 
    the coefficient $\operatorname{Im}(\tilde{m}_V(\sigma_{k_0}))$
    \[
    \begin{aligned}
        &J_1:=\int_{i\delta_4\to i \infty} 
        \frac{e^{-t\zeta^2}}{\zeta}(1-{\chi}_{\delta_1}(\zeta^2+\sigma_{k_0}^2))d\zeta=
        \int_{\delta_4}^{+\infty} \frac{e^{ts^2}}{s}(1-\chi_{\delta_1}(\sigma_{k_0}^2-s^2))ds\\
        &J_2:=\int_{\delta_4}^{+\infty} \frac{e^{-ts^2}}{s} ds
    \end{aligned}
    \]
    
    Finally, by summing all $a_1,a_2,b_1,b_2,c_1,c_2$ 
    terms from both integrals $I^\pm$, we obtain
    \[
    \begin{aligned}
        \operatorname{tr}(I_2)=&(\sigma_{k_0}^2+M)^{-N}
        \left(\operatorname{tr} \Pi_{\sigma_{k_0}} +\frac{\operatorname{Re}(\tilde{m}_V(\sigma_{k_0}))}{2}\right)\\
        &+
        \frac{(\sigma_{k_0}^2+M)^{-N}}{\pi}\operatorname{Im}(\tilde{m}_V(\sigma_{k_0}))
        (J_1-J_2)\\
        &+\mathcal{O}(t^{-\delta/2})+\mathcal{O}(e^{3\delta_1t}\frac{(\delta_1)^{\delta/2}}{\delta})
        +\mathcal{O}_{\delta_1,t,\tilde{N}}(\delta_2^{\tilde{N}-2})+\mathcal{O}(\delta_4)\\
    \end{aligned}
    \]
    Taking $I_1$ into consideration, we obtain 
    \[
    \begin{aligned}
        &\operatorname{tr}\left(e^{-t(P_V-\sigma_{k_0}^2)}
        (P_V+M)^{-N}-e^{-t(P_0-\sigma_{k_0}^2)}(P_0+M)^{-N}\right)\\
        =&\operatorname{tr}\left(e^{-t(P_V-\sigma_{k_0}^2)}(P_V+M)^{-N}\chi_{\delta_1}(P_V)
        -e^{-t(P_0-\sigma_{k_0}^2)}(P_0+M)^{-N}\chi_{\delta_1}(P_0)\right)\\
        &+(\sigma_{k_0}^2+M)^{-N}\left(\operatorname{tr} \Pi_{\sigma_{k_0}} +\frac{\operatorname{Re}(\tilde{m}_V(\sigma_{k_0}))}{2}\right)\\
        &+
        \mathcal{O}(t^{-\delta/2})+\mathcal{O}(e^{3\delta_1t}\frac{(\delta_1)^{\delta/2}}{\delta})
        +\mathcal{O}_{\delta_1,t,\tilde{N}}(\delta_2^{\tilde{N}-2})+\mathcal{O}_t(\delta_4)\\
        &+\frac{(\sigma_{k_0}^2+M)^{-N}}{\pi}\operatorname{Im}(\tilde{m}_V(\sigma_{k_0}))
        (J_1-J_2)
    \end{aligned}
    \]
    We note that both $J_1$ and $J_2$ are real-valued, and we have 
    \[
        J_2-J_1\geq -\int_{\delta_4}^{\delta_1}
        \frac{e^{ts^2}-e^{-ts^2}}{s} ds+
        \int_{\delta_1}^{+\infty} \frac{e^{-ts^2}}{s}ds\geq C(t)(\ln(1/\delta_1)-\delta_1)
    \]
    for some positive constant $C(t)$ depending on $t$.
    By the assumption of the lemma, 
    we know
    \[
    \begin{aligned}
        &\operatorname{tr}\left(e^{-t(P_V-\sigma_{k_0}^2)}(P_V+M)^{-N}\chi_{\delta_1}(P_V)
        -e^{-t(P_0-\sigma_{k_0}^2)}(P_0+M)^{-N}\chi_{\delta_1}(P_0)\right)\\
        =&\int_0^{\sigma_{k_0}} e^{t(\sigma_{k_0}^2-\lambda^2)}(\lambda^2+M)^{-N}\chi_{\delta_1}(\lambda^2)
        \operatorname{tr}(S_{\operatorname{nor}}(\lambda)^{-1}\partial_{\lambda}S_{\operatorname{nor}}(\lambda))
        d\lambda \\
        &+ \sum_{E_k\in pp \operatorname{Spec} P_V,
        E_k< \sigma_{k_0}^2} e^{t(\sigma_{k_0}^2-E_k)}(E_k+M)^{-N}
    \end{aligned}
    \]
    For a fixed large $t$, we first pick $\delta_1$ sufficiently small(note that $\operatorname{tr}(S_{\operatorname{nor}}(\lambda)^{-1}\partial_{\lambda}S_{\operatorname{nor}}(\lambda))$ 
    is locally integrable). Then, by letting $\delta_2,\delta_4$
    tend to zero, we conclude that $\operatorname{Im}(\tilde{m}_V(\sigma_{k_0}))$ must be 
    zero since all other terms are bounded. This also completes the proof.
\end{proof}

\section{Upper bound and lower bound of Scattering phase}

We can rewrite the Birman-Krein trace formula in terms of a integration 
with respect to a measure $d\mu$, defined by 
\[
    d\mu(\lambda)=
    \frac{1}{4\pi i}\operatorname{tr}(S_{\operatorname{nor}}(\lambda)^{-1}\partial_{\lambda}S_{\operatorname{nor}}(\lambda))
    \frac{d\lambda}{\sqrt{\lambda}}+\sum_{E_k\in pp \operatorname{ Spec}_{P_V}} \delta_{E_k}+
    \sum_{\sigma\in \{\sigma_k\}_{\geq 0}} \frac{\tilde{m}_V(\sigma)}{2}\delta_{\sigma^2}
\]
so that for $f\in \mathscr{S}(\mathbb{R})$
\[
    \operatorname{tr}(f(P_V)-f(P_0))=\int_{\mathbb{R}} f(\lambda) d\mu(\lambda)
\]
So there is a right-continuous function $\mu$ defined on $\mathbb{R}$, formally defined by 
\[
    \mu(\lambda)=\int_{\mathbb{R}} \mathbf{1}_{(-\infty,\lambda]}(t) d\mu(t)
\]
so that $d\mu(\lambda)$ has $\mu$ as its cumulative distribution function. We will call $\mu$ as 
the \textit{scattering phase}. And we want to know the asymptotic of $\mu(\lambda^2)$ as $\lambda\to \infty$. 
For simplicity we assume the potential $V\in C_c^\infty(X,\mathbb{R})$.

\subsection{The upper bound scattering phase when $M$ is a bounded Euclidean domain}

In this subsection we prove an upper bound for the scattering phase $\mu(\lambda^2)$ when $M\subset \mathbb{R}^{m}$ 
is a bounded Euclidean domain. 
\begin{theorem}\label{thm:upper bound of mu for M Euclidean domain}
    Let $V\in C_c^\infty(X;\mathbb{R})$, and $M\subset \mathbb{R}^m$ be a bounded Euclidean domain. 
    imposed with Dirichlet or Neumann condition.
    Then there exists a constant $C_V>0$ depending on $V$, such that
    \[
        \mu(\lambda^2)\leq C_V\lambda^{n+m-1},\quad \lambda\geq 1
    \]
\end{theorem}

Actually, we assume $M$ is an $m$-dimensional compact manifold with boundary, imposed with Dirichlet or Neumann condition. 
We assume further there exists a first-order differntial operator $A_M$ defined in $M$, so that 
\[
    [A_M,\Delta_M]f=\Delta_Mf
\]
for all $C^3$ functions $f$.

\begin{remark}
    If $M$ has no boundaries, then any operator $A:C_c^\infty(M)\to \mathscr{D}'(M)$ must NOT satisfy $
    [A_M,\Delta_M]=\Delta_M$. Actually, let $\varphi_j\in C^\infty(M)$ with $-\Delta \varphi_j=\sigma_j^2\varphi_j$ 
    be an eigenfunctions with eigenvalues $\sigma_j^2\neq 0$, then we must have 
    \[
        \langle [A_M,\Delta_M]\varphi_j,\varphi_j \rangle_{L^2(M)}=
        \langle A_M\Delta_M\varphi_j,\varphi_j\rangle
        - \langle A_M\varphi_j,\Delta_M\varphi_j\rangle0\neq \langle \Delta_M\varphi_j,\varphi_j \rangle_{L^2(M)}
    \]
    However, when $M$ has boundaries, the integration by parts argument does not hold 
    since the operator $A_M$ can change the boundary behaviour of $\varphi_j$. So now it's possible to find such $A$.
\end{remark}

The most interesting(And I doubt this is the only case in which such $A_M$ exists) case 
is that $M$ is a bounded Euclidean domain lying in $\mathbb{R}^{m}_y$, so the operator $A_M$ 
can be chosen to be 
\[
    A_M=-\sum_{j=1}^m y_j\partial_{y_j}
\]
Also we define a first-order differential operator
\[
    A_0=-\sum_{j=1}^n x_j\partial_{x_j}
\]
in $\mathbb{R}^n_x$, and $A=A_0+A_M$ be a first-order differential operator defined in $X$
so that $[A,\Delta_X]=\Delta_X$. The following elegant commutator argument due to Robert \cite[Theorem 3.1]{10.1007/978-1-4612-0775-7_18}, allows us to reduce the trace of $f(P_V)-f(P_0)$ into the trace of 
those operators with compact support. We adapt the argument from \cite{1998Spectral} to the case with non-empty boundary.
\begin{lemma}\label{lem:Robert's commutator argument}
    Let $f\in \mathscr{S}(\mathbb{R})$ and $f(0)=0$. 
    % Let $A_0=-\sum_{j=1}^n x_j\partial_{x_j}$ be 
    % a first-order differential operator defined in $\mathbb{R}^n_x$, and $A=A_0+A_M$ 
    % be a first-order differential operator defined in $X$. Then $[A,\Delta_X]=\Delta_X$
    Then 
    for $\chi\in C_c^\infty(\mathbb{R}^n)$ so that $\chi=1$ in a neighborhood of $\operatorname{supp} V$,
    we have
    \begin{equation}\label{eq:Robert's commutator argument}
        \operatorname{tr}\left((1-\chi)\left(f(P_V)-f(P_0)\right)\right)=
        \operatorname{tr}\left([\chi,P_0]AP_V^{-1}f(P_V)\right)-
        \operatorname{tr}\left([\chi,P_0]AP_0^{-1}f(P_0)\right)
    \end{equation}
\end{lemma}

\begin{proof}
    We first choose a cutoff $\eta\in C_c^\infty(B_{\mathbb{R}^n}(0,2))$ with $\eta=1$ in $B_{\mathbb{R}^n}(0,1)$, 
    and let $\eta_R(x):=\eta(x/R)$ defined in $\mathbb{R}^n$. We note that $1-\eta_R$ converges to zero in strong operator toplogy 
    in $\mathcal{L}(L^2(X))$, and the following elementary fact in functional analysis 
    \[
        (1-\eta_R)B\to 0 \text{ in }\mathcal{L}_1(L^2(X)),\quad \forall B\in \mathcal{L}_1(L^2(X))
    \]
    so we can write the left side of \eqref{eq:Robert's commutator argument} as 
    \[
        \lim_{R\to +\infty} \operatorname{tr}\left((1-\chi)(1-\eta_R)(f(P_V)-f(P_0))\right)=0
    \]
    With the help of cutoff $\eta_R$, we know $\eta_Rf(P_V)$ is of trace-class for any bounded function $f$ with rapid decay at $+\infty$. 
    So we have 
    \begin{equation}\label{eq:equation 1 in Robert's commutator argument}
        \begin{aligned}
            \operatorname{tr}\left((1-\chi)\eta_R(f(P_V)-f(P_0))\right)
            =&\operatorname{tr}\left((1-\chi)\eta_RP_0(P_V^{-1}f(P_V)-P_0^{-1}f(P_0))\right)\\
            =&\operatorname{tr}\left((1-\chi)\eta_R[A,P_0](P_V^{-1}f(P_V)-P_0^{-1}f(P_0))\right)\\
            =&\operatorname{tr}\left((1-\chi)\eta_RA(f(P_V)-f(P_0))\right)-
            \\ &\operatorname{tr}\left(P_0(1-\chi)\eta_R A(P_V^{-1}f(P_V)-P_0^{-1}f(P_0))\right)+
            \\ &\operatorname{tr}\left([\chi\eta_R,-\Delta_{\mathbb{R}^n}] A(P_V^{-1}f(P_V)-P_0^{-1}f(P_0))\right)
        \end{aligned}
    \end{equation}
    We can rewrite the second term into
    \[
    \begin{aligned}
        &\operatorname{tr}\left(P_0(1-\chi)\eta_R A(P_V^{-1}f(P_V)-P_0^{-1}f(P_0))\right)\\
        =&
        \operatorname{tr}\left(P_V(1-\chi)\eta_R AP_V^{-1}f(P_V)\right)-
        \operatorname{tr}\left(P_0(1-\chi)\eta_R AP_0^{-1}f(P_0)\right)
    \end{aligned}
    \]
    Using that $\eta_R(P_V)$ property of functional calculus 
    and cyclity of the trace we know
    \begin{equation}\label{eq:equation 2 in Robert's commutator argument}
        \begin{aligned}
        \operatorname{tr}\left(P_V(1-\chi)\eta_R AP_V^{-1}f(P_V)\right)&=
        \operatorname{tr}\left(P_V(P_V+i)^{-1}(P_V+i)(1-\chi)\eta_R AP_V^{-1}f(P_V)\right)\\
        &=\lim_{t\to +\infty} 
        \operatorname{tr}\left(\eta_t(P_V)(P_V+i)^{-1}(P_V+i)(1-\chi)\eta_R AP_V^{-1}f(P_V)\right)\\
        &=\lim_{t\to +\infty} 
        \operatorname{tr}\left((1-\chi)\eta_R AP_V^{-1}f(P_V)\eta_t(P_V)\right)\\
        &=\operatorname{tr}\left((1-\chi)\eta_R AP_V^{-1}f(P_V)\right)
        \end{aligned}
    \end{equation}
    where in the second equality we use the fact that $\eta_t(P_V)$ 
    converges to $P_V$ in strong toplogy $\mathcal{L}(H^2,L^2)$.
    We note it cancels the first term of the right hand of the last equality in \eqref{eq:equation 1 in Robert's commutator argument}.
    So we obtain 
    \[
        \operatorname{tr}\left((1-\chi)\eta_R(f(P_V)-f(P_0))\right)=
        \operatorname{tr}\left([\chi\eta_R,-\Delta_{\mathbb{R}^n}] A(P_V^{-1}f(P_V)-P_0^{-1}f(P_0))\right)
    \]
    Letting $R$ tends to zero, we see $\chi\eta_R=\chi$, so the proof is complete.
\end{proof}

There is another elegant lemma due to T.Christiansen, which actually essentially follows from Hormander\cite{hormander1968spectral}, 
allows us to compare the trace of the cutoff spectral projections of two operators which coincide in a neighborhood 
of the support of the cutoff function. We will present a simple version here, sufficient for our application.
\begin{lemma}\label{lem:Trace Comparision}
    Assume one of the following two cases
    \begin{itemize}
        \item $M_1$ and $M_2$ be two Riemannian manifolds with boundary, and $U$ 
        is an open set of both $M_1$ and $M_2$ which is bounded, $\chi\in C_c^\infty(U)$. 
        Let $P_j=-\Delta_{M_j}+V$ be self-adjoint operators on $L^2(M_j)$ 
        whose domain $\mathcal{D}(P_j)$ is a subset of $H^2_{\text{loc}}(M_j)$, with Dirichlet or Neumann boundary condition, with $E_j(\lambda)$ as the spectral projection, 
        where $j=1,2$, and $V\in C_c^\infty(U,\mathbb{R})$. 
        \item $X=\mathbb{R}^n\times M$ as our product setting and $P_j=-\Delta_{X}+V_j$ with Dirichlet or Neumann boundary condition, 
        with $E_j(\lambda)$ as the spectral projection, where $j=1,2$ and $V_j\in C_c^\infty(X;\mathbb{R})$. 
        Let $\chi\in C_c^\infty(X)$ with support disjoint of 
        $\operatorname{supp}(V_1)\cup \operatorname{supp}(V_2)$.
    \end{itemize}

    If we assume in addition that for some $d>0$ and $k\in \mathbb{N}_0$ and $A$ is a differntial operator so that
    \[
        \operatorname{tr}(\chi AP_1^{-k}(E_1((\lambda+1)^2)-E_1(\lambda^2)))=\mathcal{O}(\lambda^d)
    \]
    And that the function 
    \[
        \lambda\mapsto \operatorname{tr}(\chi AP_1^{-k}(E_j(\lambda^2)-E_j(1)))
    \]
    is increasing for $\lambda \geq 1$ and $j=1,2$, then we have
    \[
        |\operatorname{tr}(\chi BP_1^{-k}(E_1(\lambda^2)-E_j(1)))-\operatorname{tr}(\chi AP_2^{-k}(E_2(\lambda^2)-E_2(1)))|=\mathcal{O}(\lambda^{d})
    \]
    % If we assume in addtion for some $d>0$
    % \[
        
    % \]
    % then in the first case we have 
    % \[
    %     \operatorname{tr}(\chi(E_1(\lambda^2)))-\operatorname{tr}(\chi(E_2(\lambda^2)))=\mathcal{O}(\lambda^d)
    % \]
    % while in the second case we have
    % \[
    %     \operatorname{tr}(\chi B(E_1(\lambda^2)-E_2(\lambda^2)))=\mathcal{O}(\lambda^d)
    % \]
\end{lemma}

\begin{proof}
    Let $\tilde{\chi}\in C_c^\infty$ 
    equals to one near $\operatorname{supp} \chi$, so that 
    the support of $\chi$ lies inside $U$ in the first case, 
    or its disjoint from $\operatorname{supp}(V_1)\cup \operatorname{supp}(V_2)$ in the second case. 
    We consider $U_j(t):=\cos{t\sqrt{P_j}}$ as the functional calculus, where we choose $\sqrt{-t}=i\sqrt{t}$ for $t\geq 0$.
    Then for $u\in \mathcal{D}(P_j)$, using spectral theorem to view $P_j$ as a multiplication 
    operator on some $L^2$ space, we see $U_j(t)u\in \mathcal{D}(P_j)$ satisfies the following wave equation 
    with Dirichlet or Neumann boundary condition
    \[
        \left\{
            \begin{aligned}
            &(\partial_t^2+P_j)(U_j(t)u)=0\\
            &U_j(0)u=u\\
            &\frac{d}{dt}|_{t=0}(U_j(t)u)=0
            \end{aligned}
        \right.
    \]
    So by uniqueness and finite propogation speed of wave equation, 
    we know for either cases, there exists some $\delta>0$ so that for $|t|<\delta$ we have 
    \begin{equation}\label{eq:Wave kernel of P1 and P2 is the same in chi}
        \left(\cos{t\sqrt{P_1}}-\cos{t\sqrt{P_2}}\right)\tilde{\chi}=0
    \end{equation}

    Next we define a right-continuous function $g_j(\lambda)$ as 
    \[
        g_j(\lambda):=\operatorname{tr}(\chi AP_j^{-k}(E_j(\lambda^2)-E_j(1)))=\operatorname{tr}(\chi AP_j^{-k}(E_j(\lambda^2)-E_j(1))\tilde{\chi})
    \]
    for $\lambda\geq 1$, while $g_j(\lambda)=-g_j(-\lambda)$ for $\lambda\leq -1$ 
    and $g_j(\lambda)=0$ for $-1\leq \lambda \leq 1$.
    Then $g_j$ has at worst polynomial growth, which induces a tempered, even positive measure $dg_j$. Then $dg_j$ is actually equals to some 
    even $T_j\in \mathscr{S}'(\mathbb{R})$ defined by 
    \[
        T_j(f):=\operatorname{tr}(\chi A \left(\left(\frac{\mathbf{1}_{(1,+\infty)}(\bullet)f(\bullet)}{\bullet^{2k}}\right)\left(\sqrt{P_j}\right)+
        \left(\frac{\mathbf{1}_{(-\infty,1)}(\bullet)f(\bullet)}{\bullet^{2k}}\right)\left(-\sqrt{P_j}\right)\right)\tilde{\chi})
    \]
    Thus the Fourier transform of $\lambda^{2k}dg_j(\lambda)$ is given by for $f\in C_c^\infty(\mathbb{R})$
    \[
        \begin{aligned}
            &\langle \mathcal{F}(\lambda^{2k}{dg_j}),f\rangle\\
            =&\langle T_j,x^{2k}\hat{f}\rangle\\
            =&
            2\operatorname{tr}(\chi A \int_{\mathbb{R}} f(x)(\cos(x\sqrt{P_j}))\tilde{\chi} dx)-2\int_{\mathbb{R}} f(x)\operatorname{tr}(\chi A \cos(x\sqrt{P_j})E_{j}(1)\tilde{\chi})dx\\
        \end{aligned}
    \]
    We note that the function $\operatorname{tr}(\chi A \cos(x\sqrt{P_j})E_{j}(1)\tilde{\chi})dx$ is smooth for $x\in \mathbb{R}$, so by 
    \eqref{eq:Wave kernel of P1 and P2 is the same in chi} we see 
    \[
        \mathcal{F}({\lambda^{2k}(dg_1-dg_2)})\in C^\infty(-\delta/2,\delta/2)
    \]
    (We note that this holds for $P_1,P_2$ is not defined on the same space, since $L^2(M)=L^2(U)\oplus L^2(M\setminus U)$)
    it follows from the ODE theory of distribution(See Hormander Theorem 3.1.5)
    that, for fixed $\rho\in \mathscr{S}(\mathbb{R})$ with $\hat{\rho}=1$ near zero and $\hat{\rho}\in C_c^\infty((-\delta/2,\delta/2),[0,1])$
    \[
        \rho*(dg_1-dg_2)\in \mathscr{S}(\mathbb{R})
    \]
    We can replace $\rho$ by $\alpha \rho*\rho$ for some 
    positive constant $\alpha$ so that $\int (\rho*\rho)(x)=1/\alpha$, then the desired result follows from the following standard Tauberian lemma \ref{lem:Tauberian lemma}, 
    see for example \cite[Theorem 17.6.8]{hormander2007analysis3}.
\end{proof}

\begin{lemma}\label{lem:Tauberian lemma}
    If $\mu_1,\mu_2$ be two increasing, right-continuous functions with $\mu_1(0)=\mu_2(0)=0$ inducing two tempered measures $d\mu_1,d\mu_2$, respectively.
    Suppose for some positive $\rho\in \mathscr{S}(\mathbb{R},\mathbb{R}_{\geq 0})$ with $\hat{\rho}\in C_c^\infty(\mathbb{R}),\hat{\rho}(0)=1$ 
    and there exists $c_0>0$ with $\rho(x)\geq c_0$ for $x\in [-c_0,c_0]$, we have
    \[
        |\rho*(d\mu_1-d\mu_2)|(x)\leq C_N(1+|x|)^{-N}
    \]
    for any $N\in \mathbb{N}$. If $\mu_1$ satisfies for some $d\geq 0$
    \[
        \mu_1(\lambda+1)-\mu_1(\lambda)=\mathcal{O}(|\lambda|^d)
    \]
    then we have 
    \[
        |\mu_1(\lambda)-\mu_2(\lambda)|=\mathcal{O}(|\lambda|^d)
    \]
\end{lemma}
\begin{proof}
    We first prove that $|\rho*d\mu_1(\lambda)|=\mathcal{O}(|\lambda|^d)$. Actually
    \[
    \begin{aligned}
        |\rho*d\mu_1(\lambda)|&=|\int_{-\infty}^{\infty} \rho(\xi) d\mu_1(\lambda-\xi)|\\
            &\leq \sum_{k=-\infty}^{+\infty} \max_{\xi\in [k,k+1]} |\rho(\xi)| (\mu_1(\lambda-k)-\mu_1(\lambda-k-1))\\
            &=\mathcal{O}(\lambda^d)
    \end{aligned}
    \]
    Then we show $\mu_2(\lambda+1)-\mu_2(\lambda)=\mathcal{O}(|\lambda|^d)$. Actually we have
    \[
    \begin{aligned}
            \mu_2(\lambda+c_0)-\mu_2(\lambda)&\leq C\int_{\mathbb{R}} \rho(\lambda-\xi) d\mu_2(\xi)\\
            &\leq C|\rho*d\mu_1(\lambda)|+C|\rho*(d\mu_1-d\mu_2)(\lambda)|\\
            &=\mathcal{O}(\lambda^d)
    \end{aligned}
    \]
    Next we will show $|\mu_j(\lambda)-\rho*d\mu_j(\lambda)|=\mathcal{O}(|\lambda|^d)$, for $j=1,2$, this will completes the proof. 
    Actually we have
    \[
    \begin{aligned}
        |\mu_j(\lambda)-\rho*\mu_j(\lambda)|&=|\int_{-\infty}^{+\infty}(\mu_j(\lambda)-\mu_j(\lambda-\xi))\rho(\xi)d\xi| \\
        &\leq \int_{-\infty}^{+\infty} C(1+|\lambda|+|\xi|)^d \rho(\xi)d\xi=\mathcal{O}(|\lambda|^d)
    \end{aligned}
    \]
    Finally we note that 
    \[
        (\rho*(\mu_1-\mu_2))(\lambda)
        =(\rho*(\mu_1-\mu_2))(0)\pm \int_0^{\lambda} \rho*(d\mu_1-d\mu_2)
        =\mathcal{O}(1)
    \]
    This completes the proof.
\end{proof}

\begin{proof}[Proof of Theorem \ref{thm:upper bound of mu for M Euclidean domain}]
    The commutator argument shows that 
    \begin{equation}
        \begin{aligned}
            \mu(\lambda^2)-\mu(1)=&\operatorname{tr}(\chi(E_V(\lambda^2)-E_V(1)))-\operatorname{tr}(\chi(E_0(\lambda^2)-E_0(1)))+\\
            &\operatorname{tr}\left([\chi,P_0]AP_V^{-1}(E_V(\lambda^2)-E_V(1))\right)-
            \operatorname{tr}\left([\chi,P_0]AP_0^{-1}(E_0(\lambda^2)-E_0(1))\right)
        \end{aligned}
    \end{equation}
    where we assume $\chi\in C_c^\infty(\mathbb{R}^n,[0,1])$ equals to one in a neighborhood of $\operatorname{supp} V$, 
    and we can further assume $\sqrt{\chi}$ is smooth.

    Choose $R>0$ so that $\operatorname{supp} \chi\subset B(0,R-1)$, let $\mathbb{T}_{R}^n$ be the torus centered at $0\in \mathbb{R}^n$, 
    with side length $2R$. Consider  
    \[
        P_1=-\Delta_X+V,\quad P_2=-\Delta_{\mathbb{T}_{R}^n\times M}+V
    \]
    with Dirichlet or Neumann boundary condition. Then it follows from the comparision lemma 
    \ref{lem:Trace Comparision} and Weyl's law on $P_2$ that we see 
    \[
        \operatorname{tr}(\chi E_V(\lambda^2))=
        \operatorname{tr}(\chi E_{P_2}(\lambda^2))+\mathcal{O}(\lambda^{n+m-1})
    \]
    Using Lemma 
    \ref{lem:Trace Comparision} once again to compare $P_2$ and $-\Delta_{\mathbb{T}_{R}^n\times M}$ we obtain 
    by Weyl's law and the fact that the eigenfunctions on $\mathbb{T}_{R}^n$ are of constant modules
    \[
        \begin{aligned}
            \operatorname{tr}(\chi E_{P_2}(\lambda^2))&=\operatorname{tr}(E_{P_2}(\lambda^2))-
            \operatorname{tr}((1-\chi)E_{P_2}(\lambda^2))\\
            &=\operatorname{tr}(E_{P_2}(\lambda^2))-\operatorname{tr}((1-\chi)E_{-\Delta_{\mathbb{T}_{R}^n\times M}}(\lambda^2))
            +\mathcal{O}(\lambda^{n+m-1})\\
            &=c_{n+m}\lambda^{n+m} \operatorname{vol}(\mathbb{T}_{R}^n\times M)-c_{n+m}\lambda^{n+m}\operatorname{vol}(M)\int_{\mathbb{T}_{R}^n}(1-\chi(x)) dx+\mathcal{O}(\lambda^{n+m-1})\\
            &=c_{n+m}\lambda^{n+m}\operatorname{vol}(M) \int_{\mathbb{R}^n}\chi(x) dx+\mathcal{O}(\lambda^{n+m-1})
        \end{aligned}
    \]
    where $c_{d}=(2\pi)^{-d}\omega_d$ is the Weyl constant, here $\omega_d$ is 
    the volume of unit ball in $\mathbb{R}^d$. On the other hand 
    \[
        \begin{aligned}
            \operatorname{tr}(\chi(E_0(\lambda^2)))&=\sum_{\sigma_k\leq \lambda} \int_{\mathbb{R}^n\times M}(2\pi)^{-n}\int_{\mathbb{R}^n}\chi(x)\mathbf{1}_{(-\infty,\lambda^2-\sigma_k^2]}(|\xi|^2)
            |\varphi(y)|^2
            d\xi dx dy\\
            &=\int_{\mathbb{R}^n} \chi(x) dx \sum_{\sigma_k\leq \lambda} (2\pi)^{-n}\omega_n(\lambda^2-\sigma_k^2)^{n/2}\\
            &=\left(\int_{\mathbb{R}^n} \chi(x) dx\right) \frac{1}{\operatorname{vol}(\mathbb{T}_{R}^n)}  \sum_{\sigma_k\leq \lambda} \left(E_{-\Delta_{\mathbb{T}_{R}^n}}(\lambda^2-\sigma_k^2)+
            \mathcal{O}((\lambda^2-\sigma_k^2)^{n/2-1}+1)\right)\\
            &=\left(\int_{\mathbb{R}^n} \chi(x) dx\right) \frac{1}{\operatorname{vol}(\mathbb{T}_{R}^n)} 
            E_{-\Delta_{\mathbb{T}_{R}^n\times M}}(\lambda^2)+\mathcal{O}(\lambda^{n+m-1})
        \end{aligned}
    \]

    Consider the contribution of 
    \[
        \operatorname{tr}\left([\chi,P_0]AP_V^{-1}(E_V(\lambda^2)-E_V(1))\right)-
        \operatorname{tr}\left([\chi,P_0]AP_0^{-1}(E_0(\lambda^2)-E_0(1))\right)
    \]
    We next use the polarization identity to reduce the trace into a positive function, 
    so that we can apply Tauberian's lemma. That is, 
    if $B_1,B_2$ are bounded operator in $L^2(X)$ so that we there exists some 
    $\rho\in C_c^\infty(\mathbb{R}^n)$ with $\rho B_1=B_1$ and $\rho B_2=B_2$,
    we can define a sesquilinear map $H_\lambda(B_1,B_2):=\operatorname{tr}(B_1(E_V(\lambda^2)-E_V(1))B_2^*)$. 
    Then by polarization identity
    \[
    \begin{aligned}
        H_\lambda(B_1,B_2)=&\frac{1}{4}(H_\lambda(B_1+B_2,B_1+B_2)-H_\lambda(B_1-B_2,B_1-B_2)\\
        &+iH_\lambda(B_1+iB_2,B_1+iB_2)-iH_\lambda(B_1-iB_2,B_1-iB_2))
    \end{aligned}
    \]
    And we note that $H_\lambda(B,B)$ is a increasing function for $\lambda\geq 1$ and any bounded 
    operator $B$ with $\rho B=B$.
    Let $\tilde{\chi}\in C_c^\infty(\mathbb{R}^n)$  equals to one 
    in a neighborhood of the support of $d\chi$, so that the support of $\tilde{\chi}$ 
    is disjoint with $\operatorname{supp} V$. It's tempted to invoke the polarization 
    identity directly, but the adjoint of $A$ is disturbing since now we have a boundary condition. 
    To circumvent this technical difficulty, we consider for some $\epsilon>0$ and $R\gg 1$ 
    \[
        B_{1}^{(R,\epsilon)}=\eta_R(-\Delta_{X,\text{Dirichlet}})[\chi,-\Delta_{\mathbb{R}^n}]A(P_V+i\epsilon)^{-1},\quad B_{2}=\tilde{\chi}
    \]
    where $\eta_R=\eta(x/R)$ defined as in the proof of Lemma \ref{lem:Robert's commutator argument}, and $\eta_R(-\Delta_{X,\text{Dirichlet}})$ 
    means the functional calculus of the self-adjoint operator $-\Delta_X$ with Dirichlet boundary condition. The 
    compact support property of $B_{1}$ is not satisfied, but since $\eta_R(-\Delta)$ is a bounded operator on $L^2$ so the trace class property 
    is preserved, and thus the polarization property still works. Now since 
    the range of $[\chi,-\Delta_{\mathbb{R}^n}]\eta_R(-\Delta_{X,\text{Dirichlet}})$ lies in $H_0^1(X)$, we see that 
    \[
        (B_{1}^{(R,\epsilon)})^*=(P_V-i\epsilon)^{-1}A^*[\chi,-\Delta_{\mathbb{R}^n}]\eta_R(-\Delta_{X,\text{Dirichlet}})
    \]
    where $A^*$ is the formal adjoint of $A$, which is a first-order differential operator defined by 
    \[
        \langle Au,v\rangle_{L^2\times L^2}=\langle u,A^*v\rangle_{L^2\times L^2},\quad u,v\in C_c^\infty(\mathbb{R}^n\times \operatorname{Int} M)
    \]
    So we can apply polarization identity to write $H_\lambda(B_{1}^{(R,\epsilon)},B_2)$ into summation of terms of the following form 
    \begin{equation}\label{eq:one term in polarization identity}
        \frac{\pm\tau}{4}\operatorname{tr}\left(\left((B_{1}^{(R,\epsilon)})^*\pm\bar{\tau}\tilde{\chi}\right)
        \left(B_{1}^{(R,\epsilon)}\pm\tau\tilde{\chi}\right)(E_V(\lambda^2)-E_V(1))\right)
    \end{equation}
    where $\tau$ equals to $1$ or $i$. We can use cyclity to move the first term $(P_V-i\epsilon)^{-1}$ in $(B_{1}^{(R,\epsilon)})^*$
    to the right, so letting $\epsilon$ tends to zero, and then letting $R$ tends to infinity, 
    using the fact that $\eta_R(-\Delta_{X,\text{Dirichlet}})\to \operatorname{id}_{H^s}$ for any $s\geq 0$, we know 
    \eqref{eq:one term in polarization identity} tends to 
    \begin{equation}\label{eq:expansion of one term in polairaztion idenity in limit}
        \begin{aligned}
            &\operatorname{tr}\left(\frac{\pm\tau}{4}[\chi,-\Delta_{\mathbb{R}^n}]A^*A[\chi,-\Delta_{\mathbb{R}^n}]P_V^{-2}\left(E_V(\lambda^2)-E_V(1)\right)\right)
            \pm \\
            &\operatorname{tr}\left(\frac{\pm\tau}{4}\bar{\tau}\tilde{\chi}[\chi,-\Delta_{\mathbb{R}^n}]AP_V^{-1}\left(E_V(\lambda^2)-E_V(1)\right)\right)\mp\\
            &\operatorname{tr}\left(\frac{\pm\tau}{4}A^*[\chi,-\Delta_{\mathbb{R}^n}]\tau\tilde{\chi}P_V^{-1}\left(E_V(\lambda^2)-E_V(1)\right)\right)+\\
            &\operatorname{tr}\left(\frac{\pm\tau}{4}|\tau|^2(\tilde{\chi})^2\left(E_V(\lambda^2)-E_V(1)\right)\right)
            \\
        \end{aligned}
    \end{equation}
    We denote \eqref{eq:expansion of one term in polairaztion idenity in limit} 
    by $\frac{\pm\tau}{4}I_{\pm,\tau,V}(\lambda)$, then 
    $I_{\pm,\tau,V}(\lambda)$ is an increasing function of $\lambda$ 
    for $\lambda\geq 1$.

    The following lemma will show that $I_{\pm,\tau,0}(\lambda+1)-I_{\pm,\tau,0}(\lambda)=\mathcal{O}(\lambda^{n+m-1})$. 
    And the proof of the comparision lemma shows that $I_{\pm,\tau,0}$ and $I_{\pm,\tau,V}$ satisfies the assumption of 
    Tauberian's lemma \ref{lem:Tauberian lemma}, so we have 
    \[
        I_{\pm,\tau,V}(\lambda)-I_{\pm,\tau,0}(\lambda)=\mathcal{O}(\lambda^{n+m-1})
    \]
    which completes the proof.
\end{proof}

\begin{lemma}
    Let $I_{\pm,\tau,V}(\lambda)$ defined as in \eqref{eq:expansion of one term in polairaztion idenity in limit} 
    as above, then we have 
    \[
        I_{\pm,\tau,0}(\lambda+1)-I_{\pm,\tau,0}(\lambda)=\mathcal{O}(\lambda^{n+m-1})
    \]
\end{lemma}

\begin{proof}
    We only consider the term 
    \[
        \operatorname{tr}\left([\chi,-\Delta_{\mathbb{R}^n}]A^*A[\chi,-\Delta_{\mathbb{R}^n}]P_0^{-2}\left(E_0(\lambda^2)-E_0(1)\right)\right)
    \]
    since the other terms is similar, which will be clear from the following proof. 
    We can rewrite it into 
    \[
        \operatorname{tr}\left(QP_0^{-2}\left(E_0(\lambda^2)-E_0(1)\right)\right)
    \]
    where $Q\in \operatorname{Diff}^4(\mathbb{R}^n\times M)$ with coefficients compactly support.
    We note that the kernel of $P_0^{-2}\left(E_0(\lambda^2)-E_0(1)\right)$ is
    in terms of $\mathbb{R}^n\times M \ni(x_2,y_2)\to (x_1,y_1)\in \mathbb{R}^n\times M$
    \[
        \begin{aligned}
        &P_0^{-2}\left(E_0(\lambda^2)-E_0(1)\right)(x_1,y_1,x_2,y_2)\\
        =&\sum_{\sigma_k\leq \lambda} 
        \frac{1}{(2\pi i)^n} \int_{\mathbb{R}^n} e^{i\langle x_1-x_2,\xi\rangle} \frac{1}{(\xi^2+\sigma_k^2)^2}
        \mathbf{1}_{(1-\sigma_k^2,\lambda^2-\sigma_k^2]}(|\xi|^2)d\xi \varphi_k(y_1)\varphi_k(y_2)
        \end{aligned}
    \]
    So if we write $Q=\sum_{\alpha}q_\alpha(x,D,y)\partial_y^\alpha$ in view of pseudodifferential oeprators, we have 
    \[
        \begin{aligned}
            &\operatorname{tr}\left(QP_0^{-2}\left(E_0(\lambda^2)-E_0(1)\right)\right)\\
            =&
            \sum_{\alpha}\sum_{\sigma_k\leq \lambda} 
        \frac{1}{(2\pi i)^n} \int_{M} \iint_{\mathbb{R}^n\times \mathbb{R}^n} 
        q_\alpha(x,\xi,y) \frac{1}{(\xi^2+\sigma_k^2)^2}
        \mathbf{1}_{(1-\sigma_k^2,\lambda^2-\sigma_k^2]}(|\xi|^2) (\partial_y^{\alpha}\varphi_k)(y)\varphi_k(y) d\xi dx dy
        \end{aligned}
    \]
    So the difference between $\lambda+1$ and $\lambda$ can written into two terms $J_1$ and $J_2$, 
    while $J_1$ satisfies
    \[
        J_1\leq C\sum_{l=0}^4 \sum_{\sigma_k\leq \lambda} 
         \int_{\mathbb{R}^n} 
       \frac{(1+|\xi|^2)^{l/2}}{(\xi^2+\sigma_k^2)^2}
        \mathbf{1}_{(\lambda^2-\sigma_k^2,(\lambda+1)^2-\sigma_k^2]}(|\xi|^2) (1+\sigma_k^2)^{2-l/2} d\xi
    \]
    and $J_2$ satisfies
    \[
        J_2\leq C\sum_{l=0}^4\sum_{\lambda<\sigma_k\leq \lambda+1}
        \int_{\mathbb{R}^n} 
        \frac{(1+|\xi|^2)^{l/2}}{(\xi^2+\sigma_k^2)^2}
        \mathbf{1}_{(0,(\lambda+1)^2-\sigma_k^2]}(|\xi|^2) (1+\sigma_k^2)^{2-l/2} d\xi
    \]
    for some constant $C$.

    To estimate $J_1$, we use the inequality
    \[
        \sqrt{(\lambda+1)^2-\sigma_k^2}-\sqrt{\lambda^2-\sigma_k^2}\leq \frac{2\lambda+1}{\sqrt{(\lambda+1)^2-\sigma_k^2}}
    \]
    So $J_1$ has estimate
    \[
    \begin{aligned}
        J_1&\leq C\sum_{\sigma_k\leq \lambda} 
        \left((\lambda+1)^2-\sigma_k^2\right)^{(n-1)/2} \frac{2\lambda+1}{\sqrt{(\lambda+1)^2-\sigma_k^2}}  \\
        &\leq C\lambda
        \sum_{\sigma_k\leq \lambda} 
        \left((\lambda+1)^2-\sigma_k^2\right)^{n/2-1}  \\
        &\leq C\lambda \sum_{j=0}^{\lceil\lambda \rceil} \sharp\{\sigma_k\in \sigma_{pp}(-\Delta_M):j\leq \sigma_k<j+1 \}\int_{j}^{j+1} 
        \left((\lambda+2)^2-s^2\right)^{n/2-1} ds\\
        &\leq C\lambda\int_0^{\lambda+2} (s+1)^{m-1}\left((\lambda+2)^2-s^2\right)^{n/2-1} ds\\
        &\leq C\lambda^{2} \int_0^{1} ((\lambda+2)^{m-1} t^{m-1}+1)(\lambda+2)^{n-2} dt
        =\mathcal{O}(\lambda^{n+m-1})
    \end{aligned}
    \]
    And $J_2$ has estimate 
    \[
        J_2\leq C \lambda^{m-1} (2\lambda+1)^{n/2}
    \]
    This leads to the desired result.
\end{proof}

\subsection{The heat Kernel and a lower bound for the total variation of the Scattering phase}

We first review the heat kernel $E(t,x,y)$ on a compact Riemannian manifold $M$, 
which is the Schwartz kernel of $e^{-t\Delta_{M}}$, when $M$ has boundary we only consider the Dirichlet boundary condition. 
We refer to the note \cite{Grieser2004NOTESOH}.

We first review the case when $M$ has no boundary.
\begin{proposition}\label{prop:Heat kernel on manifold without boundary}
    Let $M$ be a compact Riemannian manifold without boundary of dimension $m$. The heat kernel $E(t,x,y)$ 
    defined as the Schwartz kernel of $e^{-t\Delta_M}$ for $t>0$, is a smooth function for $(t,x,y)\in (0,\infty)\times M^2$.
    And for any $p_0\in M$ there exists a local chart $U\subset \mathbb{R}^m$ diffeomorphic 
    to a neighborhood of $p_0$ in $M$, 
    and a function $\tilde{E}$ 
    \[
        \tilde{E}\in C^\infty([0,\infty)\times \mathbb{R}^m\times U)
    \]
    so that in this chart we can write $E$ as 
    \[
        E(t,x,y)=t^{-\frac{m}{2}}\tilde{E}(\sqrt{t},\frac{x-y}{\sqrt{t}},y)
    \]
    for $x,y\in U,t>0$. Moreover $\tilde{E}$ has an asymptotic expansion near $t=0$ as 
    \[
        \tilde{E}(\sqrt{t},X,y)\sim \sum_{j=0}^{+\infty} \tilde{E}_{2j}(X,y) t^{j}
    \]
    for $\tilde{E}_{2j}(X,y)\in C^\infty(\mathbb{R}^m\times U)$, with the leading term $\tilde{E}_0$ as
    \[
        \tilde{E}_0(X,y)=\frac{1}{(4\pi)^{m/2}}e^{-\frac{|X|^2_{g(y)}}{4}}
    \]
    where $g(y)$ is the Riemannian metric at $y$, pull back from $M$ to $\mathbb{R}^n$. In addition we know the second order term satisfise
    \[
        \tilde{E}_2(0,y)=\frac{1}{6(4\pi)^{m/2}}\operatorname{Scal}(y)
    \]
    where $\operatorname{Scal}(y)$ is the scalar curvature at $y$.
\end{proposition}
When $M$ has boundaries and imposed with Dirichlet condition, there exsists a reflection term, 
corresponding to the heat kernel on the half space $\mathbb{R}^n_{+}$ is given by 
for $x=(x',x_n),y\in \mathbb{R}^n_{+}\times \mathbb{R}^n_{+}$
\[
    E_{\mathbb{R}^n_{+}}(t,x,y)=\frac{1}{(4\pi t)^{n/2}}\left(e^{-\frac{|x-y|^{2}}{4t}}-e^{-\frac{|x^*-y|^{2}}{4t}}\right),\quad x^*:=(x',-x_n)
\]
We have the following theorem.
\begin{proposition}\label{prop:Heat kernel on manifold with boundary}
    Let $M$ be a compact Riemannian manifold with boundary of dimension $m$. 
    Let $P_0$ be the Laplace operator, with Dirichlet condition. The heat kernel $E(t,x,y)$ 
    defined as the Schwartz kernel of $e^{-tP_0}$ for $t>0$, is a smooth function for $(t,x,y)\in (0,\infty)\times M^2$.
    And for any $p_0\in M$, we have 
    \begin{itemize}
        \item If $p_0\in \partial M$, then there exists a local chart $U\subset \mathbb{R}^m_{+}$ diffeomorphic 
        to a neighborhood of $p_0$ in $M$ of the form
        \[
            x=(x',x_n)\in U=U'\times [0,\epsilon),\quad U'\times \{0\}=U\cap \partial M, \quad U'\subset \mathbb{R}^{n-1}
        \]
        And there exists functions $\tilde{E}^{\operatorname{dir}},\tilde{E}^{\operatorname{refl}}$
        \[
            \begin{aligned}
                &\tilde{E}^{\operatorname{dir}}\in C^\infty([0,\infty)\times \mathbb{R}^n\times U)\\
                &\tilde{E}^{\operatorname{refl}}\in C^\infty([0,\infty)\times \mathbb{R}^{n-1}\times (\mathbb{R}_{\geq 0})^2\times U)
            \end{aligned}
        \]
        so that for $x,y\in U$ and $t>0$ one has 
        \[
            \begin{aligned}
                E(t,x,y)&=t^{-\frac{m}{2}}\left(\tilde{E}^{\operatorname{dir}}(\sqrt{t},\frac{x-y}{\sqrt{t}},y)-
                \tilde{E}^{\operatorname{refl}}(\sqrt{t},\frac{x'-y'}{\sqrt{t}},\frac{x_n}{\sqrt{t}},\frac{y_n}{\sqrt{t}},y)\right)\\
                &:=t^{-\frac{m}{2}}\tilde{E}(\sqrt{t},\frac{x'-y'}{\sqrt{t}},\frac{x_n}{\sqrt{t}},\frac{y_n}{\sqrt{t}},y)
            \end{aligned}
        \]
        Moreover, the leading term of $\tilde{E}$ is given by 
        \[
        \begin{aligned}
            \tilde{E}(\sqrt{t},X',\xi,\eta,y)&=
            \frac{1}{(4\pi)^{m/2}}\left(e^{-\frac{|(X,\xi-\eta)|^2_{g(y)}}{4}}-
            e^{-\frac{|(X',-\xi-\eta)|^2_{g(y)}}{4}}\right)\\
            &+t^{1/2}C^\infty([0,\infty)_{\sqrt{t}},\mathbb{R}^{n-1}\times (\mathbb{R}_{\geq 0})^2\times U)
        \end{aligned}
        \]

        \item If $p_0$ lies in the interior of $M$, then there exists a local chart $U\subset \mathbb{R}^m$ diffeomorphic 
        to a neighborhood of $p_0$ in $\operatorname{Int} M$, and a function $\tilde{E}$ 
        \[
            \tilde{E}\in C^\infty([0,\infty)\times \mathbb{R}^m\times U)
        \]
        so that in this chart we can write $E$ as 
        \[
            E(t,x,y)=t^{-\frac{m}{2}}\tilde{E}(\sqrt{t},\frac{x-y}{\sqrt{t}},y)
        \]
        for $x,y\in U,t>0$. Moreover the leading term of $\tilde{E}$ is given by $\tilde{E}$
        \[
        \begin{aligned}
            \tilde{E}(\sqrt{t},X,y)=
            \frac{1}{(4\pi)^{m/2}}e^{-\frac{|X|^2_{g(y)}}{4}}+t^{1/2}C^\infty([0,\infty)_{\sqrt{t}},\mathbb{R}^{n}\times U)
        \end{aligned}
        \]
    \end{itemize} 
\end{proposition}

Now using the heat kernel on compact manifold and on $\mathbb{R}^n$, 
we can directly compute the trace of $f(P_V)-f(P_0)$ for $f(x)=e^{-tx}$, using the method essentially 
the same as in \cite[Theorem 3.64]{MathematicalTheoryofScatteringResonances}.
We first present a lemma which is exactly the same as \cite[Lemma 3.63]{MathematicalTheoryofScatteringResonances}.
\begin{lemma}\label{lem:3.63}
    Suppose $V\in C_c^\infty(X;\mathbb{C})$. Then for any $M\in \mathbb{N}$ and $\operatorname{Im}(\lambda)>0$
    \begin{equation}
        R_V(\lambda)=\sum_{l=0}^L Y_lR_0(\lambda)^{l+1}+R_V(\lambda)Y_{L+1}R_0(\lambda)^{L+1}
    \end{equation}
    where for $l\geq 1$ the operators $Y_l$ is a differential operator of order $\leq l-1$ with compactly supported coefficients, 
    defined by induction as follows 
    \[
        Y_0:=I,\quad Y_{l+1}=-VY_l+[X_l,P_0]
    \]
\end{lemma}

\begin{proposition}
    Suppose that $V\in C_c^\infty(X;\mathbb{R})$, then 
    \[
        e^{-tP_V}-e^{-tP_0}\in \mathcal{L}_1(L^2(X)),\quad t>0
    \]
    And we have
    \begin{itemize}
        \item If $M$ has no boundaries, then 
        \[
            \operatorname{tr}(e^{-tP_V}-e^{-tP_0})=\frac{1}{(4\pi t)^{(n+m)/2}}(a_1(V)t+a_2(V)t^2)+\mathcal{O}(t^{5/2-(n+m)/2})
        \]
        where 
        \[
            a_1(V)=-\int_{\mathbb{R}^n\times M} V(x,y) dxdy,\quad a_2(V)=\int_{\mathbb{R}^n\times M} \frac{V(x,y)^2}{2}-\frac{\operatorname{Scal}(y)V(x,y)}{6} dxdy
        \]
        where $\operatorname{Scal}(y)$ is the scalar curvature of $M$ at $y\in M$.
        \item If $M$ has non-empty boundary and we impose the Dirichlet condition, then  
        \[
            \operatorname{tr}(e^{-tP_V}-e^{-tP_0})=\frac{1}{(4\pi t)^{(n+m)/2}}a_1(V)t+\mathcal{O}(t^{3/2-(n+m)/2})
        \]
        where 
        \[
            a_1(V)=-\int_{\mathbb{R}^n\times M} V(x,y) dxdy
        \]
    \end{itemize}
\end{proposition}

\begin{proof}
    Functional calculus of self-adjoint operators and 
    Cauchy integral formula shows that 
    \begin{equation}\label{eq:3.11.29}
        \begin{aligned}
            e^{-tP_V}-e^{-tP_0}=\frac{1}{2\pi i}\int_{\Gamma_c} e^{-tz}\left((P_V-z)^{-1}-(P_0-z)^{-1}\right)dz\\
            \Gamma_c:\mathbb{R}\ni s\mapsto z(s):=c+i|s|e^{i\operatorname{sgn}(s)\pi/4},\quad c<-||V||_{L^\infty}-1
        \end{aligned}
    \end{equation}

    \begin{figure}
        \centering
    \begin{tikzpicture}
        % [
        %     >={Kite[inset=0pt,length=0.32cm,bend]}, 
        %     decoration={markings, mark= at position .1 with {\arrow{>}}, 
        %     mark= at position .4 with {\arrow{>}}, 
        %     mark= at position 0.7 with {\arrow{>}}, }
        % ] 
        \def\radius{4} 
        % contour 
        % \filldraw[postaction = {decorate}, thick ,fill=gray!40] 
        % (0:\radius) arc (0:180:\radius) node[below]{$-R$}-- cycle; 
        
        % \node at (30:\radius+0.3){$C_{R}$}; 
        % axes 、
        \draw[very thick] (-0.5*\radius,0) -- (1*\radius,0);
        \filldraw[black] (-0.5*\radius,0) circle (2pt) node[below]{$-||V||_{L^\infty}$};
        \filldraw[black] (-1*\radius,0) circle (2pt) node[below]{$c$};
        \filldraw[black] (-1.5*\radius,0) circle (2pt) node[below]{$-1/t$};
        \draw[dashed] (-0.5*\radius,0.5*\radius) -- (0.25*\radius,1.25*\radius);
        \draw[dashed][-Latex]  (-1*\radius,0) -- (-0.5*\radius,0.5*\radius);
        \draw[dashed][-Latex] (0.25*\radius,-1.25*\radius) node[below]{$\Gamma_c$} -- (-0.5*\radius,-0.5*\radius);
        \draw[dashed] (-0.5*\radius,-0.5*\radius) -- (-1*\radius,0);
        \draw[dashed][-Latex] (-1.5*\radius,-1.25*\radius) -- (-1.5*\radius,-0.75*\radius);
        \draw[dashed] (-1.5*\radius,-0.75*\radius) -- (-1.5*\radius,0);
        \draw[dashed][-Latex] (-1.5*\radius,0) -- (-1.5*\radius,0.75*\radius);
        \draw[dashed] (-1.5*\radius,0.75*\radius) -- (-1.5*\radius,1.25*\radius);
        \draw[-Latex] (-1.75*\radius,0) -- (1*\radius,0); 
        \draw[-Latex] (0,-1.25*\radius) -- (0,1.25*\radius); 
    \end{tikzpicture}
    \caption{The contour integration to deal with $e^{-tP_V}-e^{-tP_0}$. The spectrum of $P_V$ and $P_0$ lies in 
    the bold line on the real line, right of $-||V||_{L^\infty}$} \label{fig:M1}
    \end{figure}
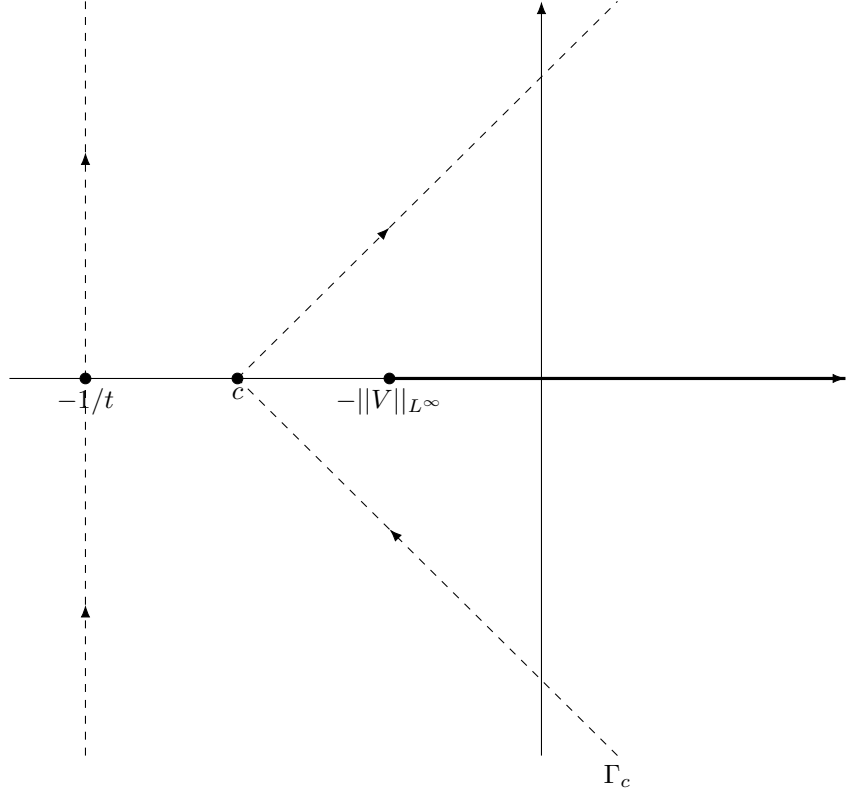

    And the Cauchy integral formula gives 
    \[
        \frac{1}{2\pi i}\int_{\Gamma_c} e^{-tP_0}(P_0-z)^{-m-1}dz=\frac{t^m}{m!}e^{-tP_0}
    \]
    so we can rewrite \eqref{eq:3.11.29} using Lemma \ref{lem:3.63} as 
    \[
        e^{-tP_V}-e^{-tP_0}=\sum_{l=1}^K \frac{t^l}{l!}Y_l e^{-tP_0}+ e_L(t)
    \]
    where the remainder term $e_M(t)$ is defined as 
        \begin{equation}\label{eq:3.11.31}
        e_L(t):=\frac{1}{2\pi i}\int_{\Gamma_c} e^{-tz}(P_V-z)^{-1}Y_{L+1}(P_0-z)^{-L-1}dz
        \end{equation}

    We first analyze the terms of the form $X_l e^{-tP_0}$. We know it's equals to the product of 
    the Euclidean heat kerenl and the heat kernel $E$ on $M$ given by proposition 
    \ref{prop:Heat kernel on manifold without boundary} 
    and proposition \ref{prop:Heat kernel on manifold with boundary}, the Schwartz kernel $K(t,x_1,x_2,y_1,y_2)$ 
    of $e^{-tP_0}$ is given by 
    \[
        K(t,x_1,x_2,y_1,y_2)=\frac{1}{(4\pi t)^{n/2}}e^{-|x_1-x_2|^2/4t}\otimes E(t,y_1,y_2)
    \]
    So since $Y_l$ is a differential operator of order $\leq l-1$ with compactly supported coefficients, we know $Y_l e^{-tP_0}\in \mathcal{L}_1$, 
    and the trace can be calculated directly as the integration along the diagonal 
    \[
    \begin{aligned}
        \frac{t^l}{l!}\operatorname{tr}(Y_le^{-tP_0})&=\frac{t^l}{l!(4\pi t)^{(n)/2}}\int_{\mathbb{R}^n\times M}
        \left(Y_l e^{-|x_1-x_2|^2/4t}\otimes E(t,y_1,y_2)\right)|_{x_1=x_2,y_1=y_2} dx_2 dy_2\\
        &=\frac{1}{(4\pi t)^{(n+m)/2}}t^{1+(l-1)/2}(a_{l,0}+a_{l,1}t^{1/2}+a_{l,2}t+a_{l,2}t^{3/2}+\mathcal{O}(t^2))
    \end{aligned}
    \]
    where we use the fact that each spatial derivative of $E(t,x,y)$ will gives a $t^{-1/2}$ 
    term, since in local chart
    \[
        E(t,x,y)=t^{-\frac{m}{2}}\tilde{E}(\sqrt{t},\frac{x'-y'}{\sqrt{t}},\frac{x_n}{\sqrt{t}},\frac{y_n}{\sqrt{t}},y)
    \]

    When $M$ has nonempty boundary, we simply use $Y_1=-V$ to write 
    \[
        \frac{t^l}{l!}\operatorname{tr}(Y_le^{-tP_0})=
        \left\{
            \begin{aligned}
                &\frac{1}{(4\pi t)^{(n+m)/2}}a_1(V)t+\mathcal{O}(t^{3/2-(n+m)/2}), \quad l=1 \\
                &\mathcal{O}(t^{1+(l-1)/2-(n+m)/2})\quad l\geq 2 
            \end{aligned}
        \right.
    \]
    When $M$ has empty boundary, we can write 
    \[
    \begin{aligned}
        Y_1=-V,\quad Y_2=V^2-\Delta V-2\nabla V\cdot \nabla\\
        Y_3(f)=-4\langle \operatorname{Hess} V,\operatorname{Hess} f\rangle+\tilde{Y}(f)
    \end{aligned}
    \]
    for $f\in C^\infty(X)$, where $\tilde{Y}$ is a differential operator of order one. Hence by direct calculation 
    for $l=1$
    \[
    \begin{aligned}
        t\operatorname{tr}(Y_1e^{-tP_0})&=\frac{1}{(4\pi t)^{(n+m)/2}}\left(a_1(V)t-
        \frac{t^2}{6}\int_{X} V(x,y)\operatorname{Scal}(y)dxdy\right)+\mathcal{O}(t^{3-(n+m)/2})\\
    \end{aligned}
    \]
    For $l=2,3$ all terms in the expansion of $\tilde{E}$ 
    but $\tilde{E}_0(X,y)$ are remainders of $\mathcal{O}(t^{5/2-(n+m)/2})$, 
    and we can use normal coordinate centered at $y_2$ so the calculation is the same 
    as the usual Euclidean heat kernel and Euclidean metric. We have
    \[
    \begin{aligned}
        \frac{t^2}{2!}\operatorname{tr}(Y_2e^{-tP_0})&=
        \frac{1}{2(4\pi t)^{(n+m)/2}}t^2\left(\int_{X}(V^2-\Delta V)\right)+\mathcal{O}(t^{5/2-(n+m)/2})\\
        &=\frac{t^2}{2(4\pi t)^{(n+m)/2}}\left(\int_{X}V^2\right)+\mathcal{O}(t^{5/2-(n+m)/2})\\
    \end{aligned}
    \]
    \[
    \begin{aligned}
        \frac{t^3}{3!}\operatorname{tr}(Y_3e^{-tP_0})&=
        \frac{t^3}{6(4\pi t)^{(n+m)/2}}\left(\frac{4}{2t}\int_{X}-\Delta V\right)+\mathcal{O}(t^{5/2-(n+m)/2})\\
        &=\mathcal{O}(t^{5/2-(n+m)/2})\\
    \end{aligned}
    \]

    It remains to deal with the remainder term $e_L$. Since we know for $k\in \mathbb{N}_{\geq 0}$
    \[ 
        ||u||_{H^{2k}}\sim ||(P_0+i)^{k/2} u||_{L^2}
    \]
    and uniformly for $\operatorname{Re} z\leq -1$
    \[
        ||(P_0-z)^{-1}||_{L^2\to L^2}\leq |z|^{-1},\quad ||(P_0-z)^{-1}||_{L^2\to H^2}\lessapprox 1 
    \]
    Thus we have for $r\in \mathbb{N}_{\geq 0}$
    \[
        ||(P_0-z)^{-1}||_{H^r\to H^{r}}\leq C_r |z|^{-1}, \quad ||(P_0-z)^{-1}||_{H^r\to H^{r+2}}\lessapprox 1 
    \]
    Let $N=\lceil\frac{n+m}{2}\rceil+1$
    So by using the $H_r\to H^r$ estimate $L/2+N$ times, 
    and then use the $H^{r}\to H^{r+2}$ esimate $L/2-N$ times, 
    we obtain for even $L$ with $L>2N$ 
    \[
        ||(P_0-z)^{-L}||_{L^2\to H^{L+2N}}\leq C_M |z|^{-L/2+N}
    \]
    uniformly for $\operatorname{Re} z\leq -1$.
    Since $Y_{L+1}$ is a differential operator with coefficients with bounded support, we know 
    \[
        ||Y_{L+1}(P_0-z)^{-L-1}||_{\mathcal{L}_1}\leq C_L
        ||(P_0-z)^{-L-1}||_{L^2\to H^{L+2N}}\leq C_L' |z|^{-L/2+N}
    \]
    Now we can return to \eqref{eq:3.11.31}. We can take the trace and deform the contour of integration to $s\mapsto -1/t+is,s\mathbb{R}$. 
    Using the estimate above we obtain 
    \[
    \begin{aligned}
        |\operatorname{tr}(e_L(t))|&\leq C \int_{-1/t-i\infty}^{-1/t+i\infty} e^{t\operatorname{Re}(z)}||Y_{L+1}(P_0-z)^{-L-1}||_{\mathcal{L}_1} |dz| \\
        &\leq C\int_{-\infty}^{\infty} (1/t+|s|)^{-L/2+N} |ds|=\mathcal{O}_L(t^{L/2-N})
    \end{aligned}
    \]

    Thus we know the remainder term can be the power of $t$ with arbitray order, this completes the proof.
\end{proof}

Now we consider the total variation $|d\mu|$ of the measure $d\mu(\lambda)$.
Since we know 
\[
    \int e^{-t\lambda} |d\mu|(\lambda)\geq 
    \int e^{-t\lambda} d\mu(\lambda)=
    \operatorname{tr}(e^{-tP_V}-e^{-tP_0})
\]
then by the usual Tauberian theory for positive measures, we have the following lower bound of the total variation of the scattering phase measure 
$d\mu$.
\begin{theorem}
    Let $V\in C_c^\infty(X;\mathbb{R})$. We have the following lower bound 
    for the cumulative function $\tilde{\mu}$ for the total variation $|d\mu|$ defined as
    \[
        \tilde{\mu}(\lambda)=\int_{-\infty}^\lambda |d\mu|(t) dt
    \]
    \begin{itemize}
        \item Suppose the mean value of $V$ is not zero, \textit{i.e.}
        \[
            \int_{X} V(x,y)dxdy\neq 0
        \]
        then we have 
        \[
            \liminf_{\lambda\to +\infty} \frac{\tilde{\mu}(\lambda)}{\lambda^{\frac{n+m}{2}-1}}>0
        \]
        \item Suppose the mean value of $V$ is now zero, and $V$ is not identically zero. Assume in addition that $M$ has no boundaries and has constant scalar curvature, then we know 
        \[
            \liminf_{\lambda\to +\infty} \frac{\tilde{\mu}(\lambda)}{\lambda^{\frac{n+m}{2}-2}}>0
        \]
    \end{itemize}
\end{theorem}

\noindent \textbf{Acknowledgements}. 
I would like to express my sincere gratitude to my advisor, Long Jin, for his invaluable guidance and constant encouragement throughout this project. 
I am also grateful to Yulin Gong and Xin Fu for the inspiring discussions. Furthermore, I wish to thank Yulai Huang and Junhao Shen for their help in clarifying the structure of the Riemann surface 
$\hat{\mathcal{Z}}$.

\bibliographystyle{alpha}
\bibliography{reference}

\end{document}